\documentclass[a4paper,twoside,10pt]{amsart}
\usepackage{amsmath,amsfonts,amssymb,amsthm}
\newtheorem{theorem}{Theorem}[section]
\newtheorem{lemma}{Lemma}[section]
\usepackage{mathrsfs}
\usepackage{color,graphicx}
\usepackage[a4paper]{geometry}
\numberwithin{equation}{section}

\author[G. Nemes]{Gerg\H{o} Nemes}
\address{Central European University, Department of Mathematics and its Applications, H-1051 Budapest, N\'ador utca 9, Hungary}
\email{nemesgery@gmail.com}

\begin{document}

\title[Resurgence of the incomplete gamma function]{The resurgence properties of\\ the incomplete gamma function I}

\begin{abstract} In this paper we derive new representations for the incomplete gamma function, exploiting the reformulation of the method of steepest descents by C. J. Howls (Howls, Proc. R. Soc. Lond. A \textbf{439} (1992) 373--396). Using these representations, we obtain a number of properties of the asymptotic expansions of the incomplete gamma function with large arguments, including explicit and realistic error bounds, asymptotics for the late coefficients, exponentially improved asymptotic expansions, and the smooth transition of the Stokes discontinuities.
\end{abstract}
\maketitle

\section{Introduction and main results}

It is known that, as $a \to \infty$ in the sector $\left|\arg a\right|\leq \frac{3\pi}{2}-\delta< \frac{3\pi}{2}$, for any $0 < \delta \leq \frac{3\pi}{2}$, the incomplete gamma function has the following asymptotic expansion
\begin{equation}\label{eq43}
\Gamma \left( {a,z} \right) \sim z^a e^{ - z} \sum\limits_{n = 0}^\infty  {\frac{{\left( { - a} \right)^n b_n \left( \lambda  \right)}}{{\left( {z - a} \right)^{2n + 1} }}} ,
\end{equation}
with $z = \lambda a$, $\lambda >1$ (see, e.g., \cite[8.11.iii]{NIST} and \cite{Tricomi}). For this series to be an asymptotic series in Poincar\'e's sense, it has to be assumed that $\left( {\lambda  - 1} \right)^{ - 1}  = o\left( {\left|a\right|^{\frac{1}{2}} } \right)$. The coefficients $b_n \left( \lambda  \right)$ are polynomials in $\lambda$ of degree $n$, the first few being
\[
b_0\left( \lambda  \right)=1,\; b_1\left( \lambda  \right)=\lambda, \; b_2\left( \lambda  \right)=2\lambda^2+\lambda,\; b_3 \left( \lambda  \right) = 6\lambda ^3  + 8\lambda ^2+\lambda.
\]
When $\lambda =1$, we have the alternative expansion
\begin{equation}\label{eq44}
\Gamma \left( {z,z} \right) \sim \sqrt {\frac{\pi }{2}} z^{z - \frac{1}{2}} e^{ - z} \sum\limits_{n = 0}^\infty {\frac{{a_n }}{{z^{\frac{n}{2}} }}},
\end{equation}
as $z \to \infty$ in the sector $\left|\arg z\right|\leq \pi-\delta< \pi$, for any $0 < \delta \leq \pi$ \cite[8.11.E12]{NIST}. Note that this asymptotic series is not the limiting case $\lambda \to 1+0$ of the expansion \eqref{eq43}. The first few coefficients $a_n$ are as follows
\[
a_0  = 1,\; a_1  =  - \frac{1}{3}\sqrt {\frac{2}{\pi }} ,\; a_2  = \frac{1}{{12}},\; a_3  =- \frac{4}{{135}}\sqrt {\frac{2}{\pi }},\; a_4=\frac{1}{288}.
\]
It will turn out later in this paper that the asymptotic series \eqref{eq44} is actually valid in the much wider region $\left|\arg z\right|\leq 2\pi-\delta< 2\pi$, for any $0 < \delta \leq 2\pi$ (see the end of Subsection \ref{subsection53}).

Asymptotic series for the coefficients $b_n \left( \lambda  \right)$ and $a_n$ were established by Dingle \cite[p. 161 and p. 164]{Dingle}. He also gave re-expansions of the remainder terms leading to exponentially improved versions of the asymptotic series \eqref{eq43} and \eqref{eq44} \cite[pp. 463--464]{Dingle}. Nevertheless, the derivation of his results is based on interpretative, rather than rigorous, methods.

The aim of this paper is to establish new resurgence-type integral representations for the functions $\Gamma \left( {a,z} \right)$ and $\Gamma \left( {z,z} \right)$. Our derivation is based on the reformulation of the method of steepest descents by Howls \cite{Howls}, and on the following integral representation (see, e.g., \cite[8.6.E7]{NIST})
\begin{equation}\label{eq0}
\Gamma \left( {a,z} \right) = z^a \int_0^{ + \infty } {\exp \left( {at - ze^t } \right)dt} \; \text{ for } \; \Re \left( z \right) > 0 .
\end{equation}
The resurgence has to be understood in the sense of Berry and Howls \cite{Berry}, meaning that the function (or a closely related function) reappears in the remainder of its own asymptotic series. Using these representations, we obtain several new properties of the expansions \eqref{eq43} and \eqref{eq44}, including explicit and numerically computable error bounds, asymptotics for the late coefficients, exponentially improved asymptotic expansions, and the smooth transition of the Stokes discontinuities. Our analysis also provides a rigorous treatment of Dingle's formal results.

In our resurgence-type formulas, the function that makes its appearance in the remainder terms, is the scaled gamma function. This function is defined in terms of the classical gamma function as
\[
\Gamma^\ast \left( z \right) = \frac{{\Gamma \left( z \right)}}{{\sqrt {2\pi } z^{z - \frac{1}{2}} e^{ - z} }},
\]
for $\left|\arg z\right|<\pi$.

Our first theorem describes the resurgence properties of the asymptotic expansion \eqref{eq43}. Throughout this paper, empty sums are taken to be zero.

\begin{theorem}\label{thm1} Let $\lambda >1$ be a fixed real number such that $z=\lambda a$, and let $N$ be a non-negative integer. Then
\begin{equation}\label{eq22}
\Gamma \left( {a,z} \right) =  z^a e^{ - z} \left( {\sum\limits_{n = 0}^{N - 1} {\frac{{\left( { - a} \right)^n b_n \left( \lambda  \right)}}{{\left( {z - a} \right)^{2n + 1} }}}  + R_N \left( {a,\lambda } \right)} \right)
\end{equation}
for $\left|\arg a\right|<\pi$, with
\begin{gather}\label{eq10}
\begin{split}
b_n \left( \lambda  \right) & = \left( { - 1} \right)^n \left( {\lambda -1} \right)^n \left[ {\frac{{d^n }}{{dt^n }}\left( {\frac{{\left( {\lambda-1 } \right)t}}{{\lambda e^{ t} - t - \lambda }}} \right)^{n + 1} } \right]_{t = 0} \\ & = \frac{{\left( {\lambda  - 1} \right)^{2n + 1} }}{{\sqrt {2\pi } }}\int_0^{ + \infty } {\frac{{t^{n - \frac{1}{2}} e^{ - t\left( {\lambda  - \log \lambda  - 1} \right)} }}{{\Gamma^\ast \left( t \right)}}dt} 
\end{split}
\end{gather}
and
\begin{equation}\label{eq9}
R_N \left( {a,\lambda } \right) = \frac{{\left( { - a} \right)^N }}{{\left( {z - a} \right)^{2N + 1} }}\frac{{\left( {\lambda  - 1} \right)^{2N + 1} }}{{\sqrt {2\pi } }}\int_0^{ + \infty } {\frac{{t^{N - \frac{1}{2}} e^{ - t\left( {\lambda  - \log \lambda  - 1} \right)} }}{{1 + t/a}}\frac{{dt}}{{\Gamma^\ast \left( t \right)}}} .
\end{equation}
\end{theorem}

In the important paper \cite{Boyd}, Boyd gave the following resurgence formulas for the scaled gamma function and its reciprocal (see also \cite{Nemes}):
\begin{equation}\label{eq45}
\Gamma^\ast \left( z \right) = \sum\limits_{n = 0}^{N - 1} {\frac{{a_{2n} }}{{z^n }}}  + M_N \left( z \right) \; \text{ and } \;
\frac{1}{\Gamma^\ast \left( z \right)} = \sum\limits_{n = 0}^{N - 1} {\left( { - 1} \right)^n \frac{{a_{2n} }}{{z^n }}}  + \widetilde M_N \left( z \right),
\end{equation}
for $\left|\arg z\right| < \frac{\pi}{2}$ and $N\geq 1$, with
\begin{equation}\label{eq46}
M_N \left( z \right) = \frac{1}{{2\pi i}}\frac{i^N}{{z^N }}\int_0^{ + \infty } {\frac{{t^{N - 1} e^{ - 2\pi t} \Gamma^\ast \left( {it} \right)}}{{1 - it/z}}dt}  - \frac{1}{{2\pi i}}\frac{{\left( { - i} \right)^N }}{{z^N }}\int_0^{ + \infty } {\frac{{t^{N - 1} e^{ - 2\pi t} \Gamma^\ast \left( { - it} \right)}}{{1 + it/z}}dt} 
\end{equation}
and
\begin{equation}\label{eq47}
\widetilde M_N \left( z \right) = \frac{1}{{2\pi i}}\frac{{\left( { - i} \right)^N}}{{z^N }}\int_0^{ + \infty } {\frac{{t^{N - 1} e^{ - 2\pi t} \Gamma^\ast \left( {it} \right)}}{{1 + it/z}}dt}  - \frac{1}{{2\pi i}}\frac{{i^N }}{{z^N }}\int_0^{ + \infty } {\frac{{t^{N - 1} e^{ - 2\pi t} \Gamma^\ast \left( { - it} \right)}}{{1 - it/z}}dt} .
\end{equation}
Here the coefficients $a_{2n}$ are the same as those appearing in \eqref{eq44}. These representations of the scaled gamma function and its reciprocal will play an important role in later sections of this paper. We will also use the fact that $M_N \left( z \right),\widetilde M_N \left( z \right) =\mathcal{O}\left(\left|z\right|^{-N}\right)$ as $z \to \infty$ in the sector $\left|\arg z\right|\leq \pi-\delta< \pi$, for any fixed $0 < \delta \leq \pi$.

Applying the connection formulas
\begin{equation}\label{eq66}
\Gamma \left( {ae^{2\pi im} ,ze^{2\pi im} } \right) = \Gamma \left( {a,ze^{2\pi im} } \right) = e^{2\pi ima} \Gamma \left( {a,z} \right) + \left( {1 - e^{2\pi ima} } \right)\Gamma \left( a \right),
\end{equation}
\begin{equation}\label{eq51}
\Gamma^\ast \left( z \right) = \frac{1}{1 - e^{ \pm 2\pi iz}}\frac{1}{\Gamma^\ast \left( {ze^{ \mp \pi i} } \right)}
\; \text{ and } \; \Gamma^\ast  \left( z \right) =  - e^{ \pm 2\pi iz} \Gamma^\ast  \left( {ze^{ \pm 2\pi i} } \right) ,
\end{equation}
and the resurgence formulas \eqref{eq22}, \eqref{eq9}, \eqref{eq45}--\eqref{eq47}, we can derive analogous representations for $\Gamma \left( {a,z} \right)$ in sectors of the form
\[
\left( {2m - 1} \right)\pi  < \arg a < \left( {2m + 1} \right)\pi ,\quad m \in \mathbb{Z}.
\]
The lines $\arg a = \left( {2m + 1} \right)\pi$ ($m \in \mathbb{Z}$) are the Stokes lines for the asymptotic series \eqref{eq43}.

The second theorem provides a resurgence formula for $\Gamma \left( {z,z} \right)$.

\begin{theorem}\label{thm2} For any integer $N \geq 2$, we have
\[
\Gamma \left( {z,z} \right) = \sqrt {\frac{\pi }{2}} z^{z - \frac{1}{2}} e^{ - z} \left( {\sum\limits_{n = 0}^{N - 1} {\frac{{a_n }}{{z^{\frac{n}{2}} }}}  + R_N \left( z \right)} \right)
\]
for $\left|\arg z\right|<\frac{3\pi}{2}$, with
\begin{gather}\label{eq18}
\begin{split}
a_n  & = \frac{1}{{2^{\frac{n}{2}} \Gamma \left( {\frac{n}{2} + 1} \right)}}\left[ {\frac{{d^n }}{{dt^n }}\left( {\frac{1}{2}\frac{{t^2 }}{{e^t  - t - 1}}} \right)^{\frac{{n + 1}}{2}} } \right]_{t = 0}\\ & = \frac{{e^{ - \frac{3}{4}n\pi i} }}{{2\pi i}}\int_0^{ + \infty } {t^{\frac{n}{2} - 1} e^{ - 2\pi t} \Gamma^\ast \left( {it} \right)dt}  - \frac{{e^{\frac{3}{4}n\pi i} }}{{2\pi i}}\int_0^{ + \infty } {t^{\frac{n}{2} - 1} e^{ - 2\pi t} \Gamma^\ast \left( { - it} \right)dt} ,
\end{split}
\end{gather}
where the second representation is true only for $n\geq 2$. The remainder term $R_N \left( z \right)$ is given by
\begin{equation}\label{eq17}
R_N \left( z \right) = \frac{{e^{ - \frac{3}{4}N\pi i} }}{{2\pi iz^{\frac{N}{2}} }}\int_0^{ + \infty } {\frac{{t^{\frac{N}{2} - 1} e^{ - 2\pi t} }}{{1 - e^{ - \frac{{3\pi i}}{4}} \left( {t/z} \right)^{\frac{1}{2}} }}\Gamma^\ast\left( {it} \right)dt}  - \frac{{e^{\frac{3}{4}N\pi i} }}{{2\pi iz^{\frac{N}{2}} }}\int_0^{ + \infty } {\frac{{t^{\frac{N}{2} - 1} e^{ - 2\pi t} }}{{1 - e^{\frac{{3\pi i}}{4}} \left( {t/z} \right)^{\frac{1}{2}} }}\Gamma^\ast\left( { - it} \right)dt} .
\end{equation}
The square roots are defined to be positive on the positive real line and are defined by analytic continuation elsewhere.
\end{theorem}

Again, this resurgence formula can be extended to other sectors of the complex plane like the one in Theorem \ref{thm1}. (One has to replace $a$ by $z$ in the
continuation formulas given above.) In this case, the Stokes lines are given by $\arg z = \left( {2m + \frac{3}{2}} \right)\pi$ for any integer $m$.

In Section \ref{section3}, we will show how to obtain numerically computable bounds for the remainder terms $R_N \left( {a,\lambda } \right)$ and $R_N \left( z \right)$ using their explicit forms given in Theorems \ref{thm1} and \ref{thm2}. To our knowledge no simple, explicit error bounds for the asymptotic series \eqref{eq43} and \eqref{eq44} exist in the literature. Some other formulas for the coefficients $b_n \left( \lambda  \right)$ and $a_n$ can be found in Appendix \ref{appendixa}.

In the following theorems, we give exponentially improved versions of the asymptotic expansions \eqref{eq43} and \eqref{eq44}. These expansions can be viewed as the mathematically rigorous forms of the terminated expansions of Dingle \cite[pp. 463--464]{Dingle}. In these theorems, we truncate the asymptotic series \eqref{eq43} and \eqref{eq44} at about their least terms and re-expand the remainders into new asymptotic expansions. The resulting exponentially improved asymptotic series are valid in larger regions than the original expansions. The terms in these new series involve the terminant function $\widehat T_p\left(w\right)$, which allows the smooth transition through the Stokes lines. For the definition and basic properties of the terminant function, see Section \ref{section5}. Throughout this paper, we use subscripts in the $\mathcal{O}$ notations to indicate the dependence of the implied constant on certain parameters. In Theorem \ref{thm3}, $R_N \left( {a,\lambda } \right)$ is extended to a wider region using analytic continuation.

\begin{theorem}\label{thm3} Let $K$ be an arbitrary fixed non-negative integer, and let $\lambda>1$ be a fixed real number. Suppose that $\left|\arg a\right| \leq 2\pi -\delta <2\pi$ with an arbitrary fixed small positive $\delta$, $\left|a\right|$ is large and $N = \left| a \right| \left( {\lambda  - \log \lambda  - 1} \right)+ \rho$ with $\rho$ being bounded. Then
\begin{equation}\label{eq85}
R_N \left( {a,\lambda } \right) = e^{a\left( {\lambda  - \log \lambda  - 1} \right)} \sqrt {\frac{{2\pi }}{a}} \sum\limits_{k = 0}^{K - 1} {\frac{{a_{2k} }}{{a^k }}\widehat T_{N - k + \frac{1}{2}} \left( {a\left( {\lambda  - \log \lambda  - 1} \right)} \right)}  + R_{N,K} \left( {a,\lambda } \right),
\end{equation}
where
\[
R_{N,K} \left( {a,\lambda } \right) = \mathcal{O}_{K,\rho } \left( {\frac{{e^{ - \left| a \right|\left( {\lambda  - \log \lambda  - 1} \right)} }}{{\left| a \right|^{K + \frac{1}{2}} }}} \right)
\]
for $\left|\arg a\right|\leq \pi$;
\[
R_{N,K} \left( {a,\lambda } \right) = \mathcal{O}_{K,\rho,\delta} \left( {\frac{{e^{\Re \left( a \right)\left( {\lambda  - \log \lambda  - 1} \right)} }}{{\left| a \right|^{K + \frac{1}{2}} }}} \right)
\]
for $\pi \le \left|\arg a\right| < 2\pi  - \delta$.
\end{theorem}

\begin{theorem}\label{thm4} Define $R_{N,M} \left( z \right)$ by
\begin{equation}\label{eq86}
\Gamma \left( {z,z} \right) = \sqrt {\frac{\pi }{2}} z^{z - \frac{1}{2}} e^{ - z} \left( {\sum\limits_{n = 0}^{N - 1} {\frac{{a_{2n} }}{{z^n }}}  + \sum\limits_{m = 0}^{M - 1} {\frac{{a_{2m + 1} }}{{z^{m + \frac{1}{2}} }}}  + R_{N,M} \left( z \right)} \right).
\end{equation}
Suppose that $\left|\arg z\right| \leq 3\pi -\delta <3\pi$ with an arbitrary fixed small positive $\delta$, $\left|z\right|$ is large and $N=2\pi\left|z\right|+\rho$, $M=2\pi\left|z\right|+\sigma$ with $\rho$ and $\sigma$ being bounded. Then
\begin{gather}\label{eq95}
\begin{split}
R_{N,M} \left( z \right) =\; & e^{2\pi iz} \sum\limits_{k = 0}^{K - 1} {\frac{{a_{2k} }}{{z^k }}\widehat T_{N - k} \left( {2\pi iz} \right)}  - e^{ - 2\pi iz} \sum\limits_{k = 0}^{K - 1} {\frac{{a_{2k} }}{{z^k }}\widehat T_{N - k} \left( { - 2\pi iz} \right)} 
\\ & - e^{2\pi iz} \sum\limits_{\ell  = 0}^{L - 1} {\frac{{a_{2\ell } }}{{z^\ell  }}\widehat T_{M - \ell  + \frac{1}{2}} \left( {2\pi iz} \right)}  -e^{ - 2\pi iz} \sum\limits_{\ell  = 0}^{L - 1} {\frac{{a_{2\ell } }}{{z^\ell  }} \widehat T_{M - \ell  + \frac{1}{2}} \left( { - 2\pi iz} \right)} 
\\ & + R_{N,M,K,L} \left( z \right),
\end{split}
\end{gather}
where $K$ and $L$ are arbitrary fixed non-negative integers, and
\begin{equation}\label{eq83}
R_{N,M,K,L} \left( z \right) = \mathcal{O}_{K,\rho } \left( {\frac{{e^{ - 2\pi \left| z \right|} }}{{\left| z \right|^K }}} \right) + \mathcal{O}_{L,\sigma } \left( {\frac{{e^{ - 2\pi \left| z \right|} }}{{\left| z \right|^L }}} \right)
\end{equation}
for $\left|\arg z\right|\leq \frac{\pi}{2}$;
\begin{equation}\label{eq84}
R_{N,M,K,L} \left( z \right) = \mathcal{O}_{K,\rho } \left( {\frac{{e^{ \mp 2\pi \Im \left( z \right)} }}{{\left|z\right|^K }}} \right) + \mathcal{O}_{L,\sigma } \left( {\frac{{e^{ \mp 2\pi \Im \left( z \right)} }}{{\left|z\right|^L }}} \right)
\end{equation}
for $\frac{\pi}{2}\leq \pm \arg z \leq \frac{3\pi}{2}$;
\[
R_{N,M,K,L} \left( z \right) = \mathcal{O}_{K,\rho,\delta} \left( {\frac{{\left| {\sin \left( {2\pi z} \right)} \right|}}{{\left| z \right|^K }}} \right) + \mathcal{O}_{L,\sigma,\delta} \left( {\frac{{\left| {\cos \left( {2\pi z} \right)} \right|}}{{\left| z \right|^L }}} \right)
\]
for $\frac{3\pi}{2}\leq \left|\arg z\right| \leq 3\pi-\delta$. Moreover, if $K=L$, then the bound \eqref{eq83} remains valid in the larger sector $\left|\arg z\right| \leq \frac{3\pi}{2}$, and the estimate \eqref{eq84} holds in the sectors $\frac{3\pi}{2}\leq \mp \arg z \leq 3\pi-\delta$ with an implied constant that also depends on $\delta$.
\end{theorem}

While proving Theorems \ref{thm3} and \ref{thm4} in Section \ref{section5}, we also obtain the following explicit bounds for the remainders in \eqref{eq85} and \eqref{eq86}.

\begin{theorem}\label{thm5} For any integers $2 \leq K \leq N$, define the remainder $R_{N,K} \left( {a,\lambda } \right)$ by \eqref{eq85}. Then we have
\begin{align*}
\left| {R_{N,K} \left( {a,\lambda } \right)} \right| \le \; & \left| {\sqrt {\frac{{2\pi }}{a}} e^{a\left( {\lambda  - \log \lambda  - 1} \right)} \widehat T_{N - K + \frac{1}{2}} \left( {a\left( {\lambda  - \log \lambda  - 1} \right)} \right)} \right|\frac{{\left( {1 + \zeta \left( K \right)} \right)\Gamma \left( K \right)}}{{\left( {2\pi } \right)^{K + 1} \left| a \right|^K }}
\\ & + \frac{{\left( {1 + \zeta \left( K \right)} \right)\Gamma \left( K \right)\Gamma \left( {N - K + \frac{1}{2}} \right)}}{{\left( {2\pi } \right)^{K + \frac{3}{2}} \left| a \right|^{N + 1} \left( {\lambda  - \log \lambda  - 1} \right)^{N - K + \frac{1}{2}} }},
\end{align*}
provided that $\left|\arg a\right|\leq \pi$.
\end{theorem}

\begin{theorem}\label{thm6} For any integers $2 \leq K < N$ and $2 \leq L \leq M$, define the remainder $R_{N,M,K,L} \left( z \right)$ by \eqref{eq86}. Then we have
\begin{align*}
\left| {R_{N,M,K,L} \left( z \right)} \right| \le \; & \left( {2\sqrt K  + 1} \right)\left( {\left| {e^{2\pi iz} \widehat T_{N-K}\left( {2\pi iz} \right)} \right| + \left| {e^{ - 2\pi iz} \widehat T_{N-K}\left( { - 2\pi iz} \right)} \right|} \right)\frac{{\zeta \left( K \right)\Gamma \left( K \right)}}{{\left( {2\pi } \right)^{K + 1} \left| z \right|^K }}
\\ & + \left( {6K + 2} \right)\frac{{\Gamma \left( {N - K} \right)\zeta \left( K \right)\Gamma \left( K \right)}}{{\left( {2\pi } \right)^{N + 2} \left| z \right|^N }}
\\ & + \left( {2\sqrt L  + 1} \right)\left( {\left| {e^{2\pi iz} \widehat T_{M - L + \frac{1}{2}} \left( {2\pi iz} \right)} \right| + \left| {e^{ - 2\pi iz} \widehat T_{M - L + \frac{1}{2}} \left( { - 2\pi iz} \right)} \right|} \right)\frac{{\zeta \left( L \right)\Gamma \left( L \right)}}{{\left( {2\pi } \right)^{L + 1} \left| z \right|^L }}
\\ & + \left( {6L + 2} \right)\frac{{\Gamma \left( {M - L + \frac{1}{2}} \right)\zeta \left( L \right)\Gamma \left( L \right)}}{{\left( {2\pi } \right)^{M + \frac{5}{2}} \left| z \right|^{M + \frac{1}{2}} }},
\end{align*}
as long as $\left|\arg z\right| \leq \frac{\pi}{2}$.
\end{theorem}

The rest of the paper is organised as follows. In Section \ref{section2}, we prove the resurgence formulas stated in Theorems \ref{thm1} and \ref{thm2}. In Section \ref{section3}, we give explicit and numerically computable error bounds for the asymptotic series \eqref{eq43} and \eqref{eq44} using the results of Section \ref{section2}. In Section \ref{section4}, asymptotic approximations for the coefficients $b_n \left( \lambda  \right)$ and $a_n$ are given. In Section \ref{section5}, we prove the exponentially improved expansions presented in Theorems \ref{thm3} and \ref{thm4} and the error bounds given in Theorems \ref{thm5} and \ref{thm6}, and provide a detailed discussion of the Stokes phenomenon related to the expansions \eqref{eq43} and \eqref{eq44}.

\section{Proofs of the resurgence formulas}\label{section2} If $z = \lambda a$, with $\lambda\geq 1$ fixed, then \eqref{eq0} becomes
\begin{equation}\label{eq1}
\Gamma \left( {a,z} \right) = z^a e^{ - z} \int_0^{ + \infty } {\exp \left( { - a\left( {\lambda e^t  - t - \lambda } \right)} \right)dt},
\end{equation}
provided that $\Re \left( a \right) > 0$. The analysis is significantly different according to whether $\lambda>1$ or $\lambda=1$. The saddle points of the integrand are the roots of the equation $\lambda e^t  - 1 =0$. Hence, the saddle points are given by $t^{\left(k\right)}=-\log \lambda +2\pi i k$ where $k$ is an arbitrary integer. We denote by $\mathscr{C}^{\left(k\right)}\left(\theta\right)$ the portion of the steepest paths that pass through the saddle point $t^{\left(k\right)}$. Here, and subsequently, we write $\theta = \arg a$. As for the path of integration $\mathscr{P}\left(\theta\right)$ in \eqref{eq1}, we take that connected component of
\[
\left\{ {t \in \mathbb{C}:\arg \left[ {e^{i\theta } \left( {\lambda e^t  - t - \lambda} \right)} \right] = 0} \right\} ,
\]
which is the positive real axis for $\theta=0$ and is the continuous deformation of the positive real axis as $\theta$ varies. If $\lambda=1$, the path $\mathscr{P}\left(\theta\right)$ is part of the contour $\mathscr{C}^{\left(0\right)}\left(\theta\right)$.

\subsection{Case (i): $\lambda>1$} Let $f\left( {t,\lambda } \right) = \lambda e^t - t - \lambda$ for some fixed $\lambda >1$. Hence, we can write \eqref{eq1} as
\[
\Gamma \left( {a,z} \right) = z^a e^{ - z} \int_0^{ + \infty } {e^{ - af\left( {t,\lambda } \right)} dt} 
\]
for any $\Re \left( a \right) > 0$. For simplicity, we assume that $a>0$. In due course, we shall appeal to an analytic continuation argument to extend our results to complex $a$. If
\begin{equation}\label{eq2}
\tau = f\left( {t,\lambda } \right),
\end{equation}
then $\tau$ is real on the curve $\mathscr{P}\left(0\right)$, and, as $t$ travels along this curve from $0$ to $+\infty$, $\tau$ increases from $0$ to $+\infty$. Therefore, corresponding to each positive value of $\tau$, there is a value of $t$, say $t\left(\tau\right)$, satisfying \eqref{eq2} with $t\left(\tau\right)>0$. In terms of $\tau$, we have
\[
\Gamma \left( {a,z} \right) = z^a e^{ - z} \int_0^{ + \infty } {e^{ - a\tau}\frac{dt}{d\tau}d\tau} = z^a e^{ - z} \int_0^{ + \infty } {\frac{{e^{ - a\tau} }}{{f'\left( {t\left( \tau \right),\lambda } \right)}}d\tau} .
\]
Following Howls, we express the function involving $t\left(\tau\right)$ as a contour integral using the residue theorem, to find
\[
\Gamma \left( {a,z} \right) = z^a e^{ - z} \int_0^{ + \infty } {e^{ - a\tau } \frac{1}{{2\pi i}}\oint_\Gamma  {\frac{f^{-1}\left( {u,\lambda } \right)}{{1 - \tau f^{-1}\left( {u,\lambda } \right)}}du} d\tau } ,
\]
where the contour $\Gamma$ encircles the path $\mathscr{P}\left(0\right)$ in the positive direction and does not enclose any of the saddle points $t^{\left(k\right)}$ (see Figure \ref{fig1}). Now, we employ the well-known expression for non-negative integer $N$
\begin{equation}\label{eq12}
\frac{1}{1 - x} = \sum\limits_{n = 0}^{N-1} {x^n}  + \frac{x^N}{1 - x},\; x \neq 1,
\end{equation}
to expand the function under the contour integral in powers of $\tau f^{ - 1} \left( {u,\lambda } \right)$. The result is
\[
\Gamma \left( {a,z} \right) =  z^a e^{ - z} \left( {\sum\limits_{n = 0}^{N - 1} {\int_0^{ + \infty } {\tau ^n e^{ - a\tau }  \frac{1}{{2\pi i}}\oint_\Gamma  {\frac{{du}}{{f^{n + 1} \left( {u,\lambda } \right)}}} d\tau } }  + R_N \left( {a,\lambda } \right)} \right)
\]
where
\begin{equation}\label{eq3}
R_N \left( {a,\lambda } \right) =  \int_0^{ + \infty } {\tau ^N e^{ - a\tau } \frac{1}{{2\pi i}}\oint_\Gamma  {\frac{f^{-N-1} \left( {u,\lambda } \right)}{{1 - \tau f^{-1}\left( {u,\lambda } \right) }}du} d\tau } .
\end{equation}
The path $\Gamma$ in the sum can be shrunk into a small circle around $0$, and we arrive at
\begin{equation}\label{eq4}
\Gamma \left( {a,z} \right) =  z^a e^{ - z} \left( {\sum\limits_{n = 0}^{N - 1} {\frac{{\left( { - a} \right)^n b_n \left( \lambda  \right)}}{{\left( {z - a} \right)^{2n + 1} }}}  + R_N \left( {a,\lambda } \right)} \right),
\end{equation}
where
\begin{align*}
b_n \left( \lambda  \right) & = \left( { - 1} \right)^n \frac{{\left( {\lambda -1} \right)^{2n + 1} \Gamma \left( {n + 1} \right)}}{{2\pi i}}\oint_{\left( {0^ +  } \right)} {\frac{{du}}{{f^{n + 1} \left( {u,\lambda } \right)}}} \\ & = \left( { - 1} \right)^n \left( {\lambda -1} \right)^n \left[ {\frac{{d^n }}{{dt^n }}\left( {\frac{{\left( {\lambda-1 } \right)t}}{{\lambda e^{ t} - t - \lambda }}} \right)^{n + 1} } \right]_{t = 0} .
\end{align*}

\begin{figure}[!t]
\def\svgwidth{0.6\columnwidth}
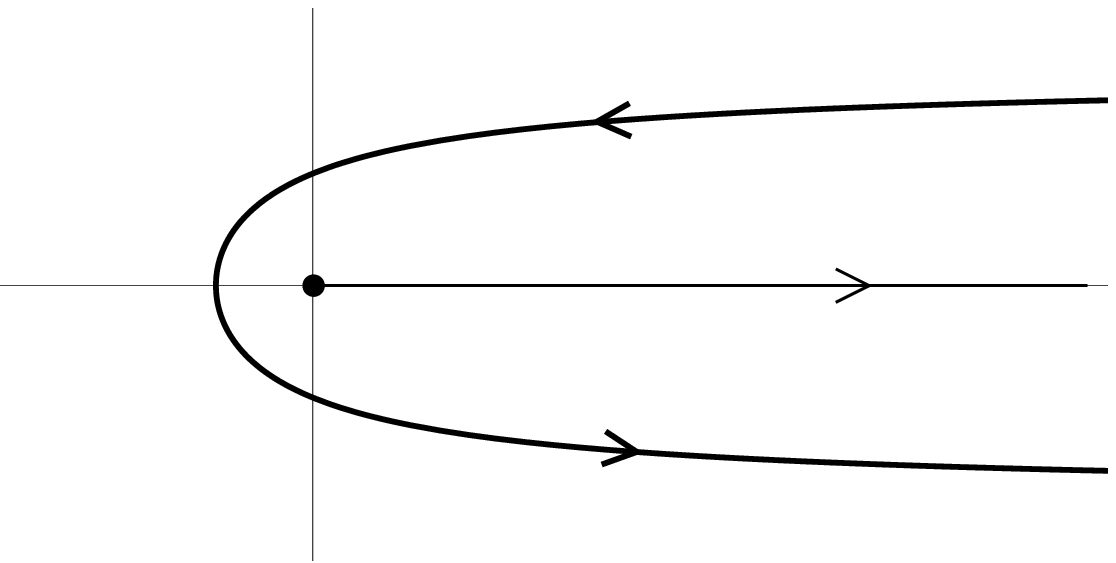
\caption{The contour $\Gamma$ encircling the path $\mathscr{P}\left(0\right)$.}
\label{fig1}
\end{figure}

Performing the change of variable $a\tau  = s$ in \eqref{eq3} yields
\begin{equation}\label{eq5}
R_N \left( {a,\lambda } \right) = \frac{{\left( { - a} \right)^N }}{{\left( {z - a} \right)^{2N + 1} }}\left( { - 1} \right)^N \left( {\lambda  - 1} \right)^{2N + 1} \int_0^{ + \infty } {s^N e^{ - s} \frac{1}{{2\pi i}}\oint_\Gamma  {\frac{{f^{ - N - 1} \left( {u,\lambda } \right)}}{{1 - \left( {s/a} \right)f^{ - 1} \left( {u,\lambda } \right)}}du} ds} .
\end{equation}
This representation of $R_N \left( {a,\lambda } \right)$ and the formula \eqref{eq4} can be continued analytically if we choose $\Gamma = \Gamma\left(\theta\right)$ to be an infinite contour that surrounds the path $\mathscr{P}\left(\theta\right)$ in the anti-clockwise direction and that does not encircle any of the saddle points $t^{\left(k\right)}$. This continuation argument works until the path $\mathscr{P}\left(\theta\right)$ runs into a saddle point. In the terminology of Howls, such saddle points are called adjacent to the endpoint $0$. In our case, when $\theta =\pm\pi$, the path $\mathscr{P}\left(\theta\right)$ connects to the saddle point $t^{\left(0\right)}=-\log \lambda$. This is the adjacent saddle. The set
\[
\Delta = \left\{u\in \mathscr{P}\left(\theta\right): -\pi <\theta<\pi \right\}
\]
forms a domain in the complex plane whose boundary is a steepest descent path through the adjacent saddle (see Figure \ref{fig2}). This path is $\mathscr{C}^{\left(0\right)}\left(\pi\right)$, and it is called the adjacent contour to the endpoint $0$. For $N\geq 0$ and fixed $a$, the function under the contour integral in \eqref{eq5} is an analytic function of $u$ in the domain $\Delta$ excluding $\mathscr{P}\left(\theta\right)$, therefore we can deform $\Gamma$ over the adjacent contour. We thus find that for $-\pi <\theta<\pi$ and $N\geq0$, \eqref{eq5} may be written
\begin{equation}\label{eq6}
R_N \left( {a,\lambda } \right) = \frac{{\left( { - a} \right)^N }}{{\left( {z - a} \right)^{2N + 1} }}\left( { - 1} \right)^N \left( {\lambda  - 1} \right)^{2N + 1} \int_0^{ + \infty } {s^N e^{ - s} \frac{1}{{2\pi i}}\int_{\mathscr{C}^{\left(0\right)}\left(\pi\right)}  {\frac{{f^{ - N - 1} \left( {u,\lambda } \right)}}{{1 - \left( {s/a} \right)f^{ - 1} \left( {u,\lambda } \right)}}du} ds} .
\end{equation}
Now we make the change of variable
\[
s = t\frac{{\left| {f\left( { - \log \lambda ,\lambda } \right) - f\left( {0,\lambda } \right)} \right|}}{{f\left( { - \log \lambda ,\lambda } \right) - f\left( {0,\lambda } \right)}}f\left( {u,\lambda } \right) =  - tf\left( {u,\lambda } \right).
\]
Clearly, by the definition of the adjacent contour, $t$ is positive. The quantity $f\left( { - \log \lambda ,\lambda } \right) - f\left( {0,\lambda } \right) = 1 + \log \lambda  - \lambda$ was essentially called a ``singulant" by Dingle \cite[p. 147]{Dingle}. With this change of variable, the representation \eqref{eq6} for $R_N \left( {a,\lambda } \right)$ becomes
\begin{equation}\label{eq8}
R_N \left( {a,\lambda } \right) = \frac{{\left( { - a} \right)^N }}{{\left( {z - a} \right)^{2N + 1} }}\left( {\lambda  - 1} \right)^{2N + 1} \int_0^{ + \infty } {\frac{{t^N }}{{1 + t/a}}\frac{{ - 1}}{{2\pi i}}\int_{\mathscr{C}^{\left(0\right)}\left(\pi\right)} {e^{tf\left( {u,\lambda } \right)} du} dt} ,
\end{equation}
for $-\pi <\theta<\pi$ and $N\geq0$. To evaluate the contour integral, we proceed as follows. We substitute $u =  - v - \log \lambda$ and revert the orientation of the resulting path of integration, to get
\begin{equation}\label{eq7}
\frac{{ - 1}}{{2\pi i}}\int_{\mathscr{C}^{\left(0\right)}\left(\pi\right)} {e^{tf\left( {u,\lambda } \right)} du}  = -\frac{{e^{ - t\left( {\lambda  - \log \lambda  - 1} \right)} }}{{2\pi i}}\int_{\mathscr{H}} {e^{t\left( {e^{ - v}  + v - 1} \right)} dv} .
\end{equation}
The path $\mathscr{H}$ of integration is shown in Figure \ref{fig3}. Using Hankel's formula \cite[5.9.E2]{NIST} for the reciprocal of the gamma function, we find
\[
\frac{1}{{\Gamma \left( t \right)}} = \frac{t}{{\Gamma \left( {t + 1} \right)}} = \frac{t}{{2\pi i}}\int_{ - \infty }^{\left( {0 + } \right)} {e^s s^{ - t - 1} ds}  =  - \frac{{t^{1 - t} e^t }}{{2\pi i}}\int_{ - \infty }^{\left( {0 + } \right)} {e^{t\left( {e^{ - v}  + v - 1} \right)} dv}  =  - \frac{{t^{1 - t} e^t }}{{2\pi i}}\int_{\mathscr{H}} {e^{t\left( {e^{ - v}  + v - 1} \right)} dv} ,
\]
where we have substituted $s = te^{ - v}$ and have used the fact that, by Cauchy's theorem, the paths of integration can be deformed. Combining this expression with \eqref{eq7} and using the definition of the scaled gamma function, we obtain
\[
\frac{{ - 1}}{{2\pi i}}\int_{\mathscr{C}^{\left(0\right)}\left(\pi\right)} {e^{tf\left( {u,\lambda } \right)} du}  = \frac{1}{{\sqrt {2\pi } }}t^{ - \frac{1}{2}} e^{ - t\left( {\lambda  - \log \lambda  - 1} \right)} \frac{1}{{\Gamma^\ast  \left( t \right)}}.
\]
Substitution into \eqref{eq8} then yields the desired result \eqref{eq9}. To prove the second representation in \eqref{eq10}, we apply \eqref{eq9} for the right-hand side of
\begin{equation}\label{eq19}
b_n \left( \lambda  \right) = \frac{{\left( {z - a} \right)^{2n + 1} }}{{\left( { - a} \right)^n }}\left( {R_n \left( {a,\lambda } \right) - R_{n + 1} \left( {a,\lambda } \right)} \right).
\end{equation}

\begin{figure}[!t]
\def\svgwidth{0.6\columnwidth}
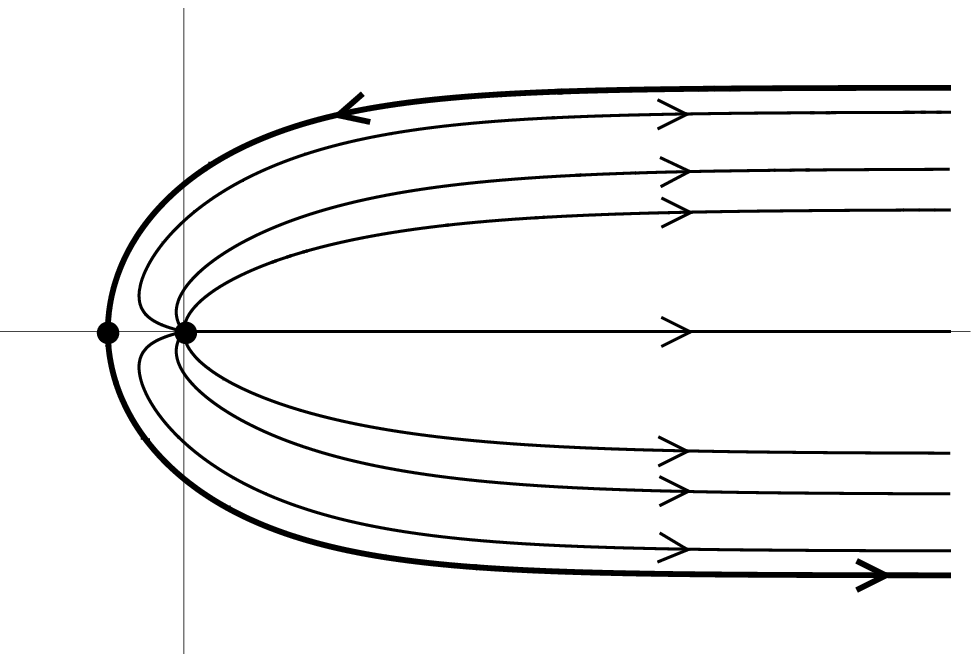
\caption{The path $\mathscr{P}\left(\theta\right)$ emanating from the origin when (i) $\theta=0$, (ii) $\theta=-\frac{\pi}{2}$, (iii) $\theta=-\frac{2\pi}{3}$, (iv) $\theta=-\frac{9\pi}{10}$, (v) $\theta=\frac{\pi}{2}$, (vi) $\theta=\frac{2\pi}{3}$ and (vii) $\theta=\frac{9\pi}{10}$. The path $\mathscr{C} ^{\left( 0 \right)} \left( { \pi} \right)$ is the adjacent contour to $0$. The domain $\Delta$ comprises all points inside this path.}
\label{fig2}
\end{figure}

\begin{figure}[!t]
\def\svgwidth{0.5\columnwidth}
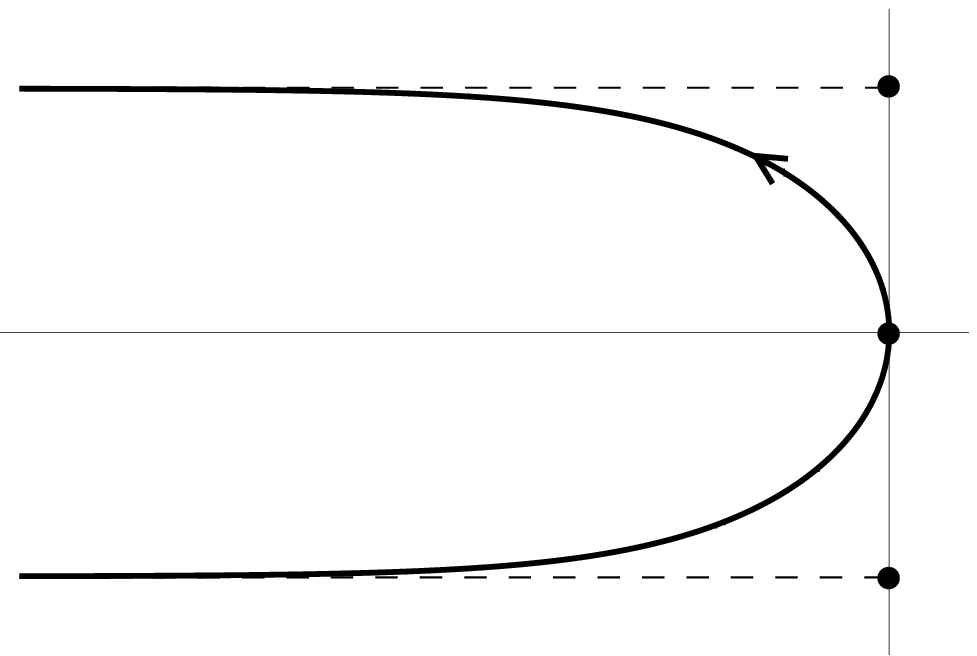
\caption{The path $\mathscr{H}$ of integration in \eqref{eq7}.}
\label{fig3}
\end{figure}

\subsection{Case (ii): $\lambda=1$} Let $f\left( t \right) = e^t - t - 1$. Hence, we can write \eqref{eq1} as
\[
\Gamma \left( {z,z} \right) = z^z e^{ - z} \int_0^{ + \infty } {e^{ - zf\left( t \right)} dt} 
\]
for any $\Re \left( z \right) > 0$. For simplicity, we assume that $z>0$. In due course, we shall appeal to an analytic continuation argument to extend our results to complex $z$. If
\begin{equation}\label{eq11}
\tau = f\left(t\right),
\end{equation}
then $\tau$ is real on the curve $\mathscr{P}\left(0\right)$, and, as $t$ travels along this curve from $0$ to $+\infty$, $\tau$ increases from $0$ to $+\infty$. Therefore, corresponding to each positive value of $\tau$, there is a value of $t$, say $t\left(\tau\right)$, satisfying \eqref{eq11} with $t\left(\tau\right)>0$. In terms of $\tau$, we have
\[
\Gamma \left( {z,z} \right) = z^z e^{ - z} \int_0^{ + \infty } {e^{ - z\tau } \frac{{dt}}{{d\tau }}d\tau}  = z^z e^{ - z} \int_0^{ + \infty } {\frac{{e^{ - z\tau } }}{{f'\left( {t\left( \tau  \right)} \right)}}d\tau} .
\]
As in the first case, we express the function involving $t\left(\tau\right)$ as a contour integral using the residue theorem, to obtain
\[
\Gamma \left( {z,z} \right) = \mathop {\lim }\limits_{T \to  + \infty } \frac{{z^z e^{ - z} }}{2}\int_0^{ T } {\tau ^{ - \frac{1}{2}} e^{ - z\tau } \frac{1}{{2\pi i}}\oint_\Gamma  {\frac{{f^{ - \frac{1}{2}} \left( u \right)}}{{1 - \tau ^{\frac{1}{2}} f^{ - \frac{1}{2}} \left( u \right)}}du} } d\tau ,
\]
where the contour $\Gamma=\Gamma\left(\tau\right)$ encircles $t\left(\tau\right)$ in the anti-clockwise direction and does not enclose any of the saddle points $t^{\left(k\right)}\neq t^{\left(0\right)}$. The square root is defined so that ${f^{\frac{1}{2}} \left( t \right)}$ is positive
on the path $\mathscr{P}\left(0\right)$. Next we apply the expression \eqref{eq12} to expand the function under the contour integral in powers of $\tau ^{\frac{1}{2}} f^{ - \frac{1}{2}} \left( u \right)$. The result is
\[
\Gamma \left( {z,z} \right) = \sqrt {\frac{\pi }{2}} z^{z - \frac{1}{2}} e^{ - z} \left( {\sqrt {\frac{z}{{2\pi }}} \sum\limits_{n = 0}^{N - 1} {\int_0^{ + \infty } {\tau ^{\frac{{n - 1}}{2}} e^{ - z\tau } \frac{1}{{2\pi i}}\oint_\Gamma  {\frac{{du}}{{f^{\frac{{n + 1}}{2}} \left( u \right)}}} d\tau } }  + R_N \left( z \right)} \right)
\]
for $N \geq 2$, where
\begin{equation}\label{eq13}
R_N \left( z \right) = \sqrt {\frac{z}{{2\pi }}} \int_0^{ + \infty } {\tau ^{\frac{{N - 1}}{2}} e^{ - z\tau } \frac{1}{{2\pi i}}\oint_\Gamma  {\frac{{f^{ - \frac{{N + 1}}{2}} \left( u \right)}}{{1 - \tau ^{\frac{1}{2}} f^{ - \frac{1}{2}} \left( u \right)}}du} d\tau } ,
\end{equation}
and the contour $\Gamma$ can be taken to be a path that encloses $\mathscr{P}\left(0\right)$ and has positive orientation (cf. Figure \ref{fig1}). The omission of the limiting process is justified by the convergence of the integral for $N \geq 2$. The path $\Gamma$ in the sum can be shrunk into a small circle around $t^{\left(0\right)}=0$, and we arrive at
\[
\Gamma \left( {z,z} \right) = \sqrt {\frac{\pi }{2}} z^{z - \frac{1}{2}} e^{ - z} \left( {\sum\limits_{n = 0}^{N - 1} {\frac{{a_n }}{{z^{\frac{n}{2}} }}}  + R_N \left( z \right)} \right)
\]
where
\[
a_n  = \frac{{\Gamma \left( {\frac{{n + 1}}{2}} \right)}}{{\sqrt {2\pi } }}\frac{1}{{2\pi i}}\oint_{\left( {0^ +  } \right)} {\frac{{du}}{{f^{\frac{{n + 1}}{2}} \left( u \right)}}}  = \frac{1}{{2^{\frac{n}{2}} \Gamma \left( {\frac{n}{2} + 1} \right)}}\left[ {\frac{{d^n }}{{dt^n }}\left( {\frac{1}{2}\frac{{t^2 }}{{e^t  - t - 1}}} \right)^{\frac{{n + 1}}{2}} } \right]_{t = 0} .
\]

\begin{figure}[!t]
\def\svgwidth{0.5\columnwidth}
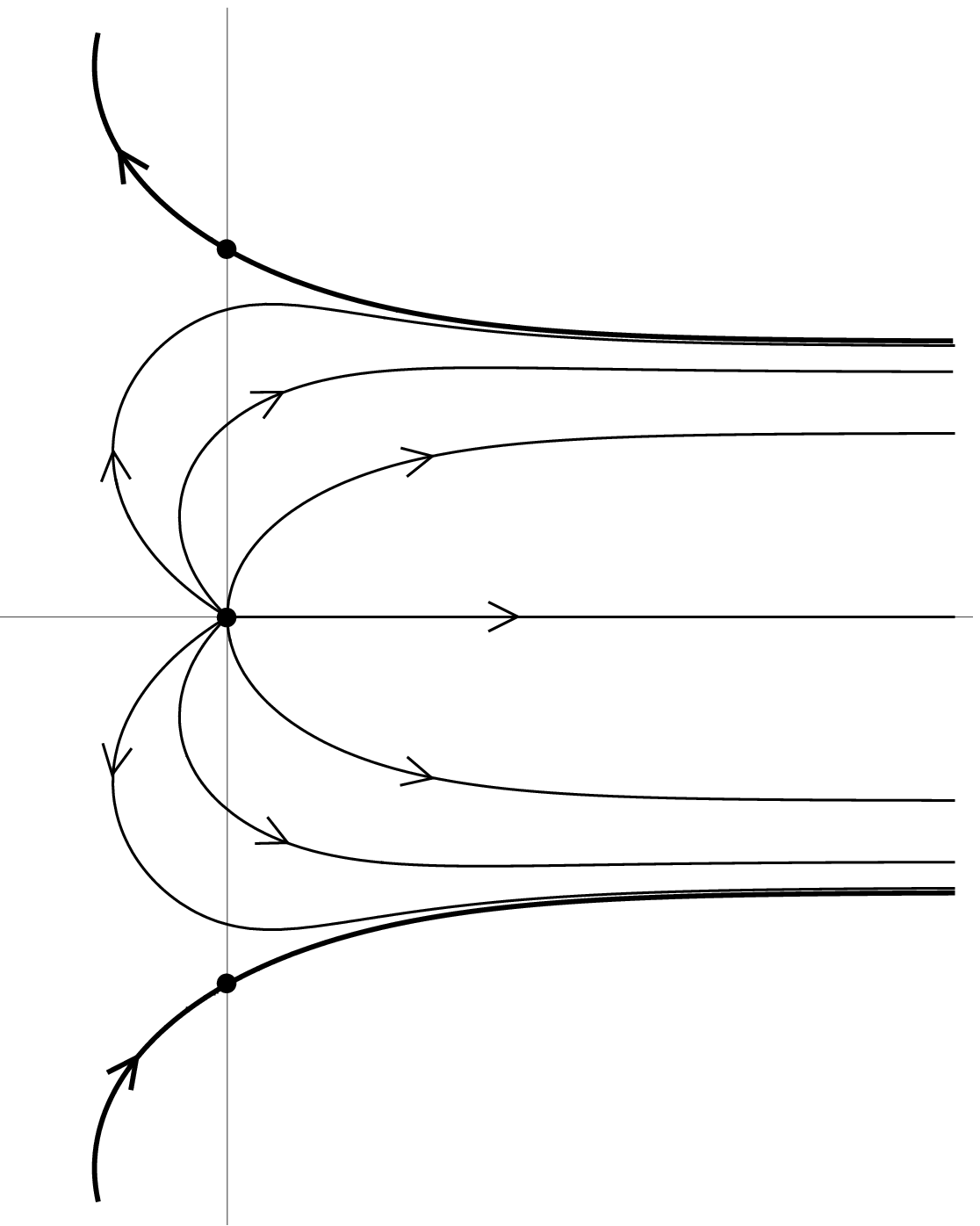
\caption{The path $\mathscr{P}\left(\theta\right)$ emanating from the saddle point $t^{\left(0\right)}$ when (i) $\theta=0$, (ii) $\theta=-\pi$, (iii) $\theta=-\frac{4\pi}{3}$, (iv) $\theta=-\frac{59\pi}{40}$, (v) $\theta=\pi$, (vi) $\theta=\frac{4\pi}{3}$ and (vii) $\theta=\frac{59\pi}{40}$. The paths $\mathscr{C} ^{\left( 1 \right)} \left( { -\frac{\pi}{2}} \right)$ and $\mathscr{C} ^{\left( -1 \right)} \left( { \frac{\pi}{2}} \right)$ are the adjacent contours to $0$.}
\label{fig4}
\end{figure}

Applying the change of variable $z\tau = s$ in \eqref{eq13} gives
\begin{equation}\label{eq14}
R_N \left( z \right) = \frac{1}{{\sqrt {2\pi } z^{\frac{N}{2}} }}\int_0^{ + \infty } {s^{\frac{{N - 1}}{2}} e^{ - s} \frac{1}{{2\pi i}}\oint_\Gamma  {\frac{{f^{ - \frac{{N + 1}}{2}} \left( u \right)}}{{1 - \left( {s/z} \right)^{\frac{1}{2}} f^{ - \frac{1}{2}} \left( u \right)}}du} ds} .
\end{equation}
As in the first case, we need to locate the adjacent saddle points. When $\theta = -\frac{3\pi}{2}$, the path $\mathscr{P}\left(\theta\right)$ connects to the saddle point $t^{\left(1\right)}=2\pi i$. Similarly, when $\theta = \frac{3\pi}{2}$, the path $\mathscr{P}\left(\theta\right)$ connects to the saddle point $t^{\left(-1\right)}=-2\pi i$. Therefore, the adjacent saddles are $t^{\left(\pm1\right)}$. The set
\[
\Delta  = \left\{ {u \in \mathscr{P}\left( \theta  \right) : - \frac{3\pi}{2} < \theta  < \frac{3\pi}{2}} \right\}
\]
forms a domain in the complex plane whose boundary contains portions of steepest descent paths through the adjacent saddles (see Figure \ref{fig4}). These paths are $\mathscr{C} ^{\left( 1 \right)} \left( { -\frac{\pi }{2}} \right)$ and $\mathscr{C}^{\left( -1 \right)} \left( { \frac{\pi }{2}} \right)$, the adjacent contours to the saddle point $t ^{\left( 0 \right)}$. For $N\geq 2$ and fixed $z$, the function under the contour integral in \eqref{eq14} is an analytic function of $u$ in the domain $\Delta$ (and in the right-half plane) excluding $\mathscr{P}\left(\theta\right)$, therefore we can deform $\Gamma$ over the adjacent contours. We thus find that for $-\frac{3\pi}{2} < \theta < \frac{3\pi}{2}$ and $N \geq 2$, \eqref{eq14} may be written
\begin{gather}\label{eq15}
\begin{split}
R_N \left( z \right) = \; & \frac{1}{{\sqrt {2\pi } z^{\frac{N}{2}} }}\int_0^{ + \infty } {s^{\frac{{N - 1}}{2}} e^{ - s} \frac{1}{{2\pi i}}\int_{\mathscr{C} ^{\left( 1 \right)} \left( { -\frac{\pi }{2}} \right)}  {\frac{{f^{ - \frac{{N + 1}}{2}} \left( u \right)}}{{1 - \left( {s/z} \right)^{\frac{1}{2}} f^{ - \frac{1}{2}} \left( u \right)}}du} ds}\\ & + \frac{1}{{\sqrt {2\pi } z^{\frac{N}{2}} }}\int_0^{ + \infty } {s^{\frac{{N - 1}}{2}} e^{ - s} \frac{1}{{2\pi i}}\int_{\mathscr{C} ^{\left( -1 \right)} \left( { \frac{\pi }{2}} \right)}  {\frac{{f^{ - \frac{{N + 1}}{2}} \left( u \right)}}{{1 - \left( {s/z} \right)^{\frac{1}{2}} f^{ - \frac{1}{2}} \left( u \right)}}du} ds}.
\end{split}
\end{gather}
Now we perform the changes of variable
\[
s = t\frac{{\left| {f\left( {2\pi i} \right)} \right|}}{{f\left( {2\pi i} \right)}}f\left( u \right) = itf\left( u \right)
\]
in the first, and
\[
s = t\frac{{\left| {f\left( { - 2\pi i} \right)} \right|}}{{f\left( { - 2\pi i} \right)}}f\left( u \right) =  - itf\left( u \right)
\]
in the second double integral. By the definition of the steepest descent paths, $t$ is positive. In this case, Dingle's singulants are $f\left( {\pm 2\pi i} \right) = \mp 2\pi i$. When using these changes of variable, we should take $i^{\frac{1}{2}}  = e^{ - \frac{3\pi i}{4}}$ in the first, and $\left( { - i} \right)^{\frac{1}{2}}  = e^{\frac{3\pi i}{4}}$ in the second double integral. With these changes of variable, the representation \eqref{eq15} for $R_N \left( z \right)$ becomes
\begin{gather}\label{eq16}
\begin{split}
R_N \left( z \right) = \; & \frac{{e^{ - \frac{{3\pi }}{4}\left( {N + 1} \right)i} }}{{\sqrt {2\pi } z^{\frac{N}{2}} }}\int_0^{ + \infty } {\frac{{t^{\frac{{N - 1}}{2}} }}{{1 - e^{ - \frac{{3\pi i}}{4}} \left( {t/z} \right)^{\frac{1}{2}} }}\frac{1}{{2\pi i}}\int_{\mathscr{C}^{\left( 1 \right)} \left( { - \frac{\pi }{2}} \right)} {e^{ - itf\left( u \right)} du} dt} 
\\ & + \frac{{e^{\frac{{3\pi }}{4}\left( {N + 1} \right)i} }}{{\sqrt {2\pi } z^{\frac{N}{2}} }}\int_0^{ + \infty } {\frac{{t^{\frac{{N - 1}}{2}} }}{{1 - e^{\frac{{3\pi i}}{4}} \left( {t/z} \right)^{\frac{1}{2}} }}\frac{1}{{2\pi i}}\int_{\mathscr{C}^{\left( { - 1} \right)} \left( {\frac{\pi }{2}} \right)} {e^{itf\left( u \right)} du} dt} ,
\end{split}
\end{gather}
for $- \frac{3\pi}{2} < \theta  < \frac{3\pi}{2}$ and $N \geq 2$. The contour integrals can themselves be represented in terms of the scaled gamma function since
\begin{multline*}
\int_{\mathscr{C}^{\left( 1 \right)} \left( { - \frac{\pi }{2}} \right)} {e^{ - itf\left( u \right)} du}  = - \int_{\mathscr{C}^{\left( 0 \right)} \left( { - \frac{\pi }{2}} \right)} {e^{ - itf\left( {u + 2\pi i} \right)}  du} \\  = - e^{ - 2\pi t} \int_{\mathscr{C}^{\left( 0 \right)} \left( { - \frac{\pi }{2}} \right)} {e^{ - itf\left( u \right)} du}  = - e^{ - 2\pi t} \sqrt {\frac{{2\pi }}{{it}}} \Gamma^\ast \left( {it} \right),
\end{multline*}
and
\begin{multline*}
\int_{\mathscr{C}^{\left( { - 1} \right)} \left( {\frac{\pi }{2}} \right)} {e^{itf\left( u \right)} du}  = \int_{\mathscr{C}^{\left( 0 \right)} \left( {\frac{\pi }{2}} \right)} {e^{itf\left( {u - 2\pi i} \right)} du} \\ = e^{ - 2\pi t} \int_{\mathscr{C}^{\left( 0 \right)} \left( {\frac{\pi }{2}} \right)} {e^{itf\left( u \right)} du}  = e^{ - 2\pi t} \sqrt {\frac{{2\pi }}{{ - it}}} \Gamma^\ast \left( { - it} \right),
\end{multline*}
where we should take $i^{\frac{1}{2}}  = e^{\frac{\pi}{4} i}$ and $\left( { - i} \right)^{\frac{1}{2}}  = e^{-\frac{\pi}{4} i}$. Substituting these into \eqref{eq16} gives \eqref{eq17}. To prove the second representation in \eqref{eq18}, we apply \eqref{eq17} for the right-hand side of
\[
a_n  = z^{\frac{{n + 1}}{2}} \left( {R_n \left( z \right) - R_{n + 1} \left( z \right)} \right).
\]

\section{Error bounds for the asymptotic expansions of\\ the incomplete gamma function}\label{section3}

\subsection{Case (i): $\lambda>1$} Our error analysis is based on the representation \eqref{eq6}. Set $u=x+iy$. Along the path $\mathscr{C}^{\left( 0 \right)} \left( \pi  \right)$, we have 
\[
\arg \left( {f\left( { - \log \lambda ,\lambda } \right) - f\left( {u,\lambda } \right)} \right) = 0,
\]
which implies that
\[
f\left( {u,\lambda } \right) = \Re \left( {f\left( {u,\lambda } \right)} \right) = \lambda e^x \cos y - x - \lambda .
\]
Therefore, along $\mathscr{C}^{\left( 0 \right)} \left( \pi  \right)$, we have $f\left( {x + iy,\lambda } \right) = f\left( {x - iy,\lambda } \right)$. Denote by $\mathscr{C}_{\pm}^{\left( 0 \right)} \left( \pi  \right)$ the portions of the steepest path $\mathscr{C}^{\left( 0 \right)} \left( \pi  \right)$
that lie in the upper- and in the lower-half plane, respectively. Using the previous observation and the fact that $\mathscr{C}^{\left( 0 \right)} \left( \pi  \right)$ is symmetric with respect to the real axis (see Figure \ref{fig2}), we find
\begin{align*}
& \int_{\mathscr{C}^{\left( 0 \right)} \left( \pi  \right)} {\frac{{f^{ - N - 1} \left( {u,\lambda } \right)}}{{1 - \left( {s/a} \right)f^{ - 1} \left( {u,\lambda } \right)}}du} 
\\ & = \int_{\mathscr{C}_ + ^{\left( 0 \right)} \left( \pi  \right)} {\frac{{f^{ - N - 1} \left( {x + iy,\lambda } \right)}}{{1 - \left( {s/a} \right)f^{ - 1} \left( {x + iy,\lambda } \right)}}d\left( {x + iy} \right)}  + \int_{\mathscr{C}_ - ^{\left( 0 \right)} \left( \pi  \right)} {\frac{{f^{ - N - 1} \left( {x + iy,\lambda } \right)}}{{1 - \left( {s/a} \right)f^{ - 1} \left( {x + iy,\lambda } \right)}}d\left( {x + iy} \right)} 
\\ & =  - \int_{\mathscr{C}_ - ^{\left( 0 \right)} \left( \pi  \right)} {\frac{{f^{ - N - 1} \left( {x - iy,\lambda } \right)}}{{1 - \left( {s/a} \right)f^{ - 1} \left( {x - iy,\lambda } \right)}}d\left( {x - iy} \right)}  + \int_{\mathscr{C}_ - ^{\left( 0 \right)} \left( \pi  \right)} {\frac{{f^{ - N - 1} \left( {x + iy,\lambda } \right)}}{{1 - \left( {s/a} \right)f^{ - 1} \left( {x + iy,\lambda } \right)}}d\left( {x + iy} \right)} 
\\ & =  - \int_{\mathscr{C}_ - ^{\left( 0 \right)} \left( \pi  \right)} {\frac{{f^{ - N - 1} \left( {x + iy,\lambda } \right)}}{{1 - \left( {s/a} \right)f^{ - 1} \left( {x + iy,\lambda } \right)}}d\left( {x - iy} \right)}  + \int_{\mathscr{C}_ - ^{\left( 0 \right)} \left( \pi  \right)} {\frac{{f^{ - N - 1} \left( {x + iy,\lambda } \right)}}{{1 - \left( {s/a} \right)f^{ - 1} \left( {x + iy,\lambda } \right)}}d\left( {x + iy} \right)} 
\\ & =  - 2i\int_{\mathscr{C}_ - ^{\left( 0 \right)} \left( \pi  \right)} {\frac{{f^{ - N - 1} \left( {u,\lambda } \right)}}{{1 - \left( {s/a} \right)f^{ - 1} \left( {u,\lambda } \right)}}d\left( { - y} \right)} .
\end{align*}
Substitution into \eqref{eq6} then gives
\begin{equation}\label{eq20}
R_N \left( {a,\lambda } \right) = \frac{{\left( { - a} \right)^N }}{{\left( {z - a} \right)^{2N + 1} }}\left( { - 1} \right)^{N+1} \left( {\lambda  - 1} \right)^{2N + 1} \int_0^{ + \infty } {s^N e^{ - s} \frac{1}{\pi }\int_{\mathscr{C}_ - ^{\left( 0 \right)} \left( \pi  \right)} {\frac{{f^{ - N - 1} \left( {u,\lambda } \right)}}{{1 - \left( {s/a} \right)f^{ - 1} \left( {u,\lambda } \right)}}d\left( { - y} \right)} ds} .
\end{equation}
By substituting this expression into the right-hand side of \eqref{eq19}, we also have
\begin{equation}\label{eq21}
b_N \left( \lambda  \right) = \left( { - 1} \right)^{N+1} \left( {\lambda  - 1} \right)^{2N + 1} \int_0^{ + \infty } {s^N e^{ - s} \frac{1}{\pi }\int_{\mathscr{C}_ - ^{\left( 0 \right)} \left( \pi  \right)} {f^{ - N - 1} \left( {u,\lambda } \right)d\left( { - y} \right)} ds} .
\end{equation}
Noting that $f\left( {u,\lambda } \right)$ is negative and $-y$ is positive and monotonically increasing along $\mathscr{C}_ - ^{\left( 0 \right)} \left( \pi  \right)$, and that for $s>0$ and $u\in \mathscr{C}_ - ^{\left( 0 \right)} \left( \pi  \right)$,
\[
\frac{1}{{\left| {1 - \left( {s/a} \right)f^{ - 1} \left( {u,\lambda } \right)} \right|}} = \frac{1}{{\left| {1 + \left( {s/\left| a \right|} \right)\left| {f^{ - 1} \left( {u,\lambda } \right)} \right|e^{ - i\theta } } \right|}} \le \begin{cases} \left|\csc \theta \right| & \; \text{ if } \; \frac{\pi}{2} < \left|\theta\right| <\pi, \\ 1 & \; \text{ if } \; \left|\theta\right| \leq \frac{\pi}{2}, \end{cases}
\]
trivial estimation of \eqref{eq20} and the expression \eqref{eq21} yield the error bound
\[
\left| {R_N \left( {a,\lambda } \right)} \right| \le \left| {\frac{{\left( { - a} \right)^N b_N \left( \lambda  \right)}}{{\left( {z - a} \right)^{2N + 1} }}} \right| \begin{cases} \left|\csc \theta \right| & \; \text{ if } \; \frac{\pi}{2} < \left|\theta\right| <\pi, \\ 1 & \; \text{ if } \; \left|\theta\right| \leq \frac{\pi}{2}. \end{cases}
\]
Here, and subsequently we write $\arg a = \theta$.  In addition, if $a > 0$ we have $0 < 1/\left( {1 - \left( {s/a} \right)f^{ - 1} \left( {u,\lambda } \right)} \right) < 1$ in \eqref{eq20}, and the mean value theorem of integration shows that
\[
R_N \left( {a,\lambda } \right) = \frac{{\left( { - a} \right)^N b_N \left( \lambda  \right)}}{{\left( {z - a} \right)^{2N + 1} }}\Theta ,
\]
where $0<\Theta<1$ is a suitable number depending on $a$, $\lambda$ and $N$.

Our bound for $R_N \left( {a,\lambda } \right)$ is unrealistic near the Stokes lines $\theta = \pm \pi$ due to the presence of the factor $\csc \theta$. To derive an estimate which is realistic near and behind these lines, we proceed as follows. Let $0<\varphi<\frac{\pi}{2}$ be an acute angle that may depend on $N$ and suppose that $\frac{\pi}{2}+\varphi<\theta<\pi+\varphi$. We rotate the path of integration in \eqref{eq20} by $\varphi$, to obtain the analytic continuation of the representation $R_N \left( {a,\lambda } \right)$ to the sector $\frac{\pi}{2}+\varphi<\theta<\pi+\varphi$:
\begin{equation}\label{eq23}
R_N \left( {a,\lambda } \right) = \frac{{\left( { - a} \right)^N }}{{\left( {z - a} \right)^{2N + 1} }}\left( { - 1} \right)^{N + 1} \left( {\lambda  - 1} \right)^{2N + 1} \int_0^{ + \infty e^{i\varphi } } {s^N e^{ - s} \frac{1}{\pi }\int_{\mathscr{C}_ - ^{\left( 0 \right)} \left( \pi  \right)} {\frac{{f^{ - N - 1} \left( {u,\lambda } \right)}}{{1 - \left( {s/a} \right)f^{ - 1} \left( {u,\lambda } \right)}}d\left( { - y} \right)} ds} .
\end{equation}
With this representation, the expansion \eqref{eq22} is well defined in the region $\frac{\pi}{2}+\varphi<\theta<\pi+\varphi$. Simple estimation of \eqref{eq23} shows that
\begin{equation}\label{eq24}
\left| {R_N \left( {a,\lambda } \right)} \right| \le \frac{{\csc \left( {\theta  - \varphi } \right)}}{{\cos ^{N + 1} \varphi }}\left| {\frac{{\left( { - a} \right)^N b_N \left( \lambda  \right)}}{{\left( {z - a} \right)^{2N + 1} }}} \right|.
\end{equation}
The minimisation of the factor $\csc \left( {\theta  - \varphi } \right)\cos ^{ - N - 1} \varphi$ as a function of $\varphi$ can be done using a lemma of Meijer \cite[pp. 953--954]{Meijer}. In our case, Meijer's lemma gives that the minimising value $\varphi = \varphi^\ast$ in \eqref{eq24}, is the unique solution of the equation
\begin{equation}\label{eq25}
\left( {N + 2} \right)\cos \left( {\theta  - 2\varphi^\ast } \right) = N\cos \theta 
\end{equation}
that satisfies $-\pi+\theta< \varphi^\ast <\frac{\pi}{2}$ if $\pi \leq \theta < \frac{3\pi}{2}$, and $0< \varphi^\ast <-\frac{\pi}{2}+\theta$ if $\frac{\pi}{2} < \theta < \pi$. With this choice of $\varphi$, \eqref{eq24} provides an error bound for the range $\frac{\pi}{2}<\theta<\frac{3\pi}{2}$.

To obtain the bound for the sector $-\frac{3\pi}{2}<\theta<-\frac{\pi}{2}$, we rotate the path of integration in \eqref{eq20} by $-\frac{\pi}{2}<\varphi<0$ and an argument similar to the above shows that
\[
\left| {R_N \left( {a,\lambda } \right)} \right| \le -\frac{{\csc \left( {\theta  - \varphi^\ast } \right)}}{{\cos ^{N + 1} \varphi^\ast }}\left| {\frac{{\left( { - a} \right)^N b_N \left( \lambda  \right)}}{{\left( {z - a} \right)^{2N + 1} }}} \right|,
\]
where $\varphi^\ast$ is the unique solution of the equation \eqref{eq25}, that satisfies $-\frac{\pi}{2}<\varphi^\ast<\pi + \theta$ if $-\frac{3\pi}{2}<\theta\leq -\pi$, and $\theta+\frac{\pi}{2}<\varphi^\ast<0$ if $-\pi<\theta<-\frac{\pi}{2}$.

We can make our bounds simpler if $\arg a$ is close to $\pm \pi$ (i.e., close to the Stokes lines) as follows. When $\arg a= \pi$, the
minimising value $\varphi^\ast$ is given explicitly by
\[
\varphi^\ast = \arctan \left( {\frac{1}{{\sqrt {N + 1} }}} \right),
\]
and therefore we have
\[
\frac{{\csc \left( {\theta  - \varphi^\ast } \right)}}{{\cos ^{N + 1} \varphi^\ast}} \le \frac{{\csc \left( { \pi  - \varphi^\ast } \right)}}{{\cos ^{N + 1} \varphi^\ast }} = \left( {1 + \frac{1}{{N + 1}}} \right)^{\frac{N}{2} + 1} \sqrt {N + 1}  \le \sqrt {e\left( {N + \frac{3}{2}} \right)} ,
\]
as long as $\frac{\pi }{2} + \varphi^\ast < \theta  \le \pi$. We also have
\[
\sqrt {e\left( {N + \frac{3}{2}} \right)} \ge \sqrt {\frac{{3e}}{2}}  \ge \csc \theta ,
\]
for $\frac{\pi}{2} < \theta  \le \frac{\pi }{2} + \varphi^\ast \le \frac{\pi }{2} + \arctan \left( {\frac{1}{{\sqrt 2 }}} \right)$, therefore
\begin{equation}\label{eq98}
\left| {R_N \left( {a,\lambda } \right)} \right| \le \sqrt {e\left( {N + \frac{3}{2}} \right)} \left| {\frac{{\left( { - a} \right)^N b_N \left( \lambda  \right)}}{{\left( {z - a} \right)^{2N + 1} }}} \right|,
\end{equation}
provided that $\frac{\pi}{2} < \theta \leq \pi$. A similar argument shows that this bound is also true when $-\pi \leq \theta < - \frac{\pi}{2}$. The appearance of the factor proportional to $\sqrt{N}$ in this bound may give the impression that this estimate is unrealistic for large $N$, but this is not the case. Applying the asymptotic form of the coefficients $b_N\left(\lambda\right)$ (see Section \ref{section4}), it can readily be shown that \eqref{eq98} implies
\[
R_N \left( {a,\lambda } \right) = \mathcal{O}\left( {\frac{{\sqrt N \Gamma \left( {N + \frac{1}{2}} \right)}}{{\left|a\right|^{N + 1} \left( {\lambda  - \log \lambda  - 1} \right)^{N + \frac{1}{2}} }}} \right)
\]
for large $N$. When the asymptotic series is truncated optimally, i.e., $N \approx \left|a\right|\left( {\lambda  - \log \lambda  - 1} \right)$, we find with the help of Stirling's formula that the above estimate is equivalent to
\[
R_N \left( {a,\lambda } \right) = \mathcal{O}\left( {\frac{{e^{ - \left| a \right|\left( {\lambda  - \log \lambda  - 1} \right)} }}{{\left| a \right|^{\frac{1}{2}} }}} \right).
\]
Noting that $R_N \left( {a,\lambda } \right) = R_{N,0} \left( {a,\lambda } \right)$, this estimate is a special case of the exponentially improved version, given in Theorem \ref{thm3}. Therefore, our bound is indeed realistic near the Stokes lines.

\subsection{Case (ii): $\lambda=1$} In this case, the argument of the previous subsection does not work, since for $u=x+iy \in \mathscr{C}^{\left( { \pm 1} \right)} \left( { \mp \frac{\pi }{2}} \right)$, $dx$ would appear in the analysis but $x$ is not monotonic along the steepest paths. Hence, we have to proceed in a different way. First we consider the range $\left|\arg z\right| \leq \pi$. From the reflection principle, it follows that $\Gamma^\ast \left( {\overline w} \right) = \overline {\Gamma^\ast \left( w \right)} $, and in particular $\Gamma^\ast \left( { \pm it} \right) = \Re \Gamma^\ast  \left( {it} \right) \pm \Im \Gamma^\ast  \left( {it} \right)$ for $t>0$. Substituting this expression into \eqref{eq17} gives
\begin{gather}\label{eq37}
\begin{split}
R_{4N} \left( z \right) = \; &  \frac{{\left( { - 1} \right)^{N+1} }}{{\pi z^{2N} }}\int_0^{ + \infty } {\frac{{t^{2N - 1} e^{ - 2\pi t} }}{{\left( {1 - e^{ - \frac{{3\pi i}}{4}} \left( {t/z} \right)^{\frac{1}{2}} } \right)\left( {1 - e^{\frac{{3\pi i}}{4}} \left( {t/z} \right)^{\frac{1}{2}} } \right)}}\left(-\Im \Gamma^\ast\left( {it} \right)\right)dt} \\ & + \frac{{\left( { - 1} \right)^{N + 1} }}{{\sqrt 2 \pi z^{2N + \frac{1}{2}} }}\int_0^{ + \infty } {\frac{{t^{2N - \frac{1}{2}} e^{ - 2\pi t} }}{{\left( {1 - e^{ - \frac{{3\pi i}}{4}} \left( {t/z} \right)^{\frac{1}{2}} } \right)\left( {1 - e^{\frac{{3\pi i}}{4}} \left( {t/z} \right)^{\frac{1}{2}} } \right)}}\left(\Re \Gamma^\ast  \left( {it} \right) -\Im \Gamma^\ast \left( {it} \right)\right)dt},
\end{split}
\end{gather}
\begin{gather}\label{eq27}
\begin{split}
R_{4N + 1} \left( z \right) = \; & \frac{{\left( { - 1} \right)^{N + 1} }}{{\sqrt 2 \pi z^{2N + \frac{1}{2}} }}\int_0^{ + \infty } {\frac{{t^{2N - \frac{1}{2}} e^{ - 2\pi t} }}{{\left( {1 - e^{ - \frac{{3\pi i}}{4}} \left( {t/z} \right)^{\frac{1}{2}} } \right)\left( {1 - e^{\frac{{3\pi i}}{4}} \left( {t/z} \right)^{\frac{1}{2}} } \right)}}\left(\Re \Gamma^\ast \left( {it} \right) +\Im \Gamma^\ast \left( {it} \right)\right)dt} 
\\ & + \frac{{\left( { - 1} \right)^{N} }}{{\pi z^{2N + 1} }}\int_0^{ + \infty } {\frac{{t^{2N} e^{ - 2\pi t} }}{{\left( {1 - e^{ - \frac{{3\pi i}}{4}} \left( {t/z} \right)^{\frac{1}{2}} } \right)\left( {1 - e^{\frac{{3\pi i}}{4}} \left( {t/z} \right)^{\frac{1}{2}} } \right)}}\left(-\Im \Gamma^\ast \left( {it} \right)\right)dt} ,
\end{split}
\end{gather}
\begin{gather}\label{eq33}
\begin{split}
R_{4N + 2} \left( z \right) = \; & \frac{{\left( { - 1} \right)^N }}{{\pi z^{2N + 1} }}\int_0^{ + \infty } {\frac{{t^{2N} e^{ - 2\pi t} }}{{\left( {1 - e^{ - \frac{{3\pi i}}{4}} \left( {t/z} \right)^{\frac{1}{2}} } \right)\left( {1 - e^{\frac{{3\pi i}}{4}} \left( {t/z} \right)^{\frac{1}{2}} } \right)}}\Re \Gamma^\ast \left( {it} \right)dt}
\\ & + \frac{{\left( { - 1} \right)^N }}{{\sqrt 2 \pi z^{2N + \frac{3}{2}} }}\int_0^{ + \infty } {\frac{{t^{2N + \frac{1}{2}} e^{ - 2\pi t} }}{{\left( {1 - e^{ - \frac{{3\pi i}}{4}} \left( {t/z} \right)^{\frac{1}{2}} } \right)\left( {1 - e^{\frac{{3\pi i}}{4}} \left( {t/z} \right)^{\frac{1}{2}} } \right)}}\left(\Re \Gamma^\ast \left( {it} \right) + \Im \Gamma^\ast \left( {it} \right) \right)dt}
\end{split}
\end{gather}
and
\begin{gather}\label{eq38}
\begin{split}
R_{4N + 3} \left( z \right) = \; & \frac{{\left( { - 1} \right)^{N + 1} }}{{\sqrt 2 \pi z^{2N + \frac{3}{2}} }}\int_0^{ + \infty } {\frac{{t^{2N + \frac{1}{2}} e^{ - 2\pi t} }}{{\left( {1 - e^{ - \frac{{3\pi i}}{4}} \left( {t/z} \right)^{\frac{1}{2}} } \right)\left( {1 - e^{\frac{{3\pi i}}{4}} \left( {t/z} \right)^{\frac{1}{2}} } \right)}}\left( {\Re \Gamma^\ast \left( {it} \right) - \Im \Gamma^\ast \left( {it} \right)} \right)dt} 
\\ & + \frac{{\left( { - 1} \right)^{N + 1} }}{{\pi z^{2N + 2} }}\int_0^{ + \infty } {\frac{{t^{2N + 1} e^{ - 2\pi t} }}{{\left( {1 - e^{ - \frac{{3\pi i}}{4}} \left( {t/z} \right)^{\frac{1}{2}} } \right)\left( {1 - e^{\frac{{3\pi i}}{4}} \left( {t/z} \right)^{\frac{1}{2}} } \right)}}\Re \Gamma^\ast \left( {it} \right)dt} ,
\end{split}
\end{gather}
Here, and throughout this subsection, we assume that the indices are at least $2$. Similar manipulation of the second representation in \eqref{eq18} yields the formulas
\begin{align}
a_{4N} & = \frac{{\left( { - 1} \right)^{N + 1} }}{\pi }\int_0^{ + \infty } {t^{2N - 1} e^{ - 2\pi t} \left( { - \Im \Gamma^\ast \left( {it} \right)} \right)dt}, \nonumber
\\ a_{4N + 1} & = \frac{{\left( { - 1} \right)^{N + 1} }}{{\sqrt 2 \pi }}\int_0^{ + \infty } {t^{2N - \frac{1}{2}} e^{ - 2\pi t} \left( {\Re \Gamma^\ast \left( {it} \right) + \Im \Gamma^\ast \left( {it} \right)} \right)dt} , \label{eq28}
\\ a_{4N + 2} & = \frac{{\left( { - 1} \right)^N }}{\pi }\int_0^{ + \infty } {t^{2N} e^{ - 2\pi t} \Re \Gamma^\ast \left( {it} \right)dt} , \label{eq29}
\\ a_{4N + 3} & = \frac{{\left( { - 1} \right)^{N + 1} }}{{\sqrt 2 \pi }}\int_0^{ + \infty } {t^{2N + \frac{1}{2}} e^{ - 2\pi t} \left( {\Re \Gamma^\ast \left( {it} \right) - \Im \Gamma^\ast \left( {it} \right)} \right)dt} . \label{eq34}
\end{align}
To proceed further, we need the following lemma.

\begin{lemma}\label{lemma1} For any $t>0$, the quantities $\Re \Gamma^\ast \left( {it} \right)$, $- \Im \Gamma^\ast \left( {it} \right)$, $\Re \Gamma^\ast \left( {it} \right) + \Im \Gamma^\ast  \left( {it} \right)$, $\Re \Gamma^\ast \left( {it} \right) - \Im \Gamma^\ast \left( {it} \right)$ are all non-negative.
\end{lemma}

\begin{proof} It was shown in \cite{Nemes} that the quantities $\Re \Gamma^\ast \left( {it} \right)$, $- \Im \Gamma^\ast \left( {it} \right)$ are non-negative if $t>0$. This implies that $\Re \Gamma^\ast \left( {it} \right) - \Im \Gamma^\ast \left( {it} \right)>0$ if $t>0$. It remains to show that $\Re \Gamma^\ast \left( {it} \right) + \Im \Gamma^\ast  \left( {it} \right)$ is non-negative if $t>0$. Let $2Q\left( s \right) = s - \left\lfloor s \right\rfloor  - \left( {s - \left\lfloor s \right\rfloor } \right)^2$. In the paper \cite{Nemes} it was shown that
\[
\Re \Gamma^\ast \left( {it} \right) = \left| {\Gamma^\ast \left( {it} \right)} \right|\cos \left( {\int_0^{ + \infty } {\frac{{2ts}}{{\left( {t^2  + s^2 } \right)^2 }}Q\left( s \right)ds} } \right),
\]
\[
\Im \Gamma^\ast \left( {it} \right) =  - \left| {\Gamma^\ast \left( {it} \right)} \right|\sin \left( {\int_0^{ + \infty } {\frac{{2ts}}{{\left( {t^2  + s^2 } \right)^2 }}Q\left( s \right)ds} } \right),
\]
whence
\begin{equation}\label{eq26}
\Re \Gamma^\ast \left( {it} \right) + \Im \Gamma^\ast \left( {it} \right) = \sqrt 2 \left| {\Gamma^\ast \left( {it} \right)} \right|\cos \left( {\frac{\pi }{4} + \int_0^{ + \infty } {\frac{{2ts}}{{\left( {t^2  + s^2 } \right)^2 }}Q\left( s \right)ds} } \right) .
\end{equation}
The following double inequality was proved in \cite{Nemes}
\[
0 \leq \int_0^{ + \infty } {\frac{{2ts}}{{\left( {t^2  + s^2 } \right)^2 }}Q\left( s \right)ds} \leq \frac{t}{2}\log \left( {\frac{{t^2 }}{{t^2  + 1}}} \right) + \frac{1}{2}\arctan \left( {\frac{1}{t}} \right) + \frac{1}{8}\frac{t}{{t^2  + 1}},
\]
for $t>0$. Elementary analysis shows that the function on the right-hand side is a monotonically decreasing function of $t$ and its limits at $t=+0$ and $t=+\infty$ are $\frac{\pi}{4}$ and $0$, respectively. Therefore the argument of the cosine in \eqref{eq26} is always between $\frac{\pi}{4}$ and $\frac{\pi}{2}$, which completes the proof.
\end{proof}

To derive our bounds we need some inequalities. For $r>0$ and $\left|\vartheta\right|\leq \pi$, it holds that
\begin{equation}\label{eq30}
\frac{1}{\left| {\left( {1 - e^{ - \left( {\frac{{3\pi }}{4} + \frac{\vartheta }{2}} \right)i} r} \right)\left( {1 - e^{\left( {\frac{{3\pi }}{4} - \frac{\vartheta }{2}} \right)i} r} \right)} \right|}  \le 1.
\end{equation}
Indeed,
\[
\left| {\left( {1 - e^{ - \left( {\frac{{3\pi }}{4} + \frac{\vartheta }{2}} \right)i} r} \right)\left( {1 - e^{\left( {\frac{{3\pi }}{4} - \frac{\vartheta }{2}} \right)i} r} \right)} \right|^2  = r^4  + 2\sqrt 2 r^3 \cos \left( {\frac{\vartheta }{2}} \right) + 2r^2 \left( {1 + \cos \vartheta } \right) + 2\sqrt 2 r\cos \left( {\frac{\vartheta }{2}} \right) + 1 \ge 1.
\]
By Lemma \ref{lemma1}, we have the following double inequalities for $- \Im \Gamma^\ast \left( {it} \right)$ and $\Re \Gamma^\ast \left( {it} \right) + \Im \Gamma^\ast \left( {it} \right)$:
\begin{equation}\label{eq31}
0\leq - \Im \Gamma^\ast \left( {it} \right) = \Re \Gamma^\ast \left( {it} \right) - \left( {\Re \Gamma^\ast \left( {it} \right) + \Im \Gamma^\ast \left( {it} \right)} \right) \le \Re \Gamma^\ast \left( {it} \right),
\end{equation}
\begin{equation}\label{eq32}
0\leq \Re \Gamma^\ast \left( {it} \right) + \Im \Gamma^\ast \left( {it} \right) = \Re \Gamma^\ast  \left( {it} \right) - \Im \Gamma^\ast \left( {it} \right) - 2\left( { - \Im \Gamma^\ast \left( {it} \right)} \right) \le \Re \Gamma^\ast  \left( {it} \right) - \Im \Gamma^\ast  \left( {it} \right),
\end{equation}
provided that $t>0$. Applying the inequalities \eqref{eq30} and \eqref{eq31} for \eqref{eq27} together with Lemma \ref{lemma1} and the representations \eqref{eq28} and \eqref{eq29}, we deduce
\[
\left| {R_{4N + 1} \left( z \right)} \right| \le \frac{{\left| {a_{4N + 1} } \right|}}{{\left| z \right|^{2N + \frac{1}{2}} }} + \frac{{\left| {a_{4N + 2} } \right|}}{{\left| z \right|^{2N + 1} }},
\]
for $\left|\arg z\right|\leq \pi$. Similarly, employing the inequalities \eqref{eq30} and \eqref{eq32} for \eqref{eq33} together with Lemma \ref{lemma1} and the representations \eqref{eq29} and \eqref{eq34}, we obtain
\[
\left| {R_{4N + 2} \left( z \right)} \right| \le \frac{{\left| {a_{4N + 2} } \right|}}{{\left| z \right|^{2N + 1} }} + \frac{{\left| {a_{4N + 3} } \right|}}{{\left| z \right|^{2N + \frac{3}{2}} }},
\]
for $\left|\arg z\right|\leq \pi$. It is easy to see that for the cases of $R_{4N} \left( z \right)$ and $R_{4N + 3} \left( z \right)$, we cannot obtain an upper bound involving just the absolute value of the first two omitted terms. Instead we use the expressions
\begin{equation}\label{eq35}
R_{4N} \left( z \right) = \frac{{a_{4N} }}{{z^{2N} }} + R_{4N + 1} \left( z \right),\; R_{4N + 3} \left( z \right) = \frac{{a_{4N + 3} }}{{z^{2N + \frac{3}{2}} }} + \frac{{a_{4N + 4} }}{{z^{2N + 2} }} + R_{4N + 5} \left( z \right)
\end{equation}
and the bounds we have just obtained, to derive that
\[
\left| {R_{4N} \left( z \right)} \right| \le \frac{{\left| {a_{4N} } \right|}}{{\left| z \right|^{2N} }} + \frac{{\left| {a_{4N + 1} } \right|}}{{\left| z \right|^{2N + \frac{1}{2}} }} + \frac{{\left| {a_{4N + 2} } \right|}}{{\left| z \right|^{2N + 1} }}
\]
and
\[
\left| {R_{4N + 3} \left( z \right)} \right| \le \frac{{\left| {a_{4N + 3} } \right|}}{{\left| z \right|^{2N + \frac{3}{2}} }} + \frac{{\left| {a_{4N + 4} } \right|}}{{\left| z \right|^{2N + 2} }} + \frac{{\left| {a_{4N + 5} } \right|}}{{\left| z \right|^{2N + \frac{5}{2}} }} + \frac{{\left| {a_{4N + 6} } \right|}}{{\left| z \right|^{2N + 3} }},
\]
for $\left|\arg z\right|\leq \pi$.

Consider now the case when $z>0$. Note that in this situation, we have
\begin{equation}\label{eq36}
0 < \frac{1}{{\left( {1 - e^{ - \frac{{3\pi i}}{4}} \left( {t/z} \right)^{\frac{1}{2}} } \right)\left( {1 - e^{\frac{{3\pi i}}{4}} \left( {t/z} \right)^{\frac{1}{2}} } \right)}} = \frac{1}{{t/z + \left( {2t/z} \right)^{\frac{1}{2}}  + 1}} < 1.
\end{equation}
Therefore, by the mean value theorem of integration we find from \eqref{eq27}, \eqref{eq33}, \eqref{eq28}--\eqref{eq34}, \eqref{eq31}, \eqref{eq32}, Lemma \ref{lemma1} and \eqref{eq35} that
\[
\left( { - 1} \right)^{N + 1} R_{4N} \left( z \right) = \frac{{\left| {a_{4N} } \right|}}{{z^{2N} }} + \Xi _1 \frac{{\left| {a_{4N + 1} } \right|}}{{z^{2N + \frac{1}{2}} }} - \Xi _2 \frac{{\left| {a_{4N + 2} } \right|}}{{z^{2N + 1} }},
\]
\[
\left( { - 1} \right)^{N + 1} R_{4N + 1} \left( z \right) = \Xi _1 \frac{{\left| {a_{4N + 1} } \right|}}{{z^{2N + \frac{1}{2}} }} - \Xi _2 \frac{{\left| {a_{4N + 2} } \right|}}{{z^{2N + 1} }},
\]
\[
\left( { - 1} \right)^N R_{4N + 2} \left( z \right) = \Xi _3 \frac{{\left| {a_{4N + 2} } \right|}}{{z^{2N + 1} }} + \Xi _4 \frac{{\left| {a_{4N + 3} } \right|}}{{z^{2N + \frac{3}{2}} }}
\]
and
\[
\left( { - 1} \right)^{N + 1} R_{4N + 3} \left( z \right) = \frac{{\left| {a_{4N + 3} } \right|}}{{z^{2N + \frac{3}{2}} }} - \frac{{\left| {a_{4N + 4} } \right|}}{{z^{2N + 2} }} - \Xi _5 \frac{{\left| {a_{4N + 5} } \right|}}{{z^{2N + \frac{5}{2}} }} + \Xi _6 \frac{{\left| {a_{4N + 6} } \right|}}{{z^{2N + 3} }},
\]
for $z>0$. Here $0<\Xi_i<1$ ($i=1,2,\ldots,6$) is a suitable number depending on $N$. From the integral formulas \eqref{eq37}, \eqref{eq27} and \eqref{eq33}, Lemma \ref{lemma1} and the inequality \eqref{eq36}, it is seen that $\left( { - 1} \right)^{N + 1} R_{4N} \left( z \right)$, $\left( { - 1} \right)^N R_{4N + 2} \left( z \right)$, $\left( { - 1} \right)^{N + 1} R_{4N + 3} \left( z \right)$ are all positive if $z$ is positive. Therefore, using the previous estimates, we obtain the following double inequalities
\[
\max\left(0,\frac{{\left| {a_{4N} } \right|}}{{z^{2N} }} - \frac{{\left| {a_{4N + 2} } \right|}}{{z^{2N + 1} }} \right)< \left( { - 1} \right)^{N + 1} R_{4N} \left( z \right) < \frac{{\left| {a_{4N} } \right|}}{{z^{2N} }} + \frac{{\left| {a_{4N + 1} } \right|}}{{z^{2N + \frac{1}{2}} }},
\]
\[
 - \frac{{\left| {a_{4N + 2} } \right|}}{{z^{2N + 1} }} < \left( { - 1} \right)^{N + 1} R_{4N + 1} \left( z \right) < \frac{{\left| {a_{4N + 1} } \right|}}{{z^{2N + \frac{1}{2}} }},
\]
\begin{equation}\label{eq65}
0 < \left( { - 1} \right)^N R_{4N + 2} \left( z \right) < \frac{{\left| {a_{4N + 2} } \right|}}{{z^{2N + 1} }} + \frac{{\left| {a_{4N + 3} } \right|}}{{z^{2N + \frac{3}{2}} }}
\end{equation}
and
\[
\max\left(0,\frac{{\left| {a_{4N + 3} } \right|}}{{z^{2N + \frac{3}{2}} }} - \frac{{\left| {a_{4N + 4} } \right|}}{{z^{2N + 2} }} - \frac{{\left| {a_{4N + 5} } \right|}}{{z^{2N + \frac{5}{2}} }}\right) < \left( { - 1} \right)^{N + 1} R_{4N + 3} \left( z \right) < \frac{{\left| {a_{4N + 3} } \right|}}{{z^{2N + \frac{3}{2}} }} - \frac{{\left| {a_{4N + 4} } \right|}}{{z^{2N + 2} }} + \frac{{\left| {a_{4N + 6} } \right|}}{{z^{2N + 3} }},
\]
for $z>0$.

Next, we consider the sector $\pi < \left|\arg z\right| <\frac{3\pi}{2}$. The integral formulas \eqref{eq37}--\eqref{eq38} can be simplified further to the forms
\begin{gather}\label{eq39}
\begin{split}
R_{4N} \left( z \right) = \; & \frac{{\left( { - 1} \right)^{N + 1} }}{{\pi z^{2N} }}\int_0^{ + \infty } {\frac{{t^{2N - 1} e^{ - 2\pi t} }}{{1 + \left( {t/z} \right)^2 }}} \left( { - \Im \Gamma^\ast  \left( {it} \right)} \right)dt + \frac{{\left( { - 1} \right)^{N + 1} }}{{\sqrt 2 \pi z^{2N + \frac{1}{2}} }}\int_0^{ + \infty } {\frac{{t^{2N - \frac{1}{2}} e^{ - 2\pi t} }}{{1 + \left( {t/z} \right)^2 }}} \left( {\Re \Gamma^\ast \left( {it} \right) + \Im \Gamma^\ast \left( {it} \right)} \right)dt
\\ & + \frac{{\left( { - 1} \right)^N }}{{\pi z^{2N + 1} }}\int_0^{ + \infty } {\frac{{t^{2N} e^{ - 2\pi t} }}{{1 + \left( {t/z} \right)^2 }}} \Re \Gamma^\ast \left( {it} \right)dt + \frac{{\left( { - 1} \right)^{N + 1} }}{{\sqrt 2 \pi z^{2N + \frac{3}{2}} }}\int_0^{ + \infty } {\frac{{t^{2N + \frac{1}{2}} e^{ - 2\pi t} }}{{1 + \left( {t/z} \right)^2 }}} \left( {\Re \Gamma^\ast \left( {it} \right) - \Im \Gamma^\ast \left( {it} \right)} \right)dt,
\end{split}
\end{gather}
\begin{gather}\label{eq40}
\begin{split}
R_{4N + 1} \left( z \right) = \; & \frac{{\left( { - 1} \right)^{N + 1} }}{{\sqrt 2 \pi z^{2N + \frac{1}{2}} }}\int_0^{ + \infty } {\frac{{t^{2N - \frac{1}{2}} e^{ - 2\pi t} }}{{1 + \left( {t/z} \right)^2 }}} \left( {\Re \Gamma^\ast \left( {it} \right) + \Im \Gamma^\ast \left( {it} \right)} \right)dt + \frac{{\left( { - 1} \right)^N }}{{\pi z^{2N + 1} }}\int_0^{ + \infty } {\frac{{t^{2N} e^{ - 2\pi t} }}{{1 + \left( {t/z} \right)^2 }}} \Re \Gamma^\ast  \left( {it} \right)dt
\\ & + \frac{{\left( { - 1} \right)^{N + 1} }}{{\sqrt 2 \pi z^{2N + \frac{3}{2}} }}\int_0^{ + \infty } {\frac{{t^{2N + \frac{1}{2}} e^{ - 2\pi t} }}{{1 + \left( {t/z} \right)^2 }}} \left( {\Re \Gamma^\ast \left( {it} \right) - \Im \Gamma^\ast \left( {it} \right)} \right)dt + \frac{{\left( { - 1} \right)^N }}{{\pi z^{2N + 2} }}\int_0^{ + \infty } {\frac{{t^{2N + 1} e^{ - 2\pi t} }}{{1 + \left( {t/z} \right)^2 }}} \left( { - \Im \Gamma^\ast \left( {it} \right)} \right)dt,
\end{split}
\end{gather}
\begin{gather}\label{eq41}
\begin{split}
R_{4N + 2} \left( z \right) = \; & \frac{{\left( { - 1} \right)^N }}{{\pi z^{2N + 1} }}\int_0^{ + \infty } {\frac{{t^{2N} e^{ - 2\pi t} }}{{1 + \left( {t/z} \right)^2 }}} \Re \Gamma^\ast \left( {it} \right)dt + \frac{{\left( { - 1} \right)^{N + 1} }}{{\sqrt 2 \pi z^{2N + \frac{3}{2}} }}\int_0^{ + \infty } {\frac{{t^{2N + \frac{1}{2}} e^{ - 2\pi t} }}{{1 + \left( {t/z} \right)^2 }}} \left( {\Re \Gamma^\ast \left( {it} \right) - \Im \Gamma^\ast \left( {it} \right)} \right)dt
\\ & + \frac{{\left( { - 1} \right)^N }}{{\pi z^{2N + 2} }}\int_0^{ + \infty } {\frac{{t^{2N + 1} e^{ - 2\pi t} }}{{1 + \left( {t/z} \right)^2 }}} \left( { - \Im \Gamma^\ast \left( {it} \right)} \right)dt + \frac{{\left( { - 1} \right)^N }}{{\sqrt 2 \pi z^{2N + \frac{5}{2}} }}\int_0^{ + \infty } {\frac{{t^{2N + \frac{3}{2}} e^{ - 2\pi t} }}{{1 + \left( {t/z} \right)^2 }}} \left( {\Re \Gamma^\ast \left( {it} \right) + \Im \Gamma^\ast \left( {it} \right)} \right)dt
\end{split}
\end{gather}
and
\begin{gather}\label{eq42}
\begin{split}
R_{4N + 3} \left( z \right) = \; & \frac{{\left( { - 1} \right)^{N + 1} }}{{\sqrt 2 \pi z^{2N + \frac{3}{2}} }}\int_0^{ + \infty } {\frac{{t^{2N + \frac{1}{2}} e^{ - 2\pi t} }}{{1 + \left( {t/z} \right)^2 }}} \left( {\Re \Gamma^\ast \left( {it} \right) - \Im \Gamma^\ast \left( {it} \right)} \right)dt + \frac{{\left( { - 1} \right)^N }}{{\pi z^{2N + 2} }}\int_0^{ + \infty } {\frac{{t^{2N + 1} e^{ - 2\pi t} }}{{1 + \left( {t/z} \right)^2 }}} \left( { - \Im \Gamma^\ast \left( {it} \right)} \right)dt
\\ & + \frac{{\left( { - 1} \right)^N }}{{\sqrt 2 \pi z^{2N + \frac{5}{2}} }}\int_0^{ + \infty } {\frac{{t^{2N + \frac{3}{2}} e^{ - 2\pi t} }}{{1 + \left( {t/z} \right)^2 }}} \left( {\Re \Gamma^\ast \left( {it} \right) + \Im \Gamma^\ast  \left( {it} \right)} \right)dt + \frac{{\left( { - 1} \right)^{N + 1} }}{{\pi z^{2N + 3} }}\int_0^{ + \infty } {\frac{{t^{2N + 2} e^{ - 2\pi t} }}{{1 + \left( {t/z} \right)^2 }}} \Re \Gamma^\ast \left( {it} \right)dt,
\end{split}
\end{gather}
for $\pi < \left|\arg z\right| <\frac{3\pi}{2}$. Actually these formulas are true in the wider range $\left|\arg z\right| <\frac{3\pi}{2}$, by taking limits when $\left|\arg z\right| = \frac{\pi}{2}$. Before we give the bounds for the sector $\pi < \left|\arg z\right| <\frac{3\pi}{2}$, we remark that $0 < 1/\left( {1 + \left( {t/z} \right)^2 } \right) < 1$ for $z,t>0$, and by the mean value theorem of integration, the expressions \eqref{eq28}--\eqref{eq34} and Lemma \ref{lemma1}, we obtain the estimate
\[
\left( { - 1} \right)^N R_{4N + 2} \left( z \right) < \frac{{\left| {a_{4N + 2} } \right|}}{{z^{2N + 1} }} + \frac{{\left| {a_{4N + 4} } \right|}}{{z^{2N + 2} }} + \frac{{\left| {a_{4N + 4} } \right|}}{{z^{2N + \frac{5}{2}} }},
\]
for $z>0$. For large $z$ this is a better bound than the previously found estimate \eqref{eq65}. It is elementary to show that
\[
\frac{1}{{\left| {1 + \left( {t/z} \right)^2 } \right|}} \le \begin{cases} \left|\csc \left(2\theta\right) \right| & \; \text{ if } \; \frac{5\pi}{4} < \left|\theta\right| <\frac{3\pi}{2}, \\ 1 & \; \text{ if } \; \pi < \left|\theta\right| \leq \frac{5\pi}{4}, \end{cases}
\]
whence, trivial estimation of \eqref{eq39}--\eqref{eq42} together with the expressions \eqref{eq28}--\eqref{eq34} and Lemma \ref{lemma1} yield the bound
\[
\left| {R_N \left( z \right)} \right| \le \left( {\frac{{\left| {a_N } \right|}}{{\left| z \right|^{\frac{N}{2}} }} + \frac{{\left| {a_{N + 1} } \right|}}{{\left| z \right|^{\frac{{N + 1}}{2}} }} + \frac{{\left| {a_{N + 2} } \right|}}{{\left| z \right|^{\frac{{N + 2}}{2}} }} + \frac{{\left| {a_{N + 3} } \right|}}{{\left| z \right|^{\frac{{N + 3}}{2}} }}} \right) \begin{cases} \left|\csc \left(2\theta\right) \right| & \; \text{ if } \; \frac{5\pi}{4} < \left|\theta\right| <\frac{3\pi}{2}, \\ 1 & \; \text{ if } \; \pi < \left|\theta\right| \leq \frac{5\pi}{4}, \end{cases}
\]
for $N\geq 2$. From the bounds we derived for the range $\left|\arg z\right|<\pi$, it is seen that the condition $\pi < \left|\theta\right| \leq \frac{5\pi}{4}$ in this estimate can be replaced by $\left|\theta\right| \leq \frac{5\pi}{4}$.

It is possible to derive error bounds that are useful in the sectors $\frac{5\pi}{4} < \left|\arg z\right|  < 2\pi $, by introducing the expressions
\[
\Re \Gamma^\ast \left( {it} \right) = \frac{{\Gamma^\ast \left( {it} \right) + \Gamma^\ast \left( { - it} \right)}}{2},\; \Im \Gamma^\ast \left( {it} \right) = \frac{{\Gamma^\ast \left( {it} \right) - \Gamma^\ast \left( { - it} \right)}}{{2i}}
\]
in \eqref{eq39}--\eqref{eq42}, and then rotating the path of integration by $0<\pm \varphi < \frac{\pi}{2}$ to obtain analytic continuation to the sectors $\frac{5\pi}{4} + \varphi  < \arg z  < \frac{3\pi }{2} + \varphi$ and $-\frac{3\pi}{2} + \varphi  < \arg z  < -\frac{5\pi }{4} + \varphi$, respectively. The analysis is similar to the one performed in the case of $b_n\left(\lambda\right)$, but at one point, a simple estimation for
\[
\Gamma^\ast \left(  - \frac{ite^{i\varphi } }{\cos \varphi} \right), \; t>0
\]
is required. To obtain simple bounds for $R_N\left(z\right)$, this estimate should not depend on $\varphi$. Nevertheless, because of the connection formulas \eqref{eq66} and \eqref{eq51}, the bounds we have obtained are sufficient.

\section{Asymptotics for the late coefficients}\label{section4} In this section, we investigate the asymptotic nature of the coefficients $b_n \left( \lambda  \right)$ and $a_n$ as $n\to +\infty$. First, we consider the coefficients $b_n \left( \lambda  \right)$. For our purposes, the most appropriate representation of these coefficients is the second integral formula in \eqref{eq10}. Upon replacing $1/\Gamma^\ast\left(t\right)$ by its representation \eqref{eq45} in this integral, we obtain
\begin{gather}\label{eq48}
\begin{split}
b_n \left( \lambda  \right) = \; & \frac{{\Gamma \left( {n + \frac{1}{2}} \right)\left( {\lambda  - 1} \right)^{2n + 1} }}{{\left( {2\pi \left( {\lambda  - \log \lambda  - 1} \right)} \right)^{\frac{1}{2}} \left( {\lambda  - \log \lambda  - 1} \right)^n }}\\ & \times \left( {\sum\limits_{k = 0}^{K - 1} {\left( { - 1} \right)^k \left( {\lambda  - \log \lambda  - 1} \right)^k a_{2k} \frac{{\Gamma \left( {n - k + \frac{1}{2}} \right)}}{{\Gamma \left( {n + \frac{1}{2}} \right)}}}  + A_K \left( {n,\lambda } \right)} \right),
\end{split}
\end{gather}
for any fixed $1 \le K \le n - 1$, provided that $n\geq 2$. The remainder term $A_K \left( {n,\lambda } \right)$ is given by the integral formula 
\[
A_K\left( {n,\lambda } \right) = \frac{{\left( {\lambda  - \log \lambda  - 1} \right)^{n + \frac{1}{2}} }}{{\Gamma \left( {n + \frac{1}{2}} \right)}}\int_0^{ + \infty } {t^{n - \frac{1}{2}} e^{ - t\left( {\lambda  - \log \lambda  - 1} \right)} \widetilde M_K \left( t \right)dt} .
\]
It was proved in \cite{Nemes} that
\[
\left| {\widetilde M_K \left( t \right)} \right| \le \frac{{\left| {a_{2K} } \right|}}{t^K} + \frac{{\left| {a_{2K + 2} } \right|}}{t^{K + 1}}
\]
for any $t>0$, whence
\begin{equation}\label{eq49}
\left| {A_K \left( {n,\lambda } \right)} \right| \le \left( {\lambda  - \log \lambda  - 1} \right)^K \left| {a_{2K} } \right|\frac{{\Gamma \left( {n - K + \frac{1}{2}} \right)}}{{\Gamma \left( {n + \frac{1}{2}} \right)}} + \left( {\lambda  - \log \lambda  - 1} \right)^{K + 1} \left| {a_{2K + 2} } \right|\frac{{\Gamma \left( {n - K - \frac{1}{2}} \right)}}{{\Gamma \left( {n + \frac{1}{2}} \right)}}.
\end{equation}
Expansions of type \eqref{eq48} are called inverse factorial series in the literature. Numerically, their character is similar to the character of asymptotic power series, because the consecutive gamma functions decrease asymptotically by a factor $n$.

For large $n$, the least value of the bound \eqref{eq49} occurs when
\[ K \approx \left( {n - \frac{1}{2}} \right)\frac{{2\pi }}{{\lambda  - \log \lambda  - 1 + 2\pi }}.
\]
With this choice of $K$, the error bound is
\[
\mathcal{O}\left({n^{ - \frac{3}{2}} \left( {\frac{{\lambda  - \log \lambda  - 1}}{{\lambda  - \log \lambda  - 1 + 2\pi }}} \right)^n }\right).
\]
This is the best accuracy we can achieve using the expansion \eqref{eq48}. Whence, the larger $\lambda$ is the larger $n$ has to be to get a reasonable approximation from \eqref{eq48}.

By extending the sum in \eqref{eq48} to infinity, we arrive at the formal series
\begin{align*}
\left( { - 1} \right)^{n + 1} \frac{{b_n \left( \lambda  \right)}}{{\lambda \left( {\lambda  - 1} \right)^{2n} }} \approx \; & \frac{{\left( { - 1} \right)^n \Gamma \left( {n + \frac{1}{2}} \right)\left( {\frac{1}{\lambda } - 1} \right)}}{{\left( {2\pi \left( {\lambda  - \log \lambda  - 1} \right)} \right)^{\frac{1}{2}} \left( {\lambda  - \log \lambda  - 1} \right)^n }}\left( 1 - \frac{{\left( {\lambda  - \log \lambda  - 1} \right)}}{{12\left( {n - \frac{1}{2}} \right)}} \right.\\ & \left. + \frac{{\left( {\lambda  - \log \lambda  - 1} \right)^2 }}{{288\left( {n - \frac{1}{2}} \right)\left( {n - \frac{3}{2}} \right)}} + \frac{{139\left( {\lambda  - \log \lambda  - 1} \right)^3 }}{{51840\left( {n - \frac{1}{2}} \right)\left( {n - \frac{3}{2}} \right)\left( {n - \frac{5}{2}} \right)}} +  \cdots  \right) .
\end{align*}
This is exactly Dingle's expansion for the late coefficients in the asymptotic series of $\Gamma\left(a,z\right)$ \cite[p. 161]{Dingle}. The mathematically rigorous form of Dingle's series is therefore the formula \eqref{eq48}.

Numerical examples illustrating the efficacy of the expansion \eqref{eq48}, truncated optimally, are given in Table \ref{table1}.

\begin{table*}[!ht]
\begin{center}
\begin{tabular}
[c]{ l r @{\,}c@{\,} l}\hline
 & \\ [-1ex]
 values of $\lambda$ and $K$ & $\lambda=8$, $K=55$ & & \\ [1ex]
 exact numerical value of $b_{100}\left(\lambda\right)$ & $0.7688481106808590674525145642326792894187$ & $\times$ & $10^{257}$ \\ [1ex]
 approximation \eqref{eq48} to $b_{100}\left(\lambda\right)$ & $0.7688481106808590674525145642326792893371$ & $\times$ & $10^{257}$  \\ [1ex]
 error & $0.816$ & $\times$ & $10^{220}$\\ [1ex] 
 error bound using \eqref{eq49} & $0.1582$ & $\times$ & $10^{221}$\\ [1ex] \hline
 & \\ [-1ex]
 values of $\lambda$ and $K$ & $\lambda=10$, $K=48$ & & \\ [1ex]
 exact numerical value of $b_{100}\left(\lambda\right)$ & $0.2328121261355049863437924678844708059197$ & $\times$ & $10^{266}$ \\ [1ex]
 approximation \eqref{eq48} to $b_{100}\left(\lambda\right)$ & $0.2328121261355049863437924678847160113033$ & $\times$ & $10^{266}$ \\ [1ex]
 error & $-0.2452053836$ & $\times$ & $10^{236}$\\ [1ex]
 error bound using \eqref{eq49} & $0.4965854763$ & $\times$ & $10^{236}$\\ [1ex] \hline
 & \\ [-1ex]
 values of $\lambda$ and $K$ &  $\lambda=15$, $K=35$ & & \\ [1ex]
 exact numerical value of $b_{100}\left(\lambda\right)$ & $0.1363711375012746147234623654087290250956$ & $\times$ & $10^{282}$ \\ [1ex]
 approximation \eqref{eq48} to $b_{100}\left(\lambda\right)$ & $0.1363711375012746147228683544438226390236$ & $\times$ & $10^{282}$ \\ [1ex]
 error & $0.5940109649063860720$ & $\times$ & $10^{261}$\\ [1ex] 
 error bound using \eqref{eq49} & $0.11702600900350747722$ & $\times$ & $10^{262}$\\ [-1ex]
 & \\\hline
\end{tabular}
\end{center}
\caption{Approximations for $b_{100}\left(\lambda\right)$ with various $\lambda$, using \eqref{eq48}.}
\label{table1}
\end{table*}

Let us now turn our attention to the coefficients $a_n$. The asymptotic behaviour of $a_{4n}$ and $a_{4n + 2}$ was investigated in the previous paper of the author \cite{Nemes}. Therefore, we consider only $a_{4n+1}$ and $a_{4n + 3}$. Throughout this section, we assume that the indices $4n+1$ and $4n+3$ are at least $2$. From \eqref{eq45}, we have
\[
\Gamma^\ast  \left( { \pm it} \right) = \sum\limits_{k = 0}^{K - 1} {\left( { \mp i} \right)^k \frac{{a_{2k} }}{{t^k }}}  + M_K \left( { \pm it} \right),
\]
for $K\geq 1$, where $M_K \left( { \pm it} \right)$ is defined via analytic continuation. Substituting this expression into the second representation in \eqref{eq18}, one finds
\begin{equation}\label{eq56}
a_{4n + 1}  = \frac{{\left( { - 1} \right)^n }}{{\sqrt 2 \pi \left( {2\pi } \right)^{2n + \frac{1}{2}} }}\sum\limits_{k = 0}^{K - 1} {\left( { - 1} \right)^{\left\lceil {k/2} \right\rceil  + 1} a_{2k} \left( {2\pi } \right)^k \Gamma \left( {2n - k + \frac{1}{2}} \right)}  + A_K \left( {4n + 1} \right)
\end{equation}
and
\begin{equation}\label{eq57}
a_{4n + 3}  = \frac{{\left( { - 1} \right)^n }}{{\sqrt 2 \pi \left( {2\pi } \right)^{2n + \frac{3}{2}} }}\sum\limits_{k = 0}^{K - 1} {\left( { - 1} \right)^{\left\lfloor {k/2} \right\rfloor  + 1} a_{2k} \left( {2\pi } \right)^k \Gamma \left( {2n - k + \frac{3}{2}} \right)}  + A_K \left( {4n + 3} \right),
\end{equation}
provided that $1\leq K \leq 2n$. The remainder terms $A_K \left( {4n + 1} \right)$ and $A_K \left( {4n + 3} \right)$ are given by the integral formulas
\[
A_K \left( {4n + 1} \right) = \frac{{\left( { - 1} \right)^n }}{{2\pi i}}\int_0^{ + \infty } {t^{2n - \frac{1}{2}} e^{ - 2\pi t} \left( {e^{ - \frac{{3\pi i}}{4}} M_K \left( {it} \right) - e^{\frac{{3\pi i}}{4}} M_K \left( { - it} \right)} \right)dt} 
\]
and
\[
A_K \left( {4n + 3} \right) = \frac{{\left( { - 1} \right)^n }}{{2\pi i}}\int_0^{ + \infty } {t^{2n + \frac{1}{2}} e^{ - 2\pi t} \left( {e^{ - \frac{\pi }{4}i} M_K \left( {it} \right) - e^{\frac{\pi }{4}i} M_K \left( { - it} \right)} \right)dt} ,
\]
respectively. It was shown in \cite{Nemes} that
\begin{equation}\label{eq50}
\left| {M_K \left( { \pm it} \right)} \right| \le \frac{{\left( {1 + \zeta \left( K \right)} \right)\Gamma \left( K \right)}}{{\left( {2\pi } \right)^{K + 1} t^K }}\frac{{2\sqrt K  + 1}}{2},
\end{equation}
for $K\geq 2$, from which we obtain the error bounds
\begin{equation}\label{eq52}
\left| {A_K \left( {4n + 1} \right)} \right| \le \frac{{2\sqrt K  + 1}}{{\left( {2\pi } \right)^{2n + \frac{5}{2}} }}\left( {1 + \zeta \left( K \right)} \right)\Gamma \left( K \right)\Gamma \left( {2n - K + \frac{1}{2}} \right)
\end{equation}
and
\begin{equation}\label{eq53}
\left| {A_K \left( {4n + 3} \right)} \right| \le \frac{{2\sqrt K  + 1}}{{\left( {2\pi } \right)^{2n + \frac{7}{2}} }}\left( {1 + \zeta \left( K \right)} \right)\Gamma \left( K \right)\Gamma \left( {2n - K + \frac{3}{2}} \right).
\end{equation}
Here $\zeta$ denotes Riemann's zeta function. An alternative set of approximations can be derived as follows. By \eqref{eq45} and \eqref{eq51}, we have
\[
\Gamma^\ast \left( { \pm it} \right) = \frac{1}{{1 - e^{ - 2\pi t} }}\frac{1}{{\Gamma^\ast \left( { \mp it} \right)}} = \frac{1}{{1 - e^{ - 2\pi t} }}\left( {\sum\limits_{k = 0}^{K - 1} {\left( { \mp i} \right)^k \frac{{a_{2k} }}{{t^k }}}  + \widetilde M_K \left( { \mp it} \right)} \right),
\]
for $K\geq 1$, where $\widetilde M_K \left( { \mp it} \right)$ is defined via analytic continuation. Substituting this expression into the second representation in \eqref{eq18}, one finds
\begin{equation}\label{eq58}
a_{4n + 1}  = \frac{{\left( { - 1} \right)^n }}{{\sqrt 2 \pi \left( {2\pi } \right)^{2n + \frac{1}{2}} }}\sum\limits_{k = 0}^{K - 1} {\left( { - 1} \right)^{\left\lceil {k/2} \right\rceil  + 1} a_{2k} \left( {2\pi } \right)^k \Gamma \left( {2n - k + \frac{1}{2}} \right)\zeta \left( {2n - k + \frac{1}{2}} \right)}  + \widetilde{A}_K \left( {4n + 1} \right)
\end{equation}
and
\begin{equation}\label{eq59}
a_{4n + 3}  = \frac{{\left( { - 1} \right)^n }}{{\sqrt 2 \pi \left( {2\pi } \right)^{2n + \frac{3}{2}} }}\sum\limits_{k = 0}^{K - 1} {\left( { - 1} \right)^{\left\lfloor {k/2} \right\rfloor  + 1} a_{2k} \left( {2\pi } \right)^k \Gamma \left( {2n - k + \frac{3}{2}} \right)\zeta \left( {2n - k + \frac{3}{2}} \right)}  + \widetilde{A}_K \left( {4n + 3} \right),
\end{equation}
provided that $1\leq K \leq 2n$. In deriving these expansions, we have made use of the known integral representation of the Riemann zeta function (see \cite[25.5.E1]{NIST}). The remainder terms $\widetilde{A}_K \left( {4n + 1} \right)$ and $\widetilde{A}_K \left( {4n + 3} \right)$ are given by the integral formulas
\[
\widetilde A_K \left( {4n + 1} \right) = \frac{{\left( { - 1} \right)^n }}{{2\pi i}}\int_0^{ + \infty } {\frac{{t^{2n - \frac{1}{2}} e^{ - 2\pi t} }}{{1 - e^{ - 2\pi t} }}\left( {e^{ - \frac{{3\pi i}}{4}} \widetilde M_K \left( { - it} \right) - e^{\frac{{3\pi i}}{4}} \widetilde M_K \left( {it} \right)} \right)dt} 
\]
and
\[
\widetilde A_K \left( {4n + 3} \right) = \frac{{\left( { - 1} \right)^n }}{{2\pi i}}\int_0^{ + \infty } {\frac{{t^{2n + \frac{1}{2}} e^{ - 2\pi t} }}{{1 - e^{ - 2\pi t} }}\left( {e^{ - \frac{\pi }{4}i} \widetilde M_K \left( { - it} \right) - e^{\frac{\pi }{4}i} \widetilde M_K \left( {it} \right)} \right)dt} ,
\]
respectively. It was proved in \cite{Nemes} that $\widetilde M_K \left( { \pm it} \right)$ satisfies the bound given on the right-hand side of \eqref{eq50}, whence we deduce the estimates
\begin{equation}\label{eq54}
\left| {\widetilde A_K \left( {4n + 1} \right)} \right| \le \frac{{2\sqrt K  + 1}}{{\left( {2\pi } \right)^{2n + \frac{5}{2}} }}\left( {1 + \zeta \left( K \right)} \right)\Gamma \left( K \right)\Gamma \left( {2n - K + \frac{1}{2}} \right)\zeta \left( {2n - K + \frac{1}{2}} \right)
\end{equation}
and
\begin{equation}\label{eq55}
\left| {\widetilde A_K \left( {4n + 3} \right)} \right| \le \frac{{2\sqrt K  + 1}}{{\left( {2\pi } \right)^{2n + \frac{7}{2}} }}\left( {1 + \zeta \left( K \right)} \right)\Gamma \left( K \right)\Gamma \left( {2n - K + \frac{3}{2}} \right)\zeta \left( {2n - K + \frac{3}{2}} \right),
\end{equation}
if $K\geq 2$. For large $n$, the least values of the bounds \eqref{eq52}, \eqref{eq53}, \eqref{eq54} and \eqref{eq55} occur when $K \approx n$. With this choice of $K$, the ratios of the error bounds to the leading terms are $\mathcal{O}\left( {4^{ - n} } \right)$. This is the best accuracy we can achieve using the expansions \eqref{eq56}, \eqref{eq57}, \eqref{eq58} and \eqref{eq59}. Numerical examples are given in Tables \ref{table2} and \ref{table3}.

\begin{table*}[!ht]
\begin{center}
\begin{tabular}
[c]{ l r @{\,}c@{\,} l}\hline
 & \\ [-1ex]
 values of $n$ and $K$ & $n=50$, $K=50$ & &  \\ [1ex]
 exact numerical value of $a_{4n+1}$ & $-0.12659638780775052710147996185410566$ & $\times$ & $10^{77}$ \\ [1ex]
 approximation \eqref{eq56} to $a_{4n+1}$ & $-0.12659638780775052710147996185407205$ & $\times$ & $10^{77}$ \\ [1ex]
 error & $-0.3361$ & $\times$ & $10^{46}$\\ [1ex]
 error bound using \eqref{eq52} & $0.12145$ & $\times$ & $10^{47}$\\ [1ex]
 approximation \eqref{eq58} to $a_{4n+1}$ & $-0.12659638780775052710147996185414229$ & $\times$ & $10^{77}$ \\ [1ex]
 error & $0.3663$ & $\times$ & $10^{46}$\\ [1ex]
 error bound using \eqref{eq54} & $0.12145$ & $\times$ & $10^{47}$\\ [1ex]
 Dingle's approximation \eqref{eq63} to $a_{4n+1}$ & $-0.12659638780775052710147996185410717$ & $\times$ & $10^{77}$ \\ [1ex]
 error & $-0.151$ & $\times$ & $10^{45}$\\ [-1ex]
 & \\\hline
\end{tabular}
\end{center}
\caption{Approximations for $a_{201}$, using optimal truncation.}
\label{table2}
\end{table*}

\begin{table*}[!ht]
\begin{center}
\begin{tabular}
[c]{ l r @{\,}c@{\,} l}\hline
 & \\ [-1ex]
 values of $n$ and $K$ & $n=50$, $K=50$ & &  \\ [1ex]
 exact numerical value of $a_{4n+3}$ & $-0.20462395914727659153115140698806033$ & $\times$ & $10^{78}$ \\ [1ex]
 approximation \eqref{eq57} to $a_{4n+3}$ & $-0.20462395914727659153115140698802838$ & $\times$ & $10^{78}$ \\ [1ex]
 error & $-0.3194$ & $\times$ & $10^{47}$\\ [1ex]
 error bound using \eqref{eq53} & $0.9761$ & $\times$ & $10^{47}$\\ [1ex]
 approximation \eqref{eq59} to $a_{4n+3}$ & $-0.20462395914727659153115140698808575$ & $\times$ & $10^{78}$ \\ [1ex]
 error & $0.2542$ & $\times$ & $10^{47}$\\ [1ex]
 error bound using \eqref{eq55} & $0.9761$ & $\times$ & $10^{47}$\\ [1ex]
 Dingle's approximation \eqref{eq64} to $a_{4n+3}$ & $-0.20462395914727659153115140698805707$ & $\times$ & $10^{78}$ \\ [1ex]
 error & $-0.326$ & $\times$ & $10^{46}$\\ [-1ex]
 & \\\hline
\end{tabular}
\end{center}
\caption{Approximations for $a_{203}$, using optimal truncation.}
\label{table3}
\end{table*}

The formal expansions that Dingle \cite[p. 164]{Dingle} derived for $a_{4n + 1}$ and $a_{4n + 3}$ slightly differ from ours. His results can be written, in our notation,
\begin{equation}\label{eq63}
a_{4n + 1}  \approx \frac{{\left( { - 1} \right)^n }}{{\sqrt 2 \pi \left( {2\pi } \right)^{2n + \frac{1}{2}} }}\sum\limits_{k = 0}^\infty  {\left( { - 1} \right)^{\left\lceil {k/2} \right\rceil  + 1} a_{2k} \left( {2\pi } \right)^k \Gamma \left( {2n - k + \frac{1}{2}} \right)\zeta \left( {2n - k + \frac{3}{2}} \right)} ,
\end{equation}
\begin{equation}\label{eq64}
a_{4n + 3}  \approx \frac{{\left( { - 1} \right)^n }}{{\sqrt 2 \pi \left( {2\pi } \right)^{2n + \frac{3}{2}} }}\sum\limits_{k = 0}^\infty  {\left( { - 1} \right)^{\left\lfloor {k/2} \right\rfloor  + 1} a_{2k} \left( {2\pi } \right)^k \Gamma \left( {2n - k + \frac{3}{2}} \right)\zeta \left( {2n - k + \frac{5}{2}} \right)} .
\end{equation}
These expansions resemble \eqref{eq58} and \eqref{eq59}, but the argument of the zeta functions are shifted by $1$. It is seen from Tables \ref{table2} and \ref{table3}, that using optimal truncation, Dingle's series are slightly better than those we obtained by rigorous methods. To conclude this section, we give a possible explanation of this interesting phenomenon. Our starting point is the exponentially improved asymptotic series, given by Paris and Wood \cite{Paris}
\begin{equation}\label{eq60}
\Gamma \left( { \pm it} \right) \sim \frac{1}{{\sqrt {1 - e^{ - 2\pi t} } }}\sum\limits_{k = 0}^\infty  {\left( { \mp i} \right)^k \frac{{a_{2k} }}{{t^k }}} 
\end{equation}
as $t\to +\infty$. Substituting this expansion into the second representation in \eqref{eq18}, one finds (formally) that
\begin{equation}\label{eq61}
a_{4n + 1}  \approx \frac{{\left( { - 1} \right)^n }}{{\sqrt 2 \pi \left( {2\pi } \right)^{2n + \frac{1}{2}} }}\sum\limits_{k = 0}^\infty  {\left( { - 1} \right)^{\left\lceil {k/2} \right\rceil  + 1} a_{2k} \left( {2\pi } \right)^k \Gamma \left( {2n - k + \frac{1}{2}} \right)\xi \left( {2n - k + \frac{1}{2}} \right)}
\end{equation}
and
\begin{equation}\label{eq62}
a_{4n + 3}  \approx \frac{{\left( { - 1} \right)^n }}{{\sqrt 2 \pi \left( {2\pi } \right)^{2n + \frac{3}{2}} }}\sum\limits_{k = 0}^\infty  {\left( { - 1} \right)^{\left\lfloor {k/2} \right\rfloor  + 1} a_{2k} \left( {2\pi } \right)^k \Gamma \left( {2n - k + \frac{3}{2}} \right)\xi \left( {2n - k + \frac{3}{2}} \right)} ,
\end{equation}
for large $n$. Here, the function $\xi \left( r \right)$ is given by the Dirichlet series
\[
\xi \left( r \right) = \frac{{\left( {2\pi } \right)^r }}{{\Gamma \left( r \right)}}\int_0^{ + \infty } {\frac{{t^{r - 1} e^{ - 2\pi t} }}{{\sqrt {1 - e^{ - 2\pi t} } }}dt}  = \sum\limits_{m = 0}^\infty  {\frac{{\left( {\frac{1}{2}} \right)_m }}{{m!\left( {m + 1} \right)^r }}}  = 1 + \frac{1}{2}\frac{1}{{2^r }} + \frac{3}{8}\frac{1}{{3^r }} + \frac{5}{{16}}\frac{1}{{4^r }} +  \cdots ,
\]
provided that $r > \frac{1}{2}$, and $\left(x\right)_m = \Gamma \left(x + m\right) /\Gamma \left(x\right)$ stands for the Pochhammer symbol. The formal expansions \eqref{eq61} and \eqref{eq62} can be turned into exact results by constructing error bounds for the series \eqref{eq60}, but we do not pursue the details here. Since \eqref{eq60} is a better approximation to $\Gamma \left( { \pm it} \right)$ than the previously used two approximations, we expect that, assuming optimal truncation, \eqref{eq61} and \eqref{eq62} give better estimates than \eqref{eq56} and \eqref{eq57} or \eqref{eq58} and \eqref{eq59}. When $2n-k$ is large
\[
\xi \left( {2n - k + \frac{1}{2}} \right) \approx 1 + \frac{1}{{2^{2n - k + \frac{3}{2}} }} + \frac{9}{8}\frac{1}{{3^{2n - k + \frac{3}{2}} }},\; \zeta \left( {2n - k + \frac{3}{2}} \right) \approx 1 + \frac{1}{{2^{2n - k + \frac{3}{2}} }} + \frac{1}{{3^{2n - k + \frac{3}{2}} }}
\]
and
\[
\xi \left( {2n - k + \frac{3}{2}} \right) \approx 1 + \frac{1}{2^{2n - k + \frac{5}{2}} } + \frac{9}{8}\frac{1}{3^{2n - k + \frac{5}{2}} },\; \zeta \left( {2n - k + \frac{5}{2}} \right) \approx 1 + \frac{1}{2^{2n - k + \frac{5}{2}} } + \frac{1}{3^{2n - k + \frac{5}{2}} }.
\]
Thus the approximate values produced by Dingle's formulas are very close to those given by our improved expansions \eqref{eq61} and \eqref{eq62}, which explains the superiority of his formulas over \eqref{eq56} and \eqref{eq57} or \eqref{eq58} and \eqref{eq59}.

\section{Exponentially improved asymptotic expansions}\label{section5} We shall find it convenient to express our exponentially improved expansions in terms of the (scaled) terminant function, which is defined in terms of the incomplete gamma function as
\[
\widehat T_p \left( w \right) = \frac{{e^{\pi ip} \Gamma \left( p \right)}}{{2\pi i}}\Gamma \left( {1 - p,w} \right) = \frac{e^{\pi ip} w^{1 - p} e^{ - w} }{2\pi i}\int_0^{ + \infty } {\frac{{t^{p - 1} e^{ - t} }}{w + t}dt} \; \text{ for } \; p>0 \; \text{ and } \; \left| \arg w \right| < \pi ,
\]
and by analytic continuation elsewhere. Olver \cite[equations (4.5) and (4.6)]{Olver4} showed that when $p \sim \left|w\right|$ and $w \to \infty$, we have
\begin{equation}\label{eq67}
\widehat T_p \left( w \right) = \begin{cases} \mathcal{O}\left( {e^{ - w - \left| w \right|} } \right) & \; \text{ if } \; \left| {\arg w} \right| \le \pi, \\ \mathcal{O}\left(1\right) & \; \text{ if } \; - 3\pi  < \arg w \le  - \pi. \end{cases}
\end{equation}
Concerning the smooth transition of the Stokes discontinuities, we will use the more precise asymptotic formulas
\begin{equation}\label{eq68}
\widehat T_p \left( w \right) = \frac{1}{2} + \frac{1}{2}\mathop{\text{erf}} \left( {c\left( \varphi  \right)\sqrt {\frac{1}{2}\left| w \right|} } \right) + \mathcal{O}\left( {\frac{{e^{ - \frac{1}{2}\left| w \right|c^2 \left( \varphi  \right)} }}{{\left| w \right|^{\frac{1}{2}} }}} \right)
\end{equation}
for $-\pi +\delta \leq \arg w \leq 3 \pi -\delta$, $0 < \delta  \le 2\pi$; and
\begin{equation}\label{eq69}
e^{ - 2\pi ip} \widehat T_p \left( w \right) =  - \frac{1}{2} + \frac{1}{2}\mathop{\text{erf}} \left( { - \overline {c\left( { - \varphi } \right)} \sqrt {\frac{1}{2}\left| w \right|} } \right) + \mathcal{O}\left( {\frac{{e^{ - \frac{1}{2}\left| w \right|\overline {c^2 \left( { - \varphi } \right)} } }}{{\left| w \right|^{\frac{1}{2}} }}} \right)
\end{equation}
for $- 3\pi  + \delta  \le \arg w \le \pi  - \delta$, $0 < \delta \le 2\pi$. Here $\varphi = \arg w$ and $\mathop{\text{erf}}$ denotes the error function. The quantity $c\left( \varphi  \right)$ is defined implicitly by the equation
\[
\frac{1}{2}c^2 \left( \varphi  \right) = 1 + i\left( {\varphi  - \pi } \right) - e^{i\left( {\varphi  - \pi } \right)},
\]
and corresponds to the branch of $c\left( \varphi  \right)$ which has the following expansion in the neighbourhood of $\varphi = \pi$:
\begin{equation}\label{eq70}
c\left( \varphi  \right) = \left( {\varphi  - \pi } \right) + \frac{i}{6}\left( {\varphi  - \pi } \right)^2  - \frac{1}{{36}}\left( {\varphi  - \pi } \right)^3  - \frac{i}{{270}}\left( {\varphi  - \pi } \right)^4  +  \cdots .
\end{equation}
For complete asymptotic expansions, see Olver \cite{Olver5}. We remark that Olver uses the different notation $F_p \left( w \right) = ie^{ - \pi ip} \widehat T_p \left( w \right)$ for the terminant function and the other branch of the function $c\left( \varphi  \right)$. For further properties of the terminant function, see, for example, Paris and Kaminski \cite[Chapter 6]{Paris2}.

\subsection{Proof of Theorem \ref{thm3}} First, we suppose that $\left|\arg a\right|<\pi$. Let $K \geq 2$ be a fixed integer. Substituting the second expression in \eqref{eq45} into \eqref{eq9} and using the definition of the terminant function we find that
\[
R_N \left( {a,\lambda } \right) = e^{a\left( {\lambda  - \log \lambda  - 1} \right)} \sqrt {\frac{{2\pi }}{a}} \sum\limits_{k = 0}^{K - 1} {\frac{{a_{2k} }}{{a^k }}\widehat T_{N - k + \frac{1}{2}} \left( {a\left( {\lambda  - \log \lambda  - 1} \right)} \right)}  + R_{N,K} \left( {a,\lambda } \right),
\]
with
\begin{gather}\label{eq71}
\begin{split}
R_{N,K} \left( {a,\lambda } \right) & = \frac{{\left( { - 1} \right)^N }}{{a^{N + 1} }}\frac{1}{{\sqrt {2\pi } }}\int_0^{ + \infty } {\frac{{t^{N - \frac{1}{2}} e^{ - t\left( {\lambda  - \log \lambda  - 1} \right)} }}{{1 + t/a}}\widetilde M_K \left( t \right)dt} 
\\ & = \left( { - 1} \right)^N \frac{{e^{ - i\theta \left( {N + 1} \right)} }}{{\sqrt {2\pi r} }}\int_0^{ + \infty } {\frac{{\tau ^{N - \frac{1}{2}} e^{ - r\tau \left( {\lambda  - \log \lambda  - 1} \right)} }}{{1 + \tau e^{ - i\theta } }}\widetilde M_K \left( {r\tau } \right)d\tau } ,
\end{split}
\end{gather}
under the assumption that $K\leq N$. Here we have taken $a=re^{i\theta}$. Using the integral formula \eqref{eq47}, $\widetilde M_K \left( {r\tau } \right)$ can be written as
\begin{align*}
\widetilde M_K \left( {r\tau } \right) = \; & \frac{1}{{2\pi i}}\frac{{\left( { - i} \right)^K }}{{\left( {r\tau } \right)^K }}\int_0^{ + \infty } {\frac{{s^{K - 1} e^{ - 2\pi s} \Gamma^\ast \left( {is} \right)}}{{1 + is/\left( {r\tau } \right)}}ds}  - \frac{1}{{2\pi i}}\frac{{i^K }}{{\left( {r\tau } \right)^K }}\int_0^{ + \infty } {\frac{{s^{K - 1} e^{ - 2\pi s} \Gamma^\ast \left( { - is} \right)}}{{1 - is/\left( {r\tau } \right)}}ds} 
\\ = \; & \frac{1}{{2\pi i}}\frac{{\left( { - i} \right)^K }}{{\left( {r\tau } \right)^K }}\left( {\int_0^{ + \infty } {\frac{{s^{K - 1} e^{ - 2\pi s} \Gamma^\ast  \left( {is} \right)}}{{1 + is/r}}ds}  + \left( {\tau  - 1} \right)\int_0^{ + \infty } {\frac{{s^{K - 1} e^{ - 2\pi s} \Gamma^\ast \left( {is} \right)}}{{\left( {1 - ir\tau /s} \right)\left( {1 + is/r} \right)}}ds} } \right)
\\ & - \frac{1}{{2\pi i}}\frac{{i^K }}{{\left( {r\tau } \right)^K }}\left( {\int_0^{ + \infty } {\frac{{s^{K - 1} e^{ - 2\pi s} \Gamma^\ast \left( { - is} \right)}}{{1 - is/r}}ds}  + \left( {\tau  - 1} \right)\int_0^{ + \infty } {\frac{{s^{K - 1} e^{ - 2\pi s} \Gamma^\ast \left( { - is} \right)}}{{\left( {1 + ir\tau /s} \right)\left( {1 - is/r} \right)}}ds} } \right).
\end{align*}
Noting that
\[
\left| {\frac{1}{{1 \pm is/r}}} \right|,\left| {\frac{1}{{\left( {1 \mp ir\tau /s} \right)\left( {1 \pm is/r} \right)}}} \right| \le 1
\]
for positive $r$, $\tau$ and $s$, substitution into \eqref{eq71} yields the upper bound
\begin{align*}
& \left| {R_{N,K} \left( {a,\lambda } \right)} \right| \le \frac{1}{{\sqrt {2\pi r} }}\left| {\int_0^{ + \infty } {\frac{{\tau ^{N-K - \frac{1}{2}} e^{ - r\tau \left( {\lambda  - \log \lambda  - 1} \right)} }}{{1 + \tau e^{ - i\theta } }}d\tau } } \right|\frac{1}{{2\pi r^K }}\int_0^{ + \infty } {s^{K - 1} e^{ - 2\pi s} \left| {\Gamma^\ast \left( {is} \right)} \right|ds} 
\\ & + \frac{1}{{\sqrt {2\pi r} }}\int_0^{ + \infty } {\tau ^{N -K- \frac{1}{2}} e^{ - r\tau \left( {\lambda  - \log \lambda  - 1} \right)} \left| {\frac{{\tau  - 1}}{{\tau  + e^{i\theta } }}} \right|d\tau } \frac{1}{{2\pi r^K }}\int_0^{ + \infty } {s^{K - 1} e^{ - 2\pi s} \left| {\Gamma^\ast \left( {is} \right)} \right|ds} 
\\ & + \frac{1}{{\sqrt {2\pi r} }}\left| {\int_0^{ + \infty } {\frac{{\tau ^{N -K- \frac{1}{2}} e^{ - r\tau \left( {\lambda  - \log \lambda  - 1} \right)} }}{{1 + \tau e^{ - i\theta } }}d\tau } } \right|\frac{1}{{2\pi }}\frac{1}{{2\pi r^K }}\int_0^{ + \infty } {s^{K - 1} e^{ - 2\pi s} \left| {\Gamma^\ast \left( { - is} \right)} \right|ds} 
\\ & + \frac{1}{{\sqrt {2\pi r} }}\int_0^{ + \infty } {\tau ^{N-K - \frac{1}{2}} e^{ - r\tau \left( {\lambda  - \log \lambda  - 1} \right)} \left| {\frac{{\tau  - 1}}{{\tau  + e^{i\theta } }}} \right|d\tau } \frac{1}{{2\pi r^K }}\int_0^{ + \infty } {s^{K - 1} e^{ - 2\pi s} \left| {\Gamma^\ast  \left( { - is} \right)} \right|ds} .
\end{align*}
Boyd \cite[equation (3.9)]{Boyd} showed that
\[
\frac{1}{{2\pi }}\int_0^{ + \infty } {s^{K - 1} e^{ - 2\pi s} \left| {\Gamma^\ast \left( { \pm is} \right)} \right|ds}  \le \frac{1}{2}\frac{{\left( {1 + \zeta \left( K \right)} \right)\Gamma \left( K \right)}}{{\left( {2\pi } \right)^{K + 1} }},
\]
and since $\left| {\left( {\tau - 1} \right)/\left( {\tau  + e^{i\theta } } \right)} \right| \le 1$, we find that
\begin{align*}
\left| {R_{N,K} \left( {a,\lambda } \right)} \right| \le \; & \left| {\sqrt {\frac{{2\pi }}{a}} e^{a\left( {\lambda  - \log \lambda  - 1} \right)} \widehat T_{N - K + \frac{1}{2}} \left( {a\left( {\lambda  - \log \lambda  - 1} \right)} \right)} \right|\frac{{\left( {1 + \zeta \left( K \right)} \right)\Gamma \left( K \right)}}{{\left( {2\pi } \right)^{K + 1} \left| a \right|^K }}
\\ & + \frac{{\left( {1 + \zeta \left( K \right)} \right)\Gamma \left( K \right)\Gamma \left( {N - K + \frac{1}{2}} \right)}}{{\left( {2\pi } \right)^{K + \frac{3}{2}} \left| a \right|^{N + 1} \left( {\lambda  - \log \lambda  - 1} \right)^{N - K + \frac{1}{2}} }}.
\end{align*}
By continuity, this bound holds in the closed sector $\left|\arg a\right|\leq \pi$. Assume that $N = \left| a \right| \left( {\lambda  - \log \lambda  - 1} \right)+ \rho$ where $\rho$ is bounded. Employing Stirling's formula, we find that
\[
\frac{{\left( {1 + \zeta \left( K \right)} \right)\Gamma \left( K \right)\Gamma \left( {N - K + \frac{1}{2}} \right)}}{{\left( {2\pi } \right)^{K + \frac{3}{2}} \left| a \right|^{N + 1} \left( {\lambda  - \log \lambda  - 1} \right)^{N - K + \frac{1}{2}} }} = \mathcal{O}_{K,\rho } \left( {\frac{{e^{ - \left| a \right|\left( {\lambda  - \log \lambda  - 1} \right)} }}{{\left( {\lambda  - \log \lambda  - 1} \right)^{\frac{1}{2}} \left| a \right|^{K + 1} }}} \right)
\]
as $a\to \infty$. Olver's estimation \eqref{eq67} shows that
\[
\left| {e^{a\left( {\lambda  - \log \lambda  - 1} \right)} \widehat T_{N - K + \frac{1}{2}} \left( {a\left( {\lambda  - \log \lambda  - 1} \right)} \right)} \right| = \mathcal{O}_{K,\rho } \left( {e^{ - \left| a \right|\left( {\lambda  - \log \lambda  - 1} \right)} } \right)
\]
for large $a$. Therefore, we obtain that
\begin{equation}\label{eq72}
R_{N,K} \left( {a,\lambda } \right) = \mathcal{O}_{K,\rho } \left( {\frac{{e^{ - \left| a \right|\left( {\lambda  - \log \lambda  - 1} \right)} }}{{\left| a \right|^{K + \frac{1}{2}} }}} \right)
\end{equation}
as $a\to \infty$ in the sector $\left|\arg a\right|\leq \pi$.

Consider now the sector $\pi < \arg a < 2\pi$. When $a$ enters this sector, the pole in the first integral in \eqref{eq71} crosses the integration path. According to the residue theorem, we obtain
\begin{gather}\label{eq73}
\begin{split}
R_{N,K} \left( {a,\lambda } \right) & = \sqrt {\frac{{2\pi }}{a}} e^{a\left( {\lambda  - \log \lambda  - 1} \right)} \widetilde M_K \left( {ae^{ - \pi i} } \right) + \frac{{\left( { - 1} \right)^N }}{{a^{N + 1} }}\frac{1}{{\sqrt {2\pi } }}\int_0^{ + \infty } {\frac{{t^{N - \frac{1}{2}} e^{ - t\left( {\lambda  - \log \lambda  - 1} \right)} }}{{1 + t/a}}\widetilde M_K \left( t \right)dt} 
\\ & = \sqrt {\frac{{2\pi }}{a}} e^{a\left( {\lambda  - \log \lambda  - 1} \right)} \widetilde M_K \left( {ae^{ - \pi i} } \right) + R_{N,K} \left( {ae^{ - 2\pi i} ,\lambda } \right)
\end{split}
\end{gather}
for $\pi <\arg a< 2\pi$. Let $\delta$ be a fixed small positive real number. If $\pi  \le \arg a \leq 2\pi  - \delta$, then $\widetilde M_K \left( {ae^{ - \pi i} } \right) = \mathcal{O}_{K,\delta} \left( {\left|a\right|^{ - K} } \right)$ as $a\to \infty$, whence by \eqref{eq72}, \eqref{eq73} and a continuity argument, we deduce
\begin{equation}\label{eq74}
R_{N,K} \left( {a,\lambda } \right) = \mathcal{O}_{K,\rho,\delta} \left( {\frac{{e^{\Re \left( a \right)\left( {\lambda  - \log \lambda  - 1} \right)} }}{{\left| a \right|^{K + \frac{1}{2}} }}} \right),
\end{equation}
as $a\to \infty$ in the closed sector $\pi  \le \arg a \leq 2\pi  - \delta$.

The proof of the estimate for the sector $-2\pi + \delta \leq \arg a \leq -\pi$ is completely analogous.

Consider finally the cases $K=0$ and $K=1$. We can write
\[
R_{N,0} \left( {a,\lambda } \right) = e^{a\left( {\lambda  - \log \lambda  - 1} \right)} \sqrt {\frac{{2\pi }}{a}} \sum\limits_{k = 0}^1 {\frac{{a_{2k} }}{{a^k }}\widehat T_{N - k + \frac{1}{2}} \left( {a\left( {\lambda  - \log \lambda  - 1} \right)} \right)}  + R_{N,2} \left( {a,\lambda } \right)
\]
and
\[
R_{N,1} \left( {a,\lambda } \right) = e^{a\left( {\lambda  - \log \lambda  - 1} \right)} \sqrt {\frac{{2\pi }}{a}} \frac{{a_2 }}{a}\widehat T_{N - \frac{1}{2}} \left( {a\left( {\lambda  - \log \lambda  - 1} \right)} \right) + R_{N,2} \left( {a,\lambda } \right).
\]
Employing the previously obtained bounds for $R_{N,2} \left( {a,\lambda } \right)$ and Olver's estimation \eqref{eq67} together with the connection formula for the terminant function \cite[p. 260]{Paris2}, shows that $R_{N,0} \left( {a,\lambda } \right)$ and $R_{N,1} \left( {a,\lambda } \right)$ indeed satisfy the order estimates prescribed in Theorem \ref{thm3}.

\subsection{Proof of Theorem \ref{thm4}} First, we suppose that $\left|\arg z\right|<\frac{\pi}{2}$. We write the remainder $R_2 \left( z \right)$ in the form
\begin{gather}\label{eq75}
\begin{split}
R_2 \left( z \right) = \; & \frac{1}{{2\pi z}}\int_0^{ + \infty } {\frac{{e^{ - 2\pi t} }}{{1 - it/z}}\Gamma^\ast \left( {it} \right)dt} + \frac{1}{{2\pi z}}\int_0^{ + \infty } {\frac{{e^{ - 2\pi t} }}{{1 + it/z}}\Gamma^\ast \left( { - it} \right)dt}  
\\ & - \frac{{1 + i}}{{\sqrt 2 }}\frac{1}{{2\pi z^{\frac{3}{2}} }}\int_0^{ + \infty } {\frac{{t^{\frac{1}{2}} e^{ - 2\pi t} }}{{1 - it/z}}\Gamma^\ast \left( {it} \right)dt} - \frac{{1 - i}}{{\sqrt 2 }}\frac{1}{{2\pi z^{\frac{3}{2}} }}\int_0^{ + \infty } {\frac{{t^{\frac{1}{2}} e^{ - 2\pi t} }}{{1 + it/z}}\Gamma^\ast \left( { - it} \right)dt} .
\end{split}
\end{gather}
From the second representation in \eqref{eq18}, one finds
\[
a_{2n}  = \frac{{i^{n - 1} }}{{2\pi }}\int_0^{ + \infty } {t^{n - 1} e^{ - 2\pi t} \Gamma^\ast \left( {it} \right)dt}  + \frac{{\left( { - i} \right)^{n - 1} }}{{2\pi }}\int_0^{ + \infty } {t^{n - 1} e^{ - 2\pi t} \Gamma^\ast \left( { - it} \right)dt} 
\]
and
\[
a_{2m + 1}  =  - \frac{{1 + i}}{{\sqrt 2 }}\frac{{i^{m - 1} }}{{2\pi }}\int_0^{ + \infty } {t^{m- \frac{1}{2}} e^{ - 2\pi t} \Gamma^\ast \left( {it} \right)dt}  - \frac{{1 - i}}{{\sqrt 2 }}\frac{{\left( { - i} \right)^{m - 1} }}{{2\pi }}\int_0^{ + \infty } {t^{m - \frac{1}{2}} e^{ - 2\pi t} \Gamma^\ast \left( { - it} \right)dt} ,
\]
for any $n,m\geq 1$. Let $N,M\geq 1$ be arbitrary integers. We apply the expansion \eqref{eq12} in \eqref{eq75} together with the above formulas for $a_{2n}$ and $a_{2m+1}$, to obtain
\[
R_2 \left( z \right) = \sum\limits_{n = 1}^{N - 1} {\frac{{a_{2n} }}{{z^n }}}  + \sum\limits_{m = 1}^{M - 1} {\frac{{a_{2m + 1} }}{{z^{m + \frac{1}{2}} }}}  + R_{N,M} \left( z \right)
\]
with
\begin{gather}\label{eq76}
\begin{split}
R_{N,M} \left( z \right) =\; & \frac{{i^{N - 1} }}{{2\pi z^N }}\int_0^{ + \infty } {\frac{{t^{N - 1} e^{ - 2\pi t} }}{{1 - it/z}}\Gamma^\ast \left( {it} \right)dt}  + \frac{{\left( { - i} \right)^{N - 1} }}{{2\pi z^N }}\int_0^{ + \infty } {\frac{{t^{N - 1} e^{ - 2\pi t} }}{{1 + it/z}}\Gamma^\ast \left( { - it} \right)dt} \\ & - \frac{{1 + i}}{{\sqrt 2 }}\frac{{i^{M - 1} }}{{2\pi z^{M + \frac{1}{2}} }}\int_0^{ + \infty } {\frac{{t^{M - \frac{1}{2}} e^{ - 2\pi t} }}{{1 - it/z}}\Gamma^\ast \left( {it} \right)dt}  - \frac{{1 - i}}{{\sqrt 2 }}\frac{{\left( { - i} \right)^{M - 1} }}{{2\pi z^{M + \frac{1}{2}} }}\int_0^{ + \infty } {\frac{{t^{M - \frac{1}{2}} e^{ - 2\pi t} }}{{1 + it/z}}\Gamma^\ast \left( { - it} \right)dt} .
\end{split}
\end{gather}
Let $K,L\geq 2$ be arbitrary fixed integers. We use \eqref{eq45} to expand the scaled gamma functions under the integrals in \eqref{eq76}, and apply the definition of the terminant function to deduce
\begin{align*}
R_{N,M} \left( z \right) =\; & e^{2\pi iz} \sum\limits_{k = 0}^{K - 1} {\frac{{a_{2k} }}{{z^k }}\widehat T_{N - k} \left( {2\pi iz} \right)}  - e^{ - 2\pi iz} \sum\limits_{k = 0}^{K - 1} {\frac{{a_{2k} }}{{z^k }}\widehat T_{N - k} \left( { - 2\pi iz} \right)} 
\\ & - e^{2\pi iz} \sum\limits_{\ell  = 0}^{L - 1} {\frac{{a_{2\ell } }}{{z^\ell  }}\widehat T_{M - \ell  + \frac{1}{2}} \left( {2\pi iz} \right)}  -e^{ - 2\pi iz} \sum\limits_{\ell  = 0}^{L - 1} {\frac{{a_{2\ell } }}{{z^\ell  }} \widehat T_{M - \ell  + \frac{1}{2}} \left( { - 2\pi iz} \right)} 
\\ & + R_{N,M,K,L} \left( z \right),
\end{align*}
with
\begin{gather}\label{eq77}
\begin{split}
& R_{N,M,K,L} \left( z \right) =  \frac{{i^{N - 1} }}{{2\pi z^N }}\int_0^{ + \infty } {\frac{{t^{N - 1} e^{ - 2\pi t} }}{{1 - it/z}}M_K \left( {it} \right)dt}  + \frac{{\left( { - i} \right)^{N - 1} }}{{2\pi z^N }}\int_0^{ + \infty } {\frac{{t^{N - 1} e^{ - 2\pi t} }}{{1 + it/z}}M_K \left( { - it} \right)dt} 
\\ & - \frac{{1 + i}}{{\sqrt 2 }}\frac{{i^{M - 1} }}{{2\pi z^{M + \frac{1}{2}} }}\int_0^{ + \infty } {\frac{{t^{M - \frac{1}{2}} e^{ - 2\pi t} }}{{1 - it/z}}M_L \left( {it} \right)dt}  - \frac{{1 - i}}{{\sqrt 2 }}\frac{{\left( { - i} \right)^{M - 1} }}{{2\pi z^{M + \frac{1}{2}} }}\int_0^{ + \infty } {\frac{{t^{M - \frac{1}{2}} e^{ - 2\pi t} }}{{1 + it/z}}M_L \left( { - it} \right)dt} 
\\ & =  \frac{{i^{N - 1} e^{ - i\theta N} }}{{2\pi }}\int_0^{ + \infty } {\frac{{\tau ^{N - 1} e^{ - 2\pi r\tau } }}{{1 - i\tau e^{ - i\theta } }}M_K \left( {ir\tau } \right)d\tau }  + \frac{{\left( { - i} \right)^{N - 1} e^{ - i\theta N} }}{{2\pi }}\int_0^{ + \infty } {\frac{{\tau ^{N - 1} e^{ - 2\pi r\tau } }}{{1 + i\tau e^{ - i\theta } }}M_K \left( { - ir\tau } \right)d\tau } 
\\ & - \frac{{1 + i}}{{\sqrt 2 }}\frac{{i^{M - 1} e^{ - i\theta \left( {M + \frac{1}{2}} \right)} }}{{2\pi }}\int_0^{ + \infty } {\frac{{\tau ^{M - \frac{1}{2}} e^{ - 2\pi r\tau } }}{{1 - i\tau e^{ - i\theta } }}M_L \left( {ir\tau } \right)d\tau }  - \frac{{1 - i}}{{\sqrt 2 }}\frac{{\left( { - i} \right)^{M - 1} e^{ - i\theta \left( {M + \frac{1}{2}} \right)} }}{{2\pi }}\int_0^{ + \infty } {\frac{{\tau ^{M - \frac{1}{2}} e^{ - 2\pi r\tau } }}{{1 + i\tau e^{ - i\theta } }}M_L \left( { - ir\tau } \right)d\tau } ,
\end{split}
\end{gather}
as long as $K < N$ and $L \leq  M $. Here we have taken $z=re^{i\theta}$. We consider the estimation of the first integral in \eqref{eq77}. In the paper \cite{Nemes}, it was shown that
\begin{align*}
M_K \left( {ir\tau } \right) = \; & \frac{1}{{2\pi i}}\frac{1}{{\left( {r\tau e^{ - i\varphi } } \right)^K }}\left( {\int_0^{ + \infty } {\frac{{s^{K - 1} e^{ - 2\pi se^{i\varphi } } \Gamma^\ast  \left( {ise^{i\varphi } } \right)}}{{1 - se^{i\varphi } /r}}ds}  + \left( {\tau  - 1} \right)\int_0^{ + \infty } {\frac{{s^{K - 1} e^{ - 2\pi se^{i\varphi } } \Gamma^\ast \left( {ise^{i\varphi } } \right)}}{{\left( {1 - r\tau /se^{i\varphi } } \right)\left( {1 - se^{i\varphi } /r} \right)}}ds} } \right)
\\ & - \frac{1}{{2\pi i}}\frac{1}{{\left( { - r\tau } \right)^K }}\left( {\int_0^{ + \infty } {\frac{{s^{K - 1} e^{ - 2\pi s} \Gamma^\ast \left( { - is} \right)}}{{1 + s/r}}ds}  + \left( {\tau  - 1} \right)\int_0^{ + \infty } {\frac{{s^{K - 1} e^{ - 2\pi s} \Gamma^\ast \left( { - is} \right)}}{{\left( {1 + r\tau /s} \right)\left( {1 + s/r} \right)}}ds} } \right),
\end{align*}
with an arbitrary $0<\varphi<\frac{\pi}{2}$. Substitution into the first integral in \eqref{eq77} and trivial estimation yield
\begin{align*}
& \left| {\frac{{i^{N - 1} e^{ - i\theta N} }}{{2\pi }}\int_0^{ + \infty } {\frac{{\tau ^{N - 1} e^{ - 2\pi r\tau } }}{{1 - i\tau e^{ - i\theta } }}M_K \left( {ir\tau } \right)d\tau } } \right| \le \\ & \frac{1}{{2\pi r^K }}\left| {\int_0^{ + \infty } {\frac{{\tau ^{N - K - 1} e^{ - 2\pi r\tau } }}{{1 - i\tau e^{ - i\theta } }}d\tau } } \right|\frac{1}{{2\pi }}\int_0^{ + \infty } {\frac{{s^{K - 1} e^{ - 2\pi s\cos \varphi } \left| {\Gamma^\ast \left( {ise^{i\varphi } } \right)} \right|}}{{\left| {1 - se^{i\varphi } /r} \right|}}ds} 
\\ &  + \frac{1}{{2\pi r^K }}\int_0^{ + \infty } {\tau ^{N - K - 1} e^{ - 2\pi r\tau } \left| {\frac{{\tau  - 1}}{{\tau  + ie^{i\theta } }}} \right|\frac{1}{{2\pi }}\int_0^{ + \infty } {\frac{{s^{K - 1} e^{ - 2\pi s\cos \varphi } \left| {\Gamma^\ast  \left( {ise^{i\varphi } } \right)} \right|}}{{\left| {\left( {1 - r\tau /se^{i\varphi } } \right)\left( {1 - se^{i\varphi } /r} \right)} \right|}}ds} d\tau }
\\ & + \frac{1}{{2\pi r^K }}\left| {\int_0^{ + \infty } {\frac{{\tau ^{N -K - 1} e^{ - 2\pi r\tau } }}{{1 - i\tau e^{ - i\theta } }}d\tau } } \right|\frac{1}{{2\pi }}\int_0^{ + \infty } {\frac{{s^{K - 1} e^{ - 2\pi s} \left| {\Gamma^\ast \left( { - is} \right)} \right|}}{{1 + s/r}}ds} 
\\ &  + \frac{1}{{2\pi r^K }}\int_0^{ + \infty } {\tau ^{N - K - 1} e^{ - 2\pi r\tau } \left| {\frac{{\tau  - 1}}{{\tau  + ie^{i\theta } }}} \right|\frac{1}{{2\pi }}\int_0^{ + \infty } {\frac{{s^{K - 1} e^{ - 2\pi s} \left| {\Gamma^\ast  \left( { - is} \right)} \right|}}{{\left( {1 + r\tau /s} \right)\left( {1 + s/r} \right)}}ds} d\tau } .
\end{align*}
Noting that
\[
\left| {\frac{{\tau  - 1}}{{\tau  + ie^{i\theta } }}} \right| \le 1,\; \frac{1}{{1 + s/r}} \le 1,\; \frac{1}{{\left( {1 + r\tau /s} \right)\left( {1 + s/r} \right)}} \le 1
\]
and
\[
\frac{1}{{\left| {1 - se^{i\varphi } /r} \right|}} \le \csc \varphi ,\; \frac{1}{{\left| {\left( {1 - r\tau /se^{i\varphi } } \right)\left( {1 - se^{i\varphi } /r} \right)} \right|}} \le \csc ^2 \varphi 
\]
for any positive $r$, $s$ and $\tau$, we deduce the upper bound
\begin{align*}
& \left| {\frac{{i^{N - 1} e^{ - i\theta N} }}{{2\pi }}\int_0^{ + \infty } {\frac{{\tau ^{N - 1} e^{ - 2\pi r\tau } }}{{1 - i\tau e^{ - i\theta } }}M_K \left( {ir\tau } \right)d\tau } } \right| \le 
\\ & \frac{1}{{2\pi r^K }}\left| {\int_0^{ + \infty } {\frac{{\tau ^{N - K- 1} e^{ - 2\pi r\tau } }}{{1 - i\tau e^{ - i\theta } }}d\tau } } \right|\frac{{\csc \varphi }}{{2\pi \cos ^K \varphi }}\int_0^{ + \infty } {t^{K - 1} e^{ - 2\pi t} \left| {\Gamma^\ast \left( {\frac{{ite^{i\varphi } }}{{\cos \varphi }}} \right)} \right|dt} 
\\ & + \frac{1}{{2\pi r^K }}\int_0^{ + \infty } {\tau ^{N - K - 1} e^{ - 2\pi r\tau } d\tau } \frac{{\csc ^2 \varphi }}{{2\pi \cos ^K \varphi }}\int_0^{ + \infty } {t^{K - 1} e^{ - 2\pi t} \left| {\Gamma^\ast \left( {\frac{{ite^{i\varphi } }}{{\cos \varphi }}} \right)} \right|dt} 
\\ & + \frac{1}{{2\pi r^K }}\left| {\int_0^{ + \infty } {\frac{{\tau ^{N - K - 1} e^{ - 2\pi r\tau } }}{{1 - i\tau e^{ - i\theta } }}d\tau } } \right|\frac{1}{{2\pi }}\int_0^{ + \infty } {s^{K - 1} e^{ - 2\pi s} \left| {\Gamma^\ast  \left( { - is} \right)} \right|ds} 
\\ & + \frac{1}{{2\pi r^K }}\int_0^{ + \infty } {\tau ^{N -K - 1} e^{ - 2\pi r\tau } d\tau } \frac{1}{{2\pi }}\int_0^{ + \infty } {s^{K - 1} e^{ - 2\pi s} \left| {\Gamma^\ast \left( { - is} \right)} \right|ds} .
\end{align*}
It was proved in \cite{Nemes} that with the choice $\varphi  = \arctan \left( {K^{ - \frac{1}{2}} } \right)$, we have $\csc \varphi \cos ^{ - K} \varphi  \le 2\sqrt K$ and $\csc ^2 \varphi \cos ^{ - K} \varphi  \le 3K$. It was also shown that
\[
\frac{1}{{2\pi }}\int_0^{ + \infty } {t^{K - 1} e^{ - 2\pi t} \left| {\Gamma^\ast \left( {\frac{{ite^{i\varphi } }}{{\cos \varphi }}} \right)} \right|dt} ,\; \frac{1}{{2\pi }}\int_0^{ + \infty } {s^{K - 1} e^{ - 2\pi s} \left| {\Gamma^\ast \left( { - is} \right)} \right|ds}  \le \frac{{\zeta \left( K \right)\Gamma \left( K \right)}}{{\left( {2\pi } \right)^{K + 1} }},
\]
whence we obtain the estimate
\begin{align*}
\left| {\frac{{i^{N - 1} e^{ - i\theta N} }}{{2\pi }}\int_0^{ + \infty } {\frac{{\tau ^{N - 1} e^{ - 2\pi r\tau } }}{{1 - i\tau e^{ - i\theta } }}M_K \left( {ir\tau } \right)d\tau } } \right| \le \; & \left( {2\sqrt K  + 1} \right)\left| {e^{2\pi iz} \widehat T_{N-K}\left( {2\pi iz} \right)} \right|\frac{{\zeta \left( K \right)\Gamma \left( K \right)}}{{\left( {2\pi } \right)^{K + 1} \left| z \right|^K }} \\ & + \left( {3K + 1} \right)\frac{{\Gamma \left( {N - K} \right)\zeta \left( K \right)\Gamma \left( K \right)}}{{\left( {2\pi } \right)^{N + 2} \left| z \right|^N }}.
\end{align*}
Similarly, we have the following upper bounds for the other three integrals in \eqref{eq77}:
\begin{align*}
\left| {\frac{{\left( { - i} \right)^{N - 1} e^{ - i\theta N} }}{{2\pi }}\int_0^{ + \infty } {\frac{{\tau ^{N - 1} e^{ - 2\pi r\tau } }}{{1 + i\tau e^{ - i\theta } }}M_K \left( { - ir\tau } \right)d\tau } } \right| \le \; & \left( {2\sqrt K  + 1} \right)\left| {e^{ - 2\pi iz} \widehat T_{N-K}\left( { - 2\pi iz} \right)} \right|\frac{{\zeta \left( K \right)\Gamma \left( K \right)}}{{\left( {2\pi } \right)^{K + 1} \left| z \right|^K }}
\\ & + \left( {3K + 1} \right)\frac{{\Gamma \left( {N - K} \right)\zeta \left( K \right)\Gamma \left( K \right)}}{{\left( {2\pi } \right)^{N + 2} \left| z \right|^N }},
\end{align*}
\begin{align*}
& \left| { - \frac{{1 + i}}{{\sqrt 2 }}\frac{{i^{M - 1} e^{ - i\theta \left( {M + \frac{1}{2}} \right)} }}{{2\pi }}\int_0^{ + \infty } {\frac{{\tau ^{M - \frac{1}{2}} e^{ - 2\pi r\tau } }}{{1 - i\tau e^{ - i\theta } }}M_L \left( {ir\tau } \right)d\tau } } \right| \le \\ & \left( {2\sqrt L  + 1} \right)\left| {e^{2\pi iz} \widehat T_{M - L + \frac{1}{2}} \left( {2\pi iz} \right)} \right|\frac{{\zeta \left( L \right)\Gamma \left( L \right)}}{{\left( {2\pi } \right)^{L + 1} \left| z \right|^L }} + \left( {3L + 1} \right)\frac{{\Gamma \left( {M - L + \frac{1}{2}} \right)\zeta \left( L \right)\Gamma \left( L \right)}}{{\left( {2\pi } \right)^{M + \frac{5}{2}} \left| z \right|^{M + \frac{1}{2}} }}
\end{align*}
and
\begin{align*}
& \left| { - \frac{{1 - i}}{{\sqrt 2 }}\frac{{\left( { - i} \right)^{M - 1} e^{ - i\theta \left( {M + \frac{1}{2}} \right)} }}{{2\pi }}\int_0^{ + \infty } {\frac{{\tau ^{M - \frac{1}{2}} e^{ - 2\pi r\tau } }}{{1 + i\tau e^{ - i\theta } }}M_L \left( { - ir\tau } \right)d\tau } } \right| \le  \\ & \left( {2\sqrt L  + 1} \right)\left| {e^{ - 2\pi iz} \widehat T_{M - L + \frac{1}{2}} \left( { - 2\pi iz} \right)} \right|\frac{{\zeta \left( L \right)\Gamma \left( L \right)}}{{\left( {2\pi } \right)^{L + 1} \left| z \right|^L }} + \left( {3L + 1} \right)\frac{{\Gamma \left( {M - L + \frac{1}{2}} \right)\zeta \left( L \right)\Gamma \left( L \right)}}{{\left( {2\pi } \right)^{M + \frac{5}{2}} \left| z \right|^{M + \frac{1}{2}} }}.
\end{align*}
Thus, we conclude that
\begin{align*}
\left| {R_{N,M,K,L} \left( z \right)} \right| \le \; & \left( {2\sqrt K  + 1} \right)\left( {\left| {e^{2\pi iz} \widehat T_{N-K}\left( {2\pi iz} \right)} \right| + \left| {e^{ - 2\pi iz} \widehat T_{N-K}\left( { - 2\pi iz} \right)} \right|} \right)\frac{{\zeta \left( K \right)\Gamma \left( K \right)}}{{\left( {2\pi } \right)^{K + 1} \left| z \right|^K }}
\\ & + \left( {6K + 2} \right)\frac{{\Gamma \left( {N - K} \right)\zeta \left( K \right)\Gamma \left( K \right)}}{{\left( {2\pi } \right)^{N + 2} \left| z \right|^N }}
\\ & + \left( {2\sqrt L  + 1} \right)\left( {\left| {e^{2\pi iz} \widehat T_{M - L + \frac{1}{2}} \left( {2\pi iz} \right)} \right| + \left| {e^{ - 2\pi iz} \widehat T_{M - L + \frac{1}{2}} \left( { - 2\pi iz} \right)} \right|} \right)\frac{{\zeta \left( L \right)\Gamma \left( L \right)}}{{\left( {2\pi } \right)^{L + 1} \left| z \right|^L }}
\\ & + \left( {6L + 2} \right)\frac{{\Gamma \left( {M - L + \frac{1}{2}} \right)\zeta \left( L \right)\Gamma \left( L \right)}}{{\left( {2\pi } \right)^{M + \frac{5}{2}} \left| z \right|^{M + \frac{1}{2}} }}.
\end{align*}
By continuity, this bound holds in the closed sector $\left|\arg z\right| \leq \frac{\pi}{2}$. Suppose that $N=2\pi\left|z\right|+\rho$ and $M=2\pi\left|z\right|+\sigma$, where $\rho$ and $\sigma$ are bounded quantities. Applying Stirling's formula, we find that
\[
\left( {6K + 2} \right)\frac{{\Gamma \left( {N - K} \right)\zeta \left( K \right)\Gamma \left( K \right)}}{{\left( {2\pi } \right)^{N + 2} \left| z \right|^N }} = \mathcal{O}_{K,\rho } \left( {\frac{{e^{ - 2\pi \left| z \right|} }}{{\left| z \right|^{K + \frac{1}{2}} }}} \right)
\]
and
\[
\left( {6L + 2} \right)\frac{{\Gamma \left( {M - L + \frac{1}{2}} \right)\zeta \left( L \right)\Gamma \left( L \right)}}{{\left( {2\pi } \right)^{M + \frac{5}{2}} \left| z \right|^{M + \frac{1}{2}} }} = \mathcal{O}_{L,\sigma } \left( {\frac{{e^{ - 2\pi \left| z \right|} }}{{\left| z \right|^{L + \frac{1}{2}} }}} \right),
\]
as $z\to \infty$. Olver's estimation \eqref{eq67} shows that
\[
\left( {2\sqrt K  + 1} \right)\left( {\left| {e^{2\pi iz} \widehat T_{N - K} \left( {2\pi iz} \right)} \right| + \left| {e^{ - 2\pi iz} \widehat T_{N - K} \left( { - 2\pi iz} \right)} \right|} \right)\frac{{\zeta \left( K \right)\Gamma \left( K \right)}}{{\left( {2\pi } \right)^{K + 1} \left| z \right|^K }} = \mathcal{O}_{K,\rho } \left( {\frac{{e^{ - 2\pi \left| z \right|} }}{{\left| z \right|^K }}} \right)
\]
and
\[
\left( {2\sqrt L  + 1} \right)\left( {\left| {e^{2\pi iz} \widehat T_{M - L + \frac{1}{2}} \left( {2\pi iz} \right)} \right| + \left| {e^{ - 2\pi iz} \widehat T_{M - L + \frac{1}{2}} \left( { - 2\pi iz} \right)} \right|} \right)\frac{{\zeta \left( L \right)\Gamma \left( L \right)}}{{\left( {2\pi } \right)^{L + 1} \left| z \right|^L }} = \mathcal{O}_{L,\sigma } \left( {\frac{{e^{ - 2\pi \left| z \right|} }}{{\left| z \right|^L }}} \right),
\]
for large $z$.  Therefore, we have that
\begin{equation}\label{eq78}
R_{N,M,K,L} \left( z \right) = \mathcal{O}_{K,\rho } \left( {\frac{{e^{ - 2\pi \left| z \right|} }}{{\left| z \right|^K }}} \right) + \mathcal{O}_{L,\sigma } \left( {\frac{{e^{ - 2\pi \left| z \right|} }}{{\left| z \right|^L }}} \right)
\end{equation}
as $z\to \infty$ in the sector $\left|\arg z\right| \leq \frac{\pi}{2}$.

Consider now the sector $\frac{\pi}{2}<\arg z< \frac{3\pi}{2}$. When $z$ enters the sector $\frac{\pi}{2}<\arg z< \frac{3\pi}{2}$, the poles in the first and third integrals in \eqref{eq77} cross the integration path. According to the residue theorem, we obtain
\begin{gather}\label{eq79}
\begin{split}
& R_{N,M,K,L} \left( z \right) = e^{2\pi iz} M_K \left( z \right) - e^{2\pi iz} M_L \left( z \right)
\\ & + \frac{{i^{N - 1} }}{{2\pi z^N }}\int_0^{ + \infty } {\frac{{t^{N - 1} e^{ - 2\pi t} }}{{1 - it/z}}M_K \left( {it} \right)dt}  + \frac{{\left( { - i} \right)^{N - 1} }}{{2\pi z^N }}\int_0^{ + \infty } {\frac{{t^{N - 1} e^{ - 2\pi t} }}{{1 + it/z}}M_K \left( { - it} \right)dt} 
\\ & - \frac{{1 + i}}{{\sqrt 2 }}\frac{{i^{M - 1} }}{{2\pi z^{M + \frac{1}{2}} }}\int_0^{ + \infty } {\frac{{t^{M - \frac{1}{2}} e^{ - 2\pi t} }}{{1 - it/z}}M_L \left( {it} \right)dt}  - \frac{{1 - i}}{{\sqrt 2 }}\frac{{\left( { - i} \right)^{M - 1} }}{{2\pi z^{M + \frac{1}{2}} }}\int_0^{ + \infty } {\frac{{t^{M - \frac{1}{2}} e^{ - 2\pi t} }}{{1 + it/z}}M_L \left( { - it} \right)dt} 
\end{split}
\end{gather}
when $\frac{\pi}{2}<\arg z< \frac{3\pi}{2}$. It is easy to see that the sum of the four integrals has the order of magnitude given in the right-hand side of \eqref{eq78}. It follows that when $K=L$, the bound \eqref{eq78} remains valid in the wider sector $-\frac{\pi}{2}\leq \arg z \leq  \frac{3\pi}{2}$. Otherwise, by the connection formula \eqref{eq51}, we have
\[
e^{2\pi iz} M_K \left( z \right) - e^{2\pi iz} M_L \left( z \right) = e^{2\pi iz} \widetilde M_K \left( {ze^{ - \pi i} } \right) - e^{2\pi iz} \widetilde M_L \left( {ze^{ - \pi i} } \right).
\]
If $\frac{\pi}{2}\leq \arg z \leq \frac{3\pi}{2}$, then $\widetilde M_K \left( ze^{ - \pi i} \right) = \mathcal{O}_K\left( {\left|z\right|^{ - K} } \right)$ as $z\to \infty$, whence by continuity
\begin{equation}\label{eq81}
R_{N,M,K,L} \left( z \right) = \mathcal{O}_{K,\rho } \left( {\frac{{e^{ - 2\pi \Im \left( z \right)} }}{{\left|z\right|^K }}} \right) + \mathcal{O}_{L,\sigma } \left( {\frac{{e^{ - 2\pi \Im \left( z \right)} }}{{\left|z\right|^L }}} \right),
\end{equation}
as $z\to \infty$ in the closed sector $\frac{\pi}{2}\leq \arg z \leq \frac{3\pi}{2}$.

Similarly, if $K=L$, the bound \eqref{eq78} remains valid in the wider sector $-\frac{3\pi}{2}\leq \arg z \leq  \frac{\pi}{2}$; and by the foregoing argument, it is true in the larger sector $-\frac{3\pi}{2}\leq \arg z \leq  \frac{3\pi}{2}$. Otherwise, we have
\begin{equation}\label{eq80}
R_{N,M,K,L} \left( z \right) = \mathcal{O}_{K,\rho } \left( {\frac{{e^{ 2\pi \Im \left( z \right)} }}{{\left|z\right|^K }}} \right) + \mathcal{O}_{L,\sigma } \left( {\frac{{e^{ 2\pi \Im \left( z \right)} }}{{\left|z\right|^L }}} \right),
\end{equation}
for large $z$ with $-\frac{3\pi}{2}\leq \arg z \leq -\frac{\pi}{2}$.

Consider now the sector $\frac{3\pi}{2}<\arg z< \frac{5\pi}{2}$. Rotation of the path of integration in the second and the fourth integrals in \eqref{eq79} and application of the residue theorem yields
\begin{gather}\label{eq87}
\begin{split}
 & R_{N,M,K,L} \left( z \right) =  - e^{ - 2\pi iz} M_K \left( { ze^{ - 2\pi i}} \right) - e^{ - 2\pi iz} M_L \left( { ze^{ - 2\pi i}} \right) + e^{2\pi iz} M_K \left( z \right) - e^{2\pi iz} M_L \left( z \right)
\\ & + \frac{{i^{N - 1} }}{{2\pi z^N }}\int_0^{ + \infty } {\frac{{t^{N - 1} e^{ - 2\pi t} }}{{1 - it/z}}M_K \left( {it} \right)dt}  + \frac{{\left( { - i} \right)^{N - 1} }}{{2\pi z^N }}\int_0^{ + \infty } {\frac{{t^{N - 1} e^{ - 2\pi t} }}{{1 + it/z}}M_K \left( { - it} \right)dt} 
\\ & - \frac{{1 + i}}{{\sqrt 2 }}\frac{{i^{M - 1} }}{{2\pi z^{M + \frac{1}{2}} }}\int_0^{ + \infty } {\frac{{t^{M - \frac{1}{2}} e^{ - 2\pi t} }}{{1 - it/z}}M_L \left( {it} \right)dt}  - \frac{{1 - i}}{{\sqrt 2 }}\frac{{\left( { - i} \right)^{M - 1} }}{{2\pi z^{M + \frac{1}{2}} }}\int_0^{ + \infty } {\frac{{t^{M - \frac{1}{2}} e^{ - 2\pi t} }}{{1 + it/z}}M_L \left( { - it} \right)dt} 
\end{split}
\end{gather}
for $\frac{3\pi}{2}<\arg z< \frac{5\pi}{2}$. It is easy to see that the sum of the four integrals has the order of magnitude given in the right-hand side of \eqref{eq78}. It follows that when $K=L$, the bound \eqref{eq80} holds in the sector $\frac{3\pi}{2}\leq \arg z\leq \frac{5\pi}{2}$. Otherwise, by the connection formula \eqref{eq51}, we have
\begin{equation}\label{eq89}
e^{2\pi iz} M_K \left( z \right) - e^{2\pi iz} M_L \left( z \right) = e^{2\pi iz} M_K \left( {ze^{ - 2\pi i} } \right) - e^{2\pi iz} M_L \left( {ze^{ - 2\pi i} } \right).
\end{equation}
If $\frac{3\pi}{2}\leq \arg z \leq \frac{5\pi}{2}$, then $M_K \left( ze^{ -2 \pi i} \right) = \mathcal{O}_K\left( {\left|z\right|^{ - K} } \right)$ as $z\to \infty$, whence by continuity
\begin{equation}\label{eq82}
R_{N,M,K,L} \left( z \right) = \mathcal{O}_{K,\rho } \left( {\frac{{\left| {\sin \left( {2\pi z} \right)} \right|}}{{\left| z \right|^K }}} \right) + \mathcal{O}_{L,\sigma } \left( {\frac{{\left| {\cos \left( {2\pi z} \right)} \right|}}{{\left| z \right|^L }}} \right),
\end{equation}
as $z\to \infty$ in the closed sector $\frac{3\pi}{2}\leq \arg z \leq \frac{5\pi}{2}$.

Similarly, we find that when $K=L$, the estimate \eqref{eq81} holds in the sector $-\frac{5\pi}{2}\leq \arg z \leq -\frac{3\pi}{2}$.  Otherwise, it can be shown that the estimation \eqref{eq82} is valid in this sector too.

Finally, we consider the sector $\frac{5\pi}{2}<\arg z< 3\pi$. Rotating the path of integration in the first and the third integrals in \eqref{eq79} and applying the residue theorem gives
\begin{gather}\label{eq88}
\begin{split}
& R_{N,M,K,L} \left( z \right) = e^{2\pi iz} M_K \left( {ze^{ - 2\pi i} } \right) + e^{2\pi iz} M_L \left( {ze^{ - 2\pi i} } \right)
\\ & - e^{ - 2\pi iz} M_K \left( {ze^{ - 2\pi i} } \right) - e^{ - 2\pi iz} M_L \left( {ze^{ - 2\pi i} } \right) + e^{2\pi iz} M_K \left( z \right) - e^{2\pi iz} M_L\left( z \right)
\\ & + \frac{{i^{N - 1} }}{{2\pi z^N }}\int_0^{ + \infty } {\frac{{t^{N - 1} e^{ - 2\pi t} }}{{1 - it/z}}M_K \left( {it} \right)dt}  + \frac{{\left( { - i} \right)^{N - 1} }}{{2\pi z^N }}\int_0^{ + \infty } {\frac{{t^{N - 1} e^{ - 2\pi t} }}{{1 + it/z}}M_K \left( { - it} \right)dt} 
\\ & - \frac{{1 + i}}{{\sqrt 2 }}\frac{{i^{M - 1} }}{{2\pi z^{M + \frac{1}{2}} }}\int_0^{ + \infty } {\frac{{t^{M - \frac{1}{2}} e^{ - 2\pi t} }}{{1 - it/z}}M_L \left( {it} \right)dt}  - \frac{{1 - i}}{{\sqrt 2 }}\frac{{\left( { - i} \right)^{M - 1} }}{{2\pi z^{M + \frac{1}{2}} }}\int_0^{ + \infty } {\frac{{t^{M - \frac{1}{2}} e^{ - 2\pi t} }}{{1 + it/z}}M_L \left( { - it} \right)dt}
\end{split}
\end{gather}
for $\frac{5\pi}{2}<\arg z< 3\pi$.  Again, the sum of the four integrals has the order of magnitude given in the right-hand side of \eqref{eq78}. If $K=L$, the first two lines in \eqref{eq88} simplify to
\[
2\left( {e^{2\pi iz}  - e^{ - 2\pi iz} } \right)M_K \left( {ze^{ - 2\pi i} } \right).
\]
Let $\delta$ be a fixed small positive real number. If $\frac{5\pi}{2}\leq \arg z \leq 3\pi-\delta$, then $M_K \left( {ze^{ - 2\pi i} } \right) = \mathcal{O}_{K,\delta } \left( {\left|z\right|^{ - K} } \right)$ as $z\to \infty$, whence
\[
R_{N,M,K,L} \left( z \right) = \mathcal{O}_{K,\rho,\delta } \left( {\frac{{e^{ - 2\pi \Im \left( z \right)}  + e^{2\pi \Im \left( z \right)} }}{{\left| z \right|^K }}} \right) = \mathcal{O}_{K,\rho,\delta } \left( {\frac{{e^{2\pi \Im \left( z \right)} }}{{\left| z \right|^K }}} \right)
\]
as $z\to \infty$ in the sector $\frac{5\pi}{2}\leq \arg z \leq 3\pi-\delta$. Otherwise, using \eqref{eq89}, the first two lines in \eqref{eq88} simplify to
\[
\left( {2e^{2\pi iz}  - e^{ - 2\pi iz} } \right)M_K \left( {ze^{ - 2\pi i} } \right) - e^{ - 2\pi iz} M_L \left( {ze^{ - 2\pi i} } \right),
\]
whence
\begin{align*}
R_{N,M,K,L} \left( z \right) & = \mathcal{O}_{K,\rho,\delta } \left( {\frac{{e^{ - 2\pi \Im \left( z \right)}  + e^{2\pi \Im \left( z \right)} }}{{\left| z \right|^K }}} \right) + \mathcal{O}_{L,\sigma,\delta } \left( {\frac{{e^{2\pi \Im \left( z \right)} }}{{\left| z \right|^L }}} \right) \\ & = \mathcal{O}_{K,\rho,\delta } \left( {\frac{{\left| {\sin \left( {2\pi z} \right)} \right|}}{{\left| z \right|^K }}} \right) + \mathcal{O}_{L,\sigma,\delta } \left( {\frac{{\left| {\cos \left( {2\pi z} \right)} \right|}}{{\left| z \right|^L }}} \right),
\end{align*}
as $z\to \infty$ in the sector $\frac{5\pi}{2}\leq \arg z \leq 3\pi-\delta$.

The proof for the sector $-3\pi+\delta \leq \arg z \leq -\frac{5\pi}{2}$ is completely analogous.

To extend the error estimates in Theorem \ref{thm4} to the cases $K=0,1$ and $L=0,1$, we can proceed similarly as in the proof of Theorem \ref{thm3}.

\subsection{Stokes phenomenon and Berry's transition}\label{subsection53} First, we study the Stokes phenomenon related to the asymptotic expansion \eqref{eq43} of $\Gamma \left(a,z\right)$ occurring when $\arg a$ passes through the values $\pm \pi$. In the range $\left|\arg a\right|<\pi$, the asymptotic expansion
\begin{equation}\label{eq90}
\Gamma \left( {a,z} \right) \sim z^a e^{ - z} \sum\limits_{n = 0}^\infty  {\frac{{\left( { - a} \right)^n b_n \left( \lambda  \right)}}{{\left( {z - a} \right)^{2n + 1} }}}
\end{equation}
holds as $a\to \infty$. From \eqref{eq73} we have
\[
R_N \left( {a,\lambda } \right) =R_{N,0} \left( {a,\lambda } \right)= e^{a\left( {\lambda  - \log \lambda  - 1} \right)} \sqrt {\frac{{2\pi }}{a}} \frac{1}{{\Gamma^\ast \left( {ae^{-\pi i} } \right)}} + R_N \left( {ae^{-2\pi i} ,\lambda } \right)
\]
when $\pi <\arg a <2\pi$. Similarly,
\[
R_N \left( {a,\lambda } \right) = e^{a\left( {\lambda  - \log \lambda  - 1} \right)} \sqrt {\frac{{2\pi }}{a}} \frac{1}{{\Gamma^\ast \left( {ae^{\pi i} } \right)}} + R_N \left( {ae^{2\pi i} ,\lambda } \right)
\]
for $-2\pi <\arg a <-\pi$. For the right-hand sides, we can apply the asymptotic expansions of the reciprocal of the scaled gamma function and the incomplete gamma function to deduce that
\[
\Gamma \left( {a,z} \right) \sim z^a e^{ - z} \left( {\sum\limits_{n = 0}^\infty  {\frac{{\left( { - a} \right)^n b_n \left( \lambda  \right)}}{{\left( {z - a} \right)^{2n + 1} }}}  + e^{a\left( {\lambda  - \log \lambda  - 1} \right)} \sqrt {\frac{{2\pi }}{a}} \sum\limits_{k = 0}^\infty  {\frac{{a_{2k} }}{{a^k }}} } \right)
\]
as $a\to \infty$ in the sectors $\pi < \left|\arg a\right| <2\pi$.  Therefore, as the lines $\arg a =\pm \pi$ are crossed, the additional series
\begin{equation}\label{eq91}
{e^{a\left( {\lambda  - \log \lambda  - 1} \right)} \sqrt {\frac{{2\pi }}{a}} \sum\limits_{k = 0}^\infty  {\frac{{a_{2k} }}{{a^k }}} }
\end{equation}
appears in the asymptotic expansion of $\Gamma \left( {a,z} \right)$ beside the original one \eqref{eq90}. We have encountered a Stokes phenomenon with Stokes lines $\arg a = \pm\pi$.

In the important papers \cite{Berry3, Berry2}, Berry provided a new interpretation of the Stokes phenomenon; he found that assuming optimal truncation, the transition between compound asymptotic expansions is of error function type, thus yielding a smooth, although very rapid, transition as a Stokes line is crossed.

Using the exponentially improved expansion given in Theorem \ref{thm3}, we show that the asymptotic expansion \eqref{eq90} of $\Gamma \left( {a,z} \right)$ exhibits the Berry transition between the two asymptotic series across the Stokes lines $\arg a = \pm\pi$. More precisely, we shall find that the first few terms of the series in \eqref{eq91} ``emerge" in a rapid and smooth way as $\arg a$ passes through $\pm \pi$.

From Theorem \ref{thm3}, we conclude that if $N \approx \left| a \right|\left( {\lambda  - \log \lambda  - 1} \right)$, then for large $a$, $\frac{\pi}{2}< \left|\arg a\right| <\frac{3\pi}{2}$, we
have
\[
\Gamma \left( {a,z} \right) \approx z^a e^{ - z} \left( {\sum\limits_{n = 0}^{N - 1} {\frac{{\left( { - a} \right)^n b_n \left( \lambda  \right)}}{{\left( {z - a} \right)^{2n + 1} }}}  + e^{a\left( {\lambda  - \log \lambda  - 1} \right)} \sqrt {\frac{{2\pi }}{a}} \sum\limits_{k = 0} {\frac{{a_{2k} }}{{a^k }}\widehat T_{N - k + \frac{1}{2}} \left( {a\left( {\lambda  - \log \lambda  - 1} \right)} \right)} } \right),
\]
where $\sum_{k=0}$ means that the sum is restricted to the first few terms of the series. Under the above assumption on $N$, from \eqref{eq68}--\eqref{eq70}, the terminant functions have the asymptotic behaviour
\[
\widehat T_{N - k + \frac{1}{2}} \left( {a\left( {\lambda  - \log \lambda  - 1} \right)} \right) \sim \frac{1}{2} \pm \frac{1}{2}\mathop{\text{erf}} \left( {\left( {\theta  \mp \pi } \right)\sqrt {\frac{1}{2}\left| a \right|\left( {\lambda  - \log \lambda  - 1} \right)} } \right)
\]
provided that $\arg a = \theta$ is close to $\pm \pi$, $a$ is large and $k$ is small in comparison with $N$. Therefore, when $\pm \theta  < \pi$, the terminant functions are exponentially small; for $\theta  = \pm \pi$, they are asymptotically $\frac{1}{2}$ up to an exponentially small error; and when $\pm\theta  >  \pi$, the terminant functions are asymptotic to $1$ with an exponentially small error. Thus, the transition across the Stokes lines $\arg a = \pm\pi$ is effected rapidly and smoothly.

Let us now turn our attention to the asymptotic series of $\Gamma\left(z,z\right)$.  In the range $\left|\arg z\right|< \frac{3\pi}{2}$, the asymptotic expansion
\begin{equation}\label{eq92}
\Gamma \left( {z,z} \right) \sim \sqrt {\frac{\pi }{2}} z^{z - \frac{1}{2}} e^{ - z} \sum\limits_{n = 0}^\infty {\frac{{a_n }}{{z^{\frac{n}{2}} }}}
\end{equation}
holds as $z\to \infty$. Employing the continuation formulas \eqref{eq66} and \eqref{eq51}, we find that
\[
\Gamma \left( {z,z} \right) = e^{2\pi iz} \Gamma \left( {ze^{ - 2\pi i} ,ze^{ - 2\pi i} } \right) + \left( {1 - e^{ - 2\pi iz} } \right)\sqrt {2\pi } z^{z - \frac{1}{2}} e^{ - z} \Gamma^\ast \left( {ze^{ - 2\pi i} } \right)
\]
and
\[
\Gamma \left( {z,z} \right) = e^{ - 2\pi iz} \Gamma \left( {ze^{2\pi i} ,ze^{2\pi i} } \right) + \left( {1 - e^{2\pi iz} } \right)\sqrt {2\pi } z^{z - \frac{1}{2}} e^{ - z} \Gamma^\ast \left( {ze^{2\pi i} } \right) .
\]
For the right-hand sides, we can apply the asymptotic expansions of the scaled gamma function and the incomplete gamma function function to obtain that
\begin{equation}\label{eq96}
\Gamma \left( {z,z} \right) \sim \sqrt {\frac{\pi }{2}} z^{z - \frac{1}{2}} e^{ - z} \left( {\sum\limits_{n = 0}^\infty  {\frac{{a_n }}{{z^{\frac{n}{2}} }}}  - 2e^{ - 2\pi iz} \sum\limits_{k = 0}^\infty  {\frac{{a_{2k} }}{{z^k }}} } \right)
\end{equation}
as $z\to \infty$ in the sector $\frac{3\pi}{2}<\arg z<\frac{5\pi}{2}$, and
\begin{equation}\label{eq97}
\Gamma \left( {z,z} \right) \sim \sqrt {\frac{\pi }{2}} z^{z - \frac{1}{2}} e^{ - z} \left( {\sum\limits_{n = 0}^\infty  {\frac{{a_n }}{{z^{\frac{n}{2}} }}}  - 2e^{2\pi iz} \sum\limits_{k = 0}^\infty  {\frac{{a_{2k} }}{{z^k }}} } \right)
\end{equation}
as $z\to \infty$ in the sector $-\frac{5\pi}{2}<\arg z<-\frac{3\pi}{2}$. Therefore, as the line $\arg z = \frac{3\pi}{2}$ is crossed, the additional series
\begin{equation}\label{eq93}
 - 2e^{ - 2\pi iz} \sum\limits_{k = 0}^\infty  {\frac{{a_{2k} }}{{z^k }}} 
\end{equation}
appears in the asymptotic expansion of $\Gamma \left( {z,z} \right)$ beside the original series \eqref{eq92}. Similarly, as we pass through the line $\arg z = -\frac{3\pi}{2}$, the series
\begin{equation}\label{eq94}
 - 2e^{ 2\pi iz} \sum\limits_{k = 0}^\infty  {\frac{{a_{2k} }}{{z^k }}} 
\end{equation}
appears in the asymptotic expansion of $\Gamma \left( {z,z} \right)$ beside the original one \eqref{eq92}. We have encountered a Stokes phenomenon with Stokes lines $\arg z =\pm \frac{3\pi}{2}$. With the aid of the exponentially improved expansion given in Theorem \ref{thm4}, we shall find that the asymptotic series of $\Gamma \left( {z,z} \right)$ shows the Berry transition property: the two series in \eqref{eq93} and \eqref{eq94} emerge in a rapid and smooth way as the Stokes lines $\arg z = \frac{3\pi}{2}$ and $\arg z = -\frac{3\pi}{2}$ are crossed.

Let us assume that in \eqref{eq95} $N,M \approx 2\pi \left|z\right|$ and $K=L$. When $\pi  < \arg z  < 2\pi$, the terms in \eqref{eq95} involving the terminant functions of the argument $2\pi i z$ are exponentially small, and the main contribution comes from the terms involving the terminant functions of the argument $-2\pi i z$. Therefore, from Theorem \ref{thm4}, we deduce that for large $z$, $\pi  < \arg z  < 2\pi$, we have
\[
\Gamma \left( {z,z} \right) \approx \sqrt {\frac{\pi }{2}} z^{z - \frac{1}{2}} e^{ - z} \left( {\sum\limits_{n = 0}^{N - 1} {\frac{{a_{2n} }}{{z^n }}}  + \sum\limits_{n = 0}^{M - 1} {\frac{{a_{2m + 1} }}{{z^{m + \frac{1}{2}} }}}  - 2e^{ - 2\pi iz} \sum\limits_{k = 0} {\frac{{a_{2k} }}{{z^k }}\frac{{\widehat T_{N - k} \left( { - 2\pi iz} \right) + \widehat T_{M - k + \frac{1}{2}} \left( { - 2\pi iz} \right)}}{2}} } \right),
\]
where, as before, $\sum_{k=0}$ means that the sum is restricted to the first few terms of the series. Since $N,M \approx 2\pi \left|z\right|$, from \eqref{eq68} and \eqref{eq70}, the averages of the terminant functions have the asymptotic behaviour
\[
\frac{{\widehat T_{N - k} \left( { - 2\pi iz} \right) + \widehat T_{M - k + \frac{1}{2}} \left( { - 2\pi iz} \right)}}{2} \sim \frac{1}{2} + \frac{1}{2}\mathop{\text{erf}}\left( {\left( {\theta  - \frac{{3\pi }}{2}} \right)\sqrt {\pi \left| z \right|} } \right),
\]
under the conditions that $\arg z = \theta$ is close to $\frac{3\pi}{2}$, $z$ is large and $k$ is small compared to $N$ and $M$. Thus, when $\theta < \frac{3\pi}{2}$, the averages of the terminant functions are exponentially small; for $\theta  =  \frac{3\pi}{2}$, they are asymptotic to $\frac{1}{2}$ with an exponentially small error; and when $\theta >  \frac{3\pi}{2}$, the averages of the terminant functions are asymptotically $1$ up to an exponentially small error. Thus, the transition through the Stokes line $\arg z = \frac{3\pi}{2}$ is carried out rapidly and smoothly.

Similarly, if $N,M \approx 2\pi \left|z\right|$ and $K=L$, then for large $z$, $-2\pi  < \arg z  < -\pi$, we have
\[
\Gamma \left( {z,z} \right) \approx \sqrt {\frac{\pi }{2}} z^{z - \frac{1}{2}} e^{ - z} \left( {\sum\limits_{n = 0}^{N - 1} {\frac{{a_{2n} }}{{z^n }}}  + \sum\limits_{n = 0}^{M - 1} {\frac{{a_{2m + 1} }}{{z^{m + \frac{1}{2}} }}}  - 2e^{2\pi iz} \sum\limits_{k = 0} {\frac{{a_{2k} }}{{z^k }}\frac{{ - \widehat T_{N - k} \left( {2\pi iz} \right) + \widehat T_{M - k + \frac{1}{2}} \left( {2\pi iz} \right)}}{2}} } \right).
\]
From \eqref{eq69} and \eqref{eq70}, the averages of the terminant functions have the asymptotic behaviour
\[
\frac{{ - \widehat T_{N - k} \left( {2\pi iz} \right) + \widehat T_{M - k + \frac{1}{2}} \left( {2\pi iz} \right)}}{2} \sim \frac{1}{2} - \frac{1}{2}\mathop{\text{erf}}\left( {\left( {\theta  + \frac{{3\pi }}{2}} \right)\sqrt {\pi \left| z \right|} } \right),
\]
provided that $N,M \approx 2\pi \left|z\right|$, $\arg z = \theta$ is close to $-\frac{3\pi}{2}$, $z$ is large and $k$ is small compared to $N$ and $M$. Therefore, when $\theta > - \frac{3\pi}{2}$, the averages of the terminant functions are exponentially small; for $\theta  =  -\frac{3\pi}{2}$, they are asymptotic to $\frac{1}{2}$ up to an exponentially small error; and when $\theta <  -\frac{3\pi}{2}$, the averages of the terminant functions are asymptotically $1$ with an exponentially small error. Thus, the transition through the Stokes line $\arg z = -\frac{3\pi}{2}$ is effected rapidly and smoothly.

We note that from the expansions \eqref{eq96} and \eqref{eq97}, it follows that \eqref{eq92} is an asymptotic series of $\Gamma \left( z,z \right)$ in the wider range $\left|\arg z\right| \leq 2\pi -\delta < 2\pi$, with any fixed $0 < \delta \leq 2\pi$.

\appendix

\section{Computation of the coefficients $b_n\left(\lambda\right)$ and $a_n$}\label{appendixa}

In this appendix we collect some formulas for the computation of the coefficients that appear in the asymptotic expansions of the incomplete gamma function.

\subsection{The coefficients $b_n\left(\lambda\right)$} The simplest way to generate the polynomials $b_n\left(\lambda\right)$ is to use the recurrence
\[
b_n \left( \lambda  \right) = \lambda \left( {1 - \lambda } \right)b'_{n - 1} \left( \lambda  \right) + \left( {2n - 1} \right)\lambda b_{n - 1} \left( \lambda  \right),
\]
with $b_0 \left( \lambda  \right)=1$ \cite[8.11.E9]{NIST}.

There has been a recent interest in finding explicit formulas for the coefficients in asymptotic expansions of Laplace-type integrals (see \cite{Lopez}, \cite{Nemes2}, \cite{Wojdylo1} and \cite{Wojdylo2}). There are two general formulas for these coefficients, one containing potential polynomials and one containing Bell polynomials. We derive them here for the special case of the coefficients $b_n\left(\lambda\right)$. Let
\[
\lambda e^t  - t - \lambda  = \sum\limits_{j = 0}^\infty  {a_j t^{j + 1} } ,
\]
so that
\[
a_0  = \lambda  - 1,\; a_j  = \frac{\lambda }{{\left( {j + 1} \right)!}} \; \text{ for } \; j \ge 1.
\]
Let $0 \leq i \leq j$ be integers and $\rho$ be a complex number. We define the potential polynomials
\[
\mathsf{A}_{\rho ,j}  = \mathsf{A}_{\rho ,j} \left( {\frac{{a_1 }}{{a_0 }},\frac{{a_2 }}{{a_0 }}, \ldots ,\frac{{a_k }}{{a_0 }}} \right)
\]
and the Bell polynomials
\[
\mathsf{B}_{j,i}  = \mathsf{B}_{j,i} \left( {a_1 ,a_2 , \ldots ,a_{j - i + 1} } \right)
\]
via the expansions
\begin{equation}\label{eq99}
\left( {1 + \sum\limits_{j = 1}^\infty  {\frac{{a_j }}{{a_0 }}t^j } } \right)^\rho  = \sum\limits_{j = 0}^\infty  {\mathsf{A}_{\rho ,j} t^j } \; \text{ and } \; \mathsf{A}_{\rho ,j}  = \sum\limits_{i = 0}^j {\binom{\rho}{i}\frac{1}{a_0^i}\mathsf{B}_{j,i}} .
\end{equation}
Naturally, these polynomials can be defined for arbitrary power series with $a_0\neq 0$. It is possible to express the potential polynomials with complex parameter in terms of potential polynomials with integer parameter using the following formula of Comtet \cite[p. 142]{Comtet}
\begin{equation}\label{eq100}
\mathsf{A}_{\rho ,j}  = \frac{\Gamma \left( { - \rho  + j + 1} \right)}{j!\Gamma \left( { - \rho } \right)}\sum\limits_{i = 0}^j {\frac{\left( { - 1} \right)^i}{- \rho  + i}\binom{j}{i}\mathsf{A}_{i,j} } .
\end{equation}
With these notations we can write the first representation in \eqref{eq10} as
\[
b_n \left( \lambda  \right) = \left( { - 1} \right)^n \left( {\lambda  - 1} \right)^n n!\mathsf{A}_{ - n - 1,n} 
\]
Using \eqref{eq99} and \eqref{eq100} we find
\[
b_n \left( \lambda  \right) = \sum\limits_{k = 0}^n {\left( { - 1} \right)^{n + k} \left( {\lambda  - 1} \right)^{n - k} } \frac{{\left( {n + k} \right)!}}{{k!}}\mathsf{B}_{n,k} 
\]
and
\begin{equation}\label{eq101}
b_n \left( \lambda  \right) = \frac{{\left( {2n + 1} \right)!}}{{n!}}\left( {\lambda  - 1} \right)^n \sum\limits_{k = 0}^n {\frac{{\left( { - 1} \right)^{n + k} }}{{n + k + 1}}\binom{n}{k}\mathsf{A}_{k,n} } .
\end{equation}
The quantities $\mathsf{B}_{n,k}$ and $\mathsf{A}_{k,n}$ appearing in these formulas may be computed from the recurrence relations
\[
\mathsf{B}_{n,k}  = \sum\limits_{j= 1}^{n - k + 1} {a_j \mathsf{B}_{n- j,k - 1} } \; \text{ and } \; \mathsf{A}_{k,n}  = \sum\limits_{j = 0}^{n} {\frac{a_j}{a_0}\mathsf{A}_{k - 1,n - j} }
\]
with $\mathsf{B}_{0,0} = \mathsf{A}_{0,0} = 1$, $\mathsf{B}_{j,0} = \mathsf{A}_{0,j} = 0$ $\left(j \ge 1\right)$, $\mathsf{B}_{j,1} = a_0 \mathsf{A}_{1,j} = a_j$ (see Nemes \cite{Nemes2}). Finally, we show that the potential polynomials $\mathsf{A}_{k,n}$ in \eqref{eq101} can be written in terms of the Stirling numbers of the second kind (see, e.g., Comtet \cite[pp. 204--212]{Comtet} or \cite[26.8.i]{NIST}). Indeed, a straightforward computation gives
\begin{align*}
& \mathsf{A}_{k,n}  = \frac{1}{{2\pi i}}\oint_{\left( {0^ +  } \right)} {\left( {\frac{{\lambda e^t  - t - \lambda }}{{\left( {\lambda  - 1} \right)t}}} \right)^k \frac{{dt}}{{t^{n + 1} }}} = \frac{1}{{\left( {\lambda  - 1} \right)^k }}\frac{1}{{2\pi i}}\oint_{\left( {0^ +  } \right)} {\left( {\lambda \left( {e^t  - 1} \right) - t} \right)^k \frac{{dt}}{{t^{n + k + 1} }}} 
\\ & = \frac{1}{{\left( {\lambda  - 1} \right)^k }}\frac{1}{{2\pi i}}\oint_{\left( {0^ +  } \right)} {\sum\limits_{j = 0}^k {\left( { - 1} \right)^{k - j} \binom{k}{j}\lambda ^j \left( {\frac{{e^t  - 1}}{t}} \right)^j } \frac{{dt}}{{t^{n + 1} }}} 
\\ & = \frac{1}{{\left( {\lambda  - 1} \right)^k }}\frac{1}{{2\pi i}}\oint_{\left( {0^ +  } \right)} {\sum\limits_{j = 0}^k {\left( { - 1} \right)^{k - j} \binom{k}{j}\lambda ^j \sum\limits_{i = 0}^\infty  {j!S\left( {i + j,j} \right)\frac{{t^i }}{{\left( {i + j} \right)!}}} } \frac{{dt}}{{t^{n + 1} }}} 
\\ & = \frac{1}{{\left( {\lambda  - 1} \right)^k }}\sum\limits_{j = 0}^k {\left( { - 1} \right)^{k - j} \binom{k}{j}\lambda ^j \frac{{j!}}{{\left( {n + j} \right)!}}S\left( {n + j,j} \right)} ,
\end{align*}
leading to the explicit formula
\[
b_n \left( \lambda  \right) = \left( { - 1} \right)^n \left( {2n + 1} \right)!\sum\limits_{k = 0}^n {\frac{{\left( {\lambda  - 1} \right)^{n - k} }}{{\left( {n + k + 1} \right)\left( {n - k} \right)!}}\sum\limits_{j = 0}^k {\left( { - 1} \right)^j \lambda ^j \frac{{S\left( {n + j,j} \right)}}{{\left( {k - j} \right)!\left( {n + j} \right)!}}} } .
\]
\subsection{The coefficients $a_n$} Based on the first representation in \eqref{eq18}, Lagrange's inversion formula and a result of Brassesco and M\'{e}ndez \cite{Brassesco}, we find that
\[
a_n  = \frac{{2^{\frac{n}{2} + 1} }}{{\sqrt \pi  }}\Gamma \left( {\frac{{n + 3}}{2}} \right)\frac{{b_{n + 1} }}{{\left( {n + 1} \right)!}},
\]
where the sequence $b_n$ is given by the recurrence
\[
b_n  = \frac{{2 - n}}{{3n + 3}}b_{n - 1}  - \frac{1}{2}\sum\limits_{k = 2}^{n - 3} {b_{k + 1} b_{n - k} } 
\]
with $b_1=1$, $b_2  =  - \frac{1}{6}$. Note that Brassesco and M\'{e}ndez use the slightly different notation $\widetilde{b}_n$. Like for the $b_n \left( \lambda  \right)$'s, it is possible to derive explicit formulas for the coefficients $a_n$ too. The main calculations were done in the paper \cite{Nemes2}, we just write down the final result:
\[
a_n  = \sum\limits_{k = 0}^n {\frac{{2^{\frac{n}{2} + k + 1} \Gamma \left( {\frac{{3n}}{2} + \frac{3}{2}} \right)}}{{\sqrt \pi  \left( {n + 2k + 1} \right)\left( {n - k} \right)!}}\sum\limits_{j = 0}^k {\frac{{\left( { - 1} \right)^j S\left( {n + k + j,j} \right)}}{{\left( {k - j} \right)!\left( {n + k + j} \right)!}}} }.
\]

\end{document}